\documentclass[10pt,reqno]{amsart}
\RequirePackage[OT1]{fontenc}
\RequirePackage{amsthm,amsmath}
\RequirePackage[numbers]{natbib}
\RequirePackage[colorlinks,linkcolor=blue,citecolor=blue,urlcolor=blue]{hyperref}
\usepackage{amssymb}
\usepackage{anysize}
\usepackage{hyperref}
\usepackage{extarrows}
\usepackage{multirow}
\usepackage{natbib}
\usepackage{stmaryrd}
\usepackage{amsmath}
\usepackage{enumitem}

\numberwithin{equation}{section}

\theoremstyle{plain} 
\newtheorem{theorem}{Theorem}[section]
\newtheorem{lemma}[theorem]{Lemma}
\newtheorem{corollary}[theorem]{Corollary}
\newtheorem{proposition}[theorem]{Proposition}

\newtheorem{definition}[theorem]{Definition}

\newcounter{kevin}
\numberwithin{kevin}{section}
\theoremstyle{remark}
\newtheorem{remark}[kevin]{Remark}

\renewcommand{\Re}{\mathrm{Re}\,}
\renewcommand{\Im}{\mathrm{Im}\,}

\newcommand{\im}{\mathrm{Im}\,}

\newcommand{\E}{{\mathbb E }}
\newcommand{\R}{{\mathbb R }}
\newcommand{\N}{{\mathbb N}}
\newcommand{\Z}{{\mathbb Z}}
\renewcommand{\P}{{\mathbb P}}
\newcommand{\C}{{\mathbb C}}

\newcommand{\ii}{\mathrm{i}}
\newcommand{\deq}{\mathrel{\mathop:}=}
\newcommand{\e}[1]{\mathrm{e}^{#1}}
\newcommand{\ntr}{\mathrm{tr}\,}
\newcommand{\dd}{\mathrm{d}}
\newcommand{\ie}{\emph{i.e., }}
\newcommand{\eg}{\emph{e.g., }}
\newcommand{\cf}{\emph{c.f., }}
\newcommand{\PP}{\Phi}
\newcommand{\dL}{\mathrm{d}_{\mathrm{L}}}
\newcommand{\wt}{\widetilde}

\newcommand{\bs}{\boldsymbol}
\newcommand{\la}{\langle}
\newcommand{\ra}{\rangle}
\newcommand{\IE}{ \mathbb{I}\mkern-1mu\mathbb{E}}

\DeclareMathOperator*{\supp}{supp\,}

\DeclareMathOperator*{\OSD}{O_{\mkern-2mu\scriptstyle{\prec}}\mkern-2mu}
\DeclareMathOperator*{\sgn}{sgn}

\renewcommand{\mathbf}[1]{\bs{#1}}
 
\marginsize{35mm}{35mm}{38mm}{40mm}  

\begin{document}

 \begin{minipage}{0.85\textwidth}
 \vspace{2.5cm}
 \end{minipage}
\begin{center}
\large\bf
Local law of addition of random matrices on optimal scale
\end{center}

\renewcommand{\thefootnote}{\fnsymbol{footnote}}	
\vspace{1cm}
\begin{center}
 \begin{minipage}{0.3\textwidth}
\begin{center}
Zhigang Bao\footnotemark[2]  \\
\footnotesize {IST Austria}\\
{\it zhigang.bao@ist.ac.at}
\end{center}
\end{minipage}
\begin{minipage}{0.3\textwidth}
\begin{center}
L\'aszl\'o Erd{\H o}s\footnotemark[1]  \\
\footnotesize {IST Austria}\\
{\it lerdos@ist.ac.at}
\end{center}
\end{minipage}
\begin{minipage}{0.3\textwidth}
 \begin{center}
Kevin Schnelli\footnotemark[2]\\
\footnotesize 
{IST Austria}\\
{\it kevin.schnelli@ist.ac.at}
\end{center}
\end{minipage}
\footnotetext[1]{Partially supported by ERC Advanced Grant RANMAT No.\ 338804.}
\footnotetext[2]{Supported by ERC Advanced Grant RANMAT No.\ 338804.}

\renewcommand{\thefootnote}{\fnsymbol{footnote}}	

\end{center}
\vspace{1cm}

\begin{center}
 \begin{minipage}{0.8\textwidth}\footnotesize{
The eigenvalue distribution of the sum of two large Hermitian matrices, 
when one of them is conjugated by a Haar distributed unitary matrix, is asymptotically given by the free convolution
of their spectral distributions. We prove that this convergence also holds locally in the bulk of the spectrum, down to the optimal
scales larger than the eigenvalue spacing. The corresponding eigenvectors 
are fully delocalized. Similar results hold for the sum of two real symmetric matrices, when one is conjugated
by a Haar orthogonal matrix. }
\end{minipage}
\end{center}

 \vspace{2mm}
 
 {\small
\footnotesize{\noindent\textit{Date}: \today}\\
 \footnotesize{\noindent\textit{Keywords}: Random matrices, local eigenvalue density, free convolution}
 
 \footnotesize{\noindent\textit{AMS Subject Classification (2010)}: 46L54, 60B20}
 \vspace{2mm}

 }

\thispagestyle{headings}

\section{Introduction}

The pioneering work \cite{Voi91} of Voiculescu connected free probability with random matrices,
as one of the most prominent examples for a noncommutative probability space is the space of 
Hermitian $N\times N$ matrices. 
On one hand, the law of the sum of two free random variables with laws $\mu_\alpha$ and $\mu_\beta$ 
 is given by the free additive convolution $\mu_\alpha \boxplus\mu_\beta$. On the other hand, in case
 of Hermitian matrices, the law can be identified with the distribution of the eigenvalues. Thus the
 free additive convolution computes the eigenvalue distribution of the sum of two free Hermitian matrices.
However, freeness is characterized by  an infinite collection of moment identities and cannot easily be verified
in general. A fundamental direct mechanism to generate freeness is conjugation by random unitary matrices.
More precisely, two  large Hermitian random matrices are asymptotically free if the unitary transfer matrix
between their eigenbases is Haar distributed. The most important example is when the 
spectra of the two matrices are deterministic and the unitary conjugation is the sole source of  randomness.
In other words, if $A=A^{(N)}$ and $B=B^{(N)}$ are two sequences of deterministic $N\times N$ Hermitian matrices
and~$U$ is a Haar distributed unitary, then $A$ and $UBU^*$ are asymptotically free in the large $N$ limit
and the eigenvalue distribution of $A+UBU^*$ is given by the free additive convolution $\mu_A\boxplus \mu_B$
of the eigenvalue distributions of $A$ and $B$. 

Since Voiculescu's first proof, several  alternative approaches have been developed, see \eg \cite{Bia98bis, Col03, VP, Spe93}, but all of them  were {\it global} in the sense that they
describe the eigenvalue distribution in the weak limit, \ie on the macroscopic
scale, tested against $N$-independent test functions (to fix the scaling, 
we assume that $A^{(N)}$ and $B^{(N)}$ are uniformly bounded).

The study of a {\it local law}, \ie identification of the eigenvalue distribution of $A+UBU^*$ with the free additive convolution
 below the macroscopic scale, was initiated by Kargin. First, he reached the scale  $(\log N)^{-1/2}$
 in~\cite{Kargin2012} by using the Gromov--Milman concentration inequality for the Haar measure (a weaker
 concentration result was obtained earlier by Chatterjee~\cite{Chatterjee}).
Kargin later improved his result down to scale $N^{-1/7}$ in the bulk of the spectrum~\cite{Kargin}
by analyzing the stability of the
subordination equations more efficiently. This result was valid only away from  finitely
many points in the bulk spectrum and no effective control was given on this exceptional set. Recently in~\cite{BES15}, we reduced the minimal scale to~$N^{-2/3}$  by establishing the optimal stability 
and by using a bootstrap procedure to successively localize the 
Gromov--Milman inequality from larger to smaller scales. Moreover, our result holds in the entire bulk spectrum.  In fact, the key novelty in~\cite{BES15}
was  a new stability analysis in the entire bulk spectrum.

The main result of the current paper is the local law for $H=A+UBU^*$ down to the scale $N^{-1+\gamma}$,
for any $\gamma>0$. Note that the typical eigenvalue spacing is of order $N^{-1}$, a scale where 
the eigenvalue density fluctuates and no local law holds. Thus our result holds down to the optimal scale.

There are several motivations to establish such refinements of the macroscopic limit laws. 
First, such bounds are used as {\it a priori} estimates in the proofs of  Wigner-Dyson-Mehta type universality
results on local spectral statistics; see \eg~\cite{EY12, BEYY14,ES15, LY15} and references therein.
Second, control on the diagonal resolvent matrix elements for some $\eta =\im z$ implies
that the eigenvectors are delocalized on scale $\eta^{-1}$; the optimal scale for $\eta$ yields complete
delocalization of the eigenvectors.  Third, the local law is ultimately related to an effective speed of convergence in Voiculescu's theorem on the global scale~\cite{Kargin,BES15}.
 
Our proof has three major ingredients. First, we use a {\it partial randomness decomposition} of the Haar measure
that enables us to take partial expectations of $G_{ii}$, the diagonal matrix elements of the resolvent $G(z)=(H-z)^{-1}$
at spectral parameter $z\in \C^+$.
Exploiting concentration only for the partial randomness surpasses the more general but less flexible Gromov--Milman technique. 
Second, to compute the partial expectations of $G_{ii}$,  we establish a new
system of self-consistent equations involving {\it only two} auxiliary quantities.
Keeping in mind, as a close analogy, that 
freeness  involves checking infinitely many moment  conditions for  monomials
of $A$, $B$ and $U$, one may fear that an equation for $G$ involves $BG$, whose equation involves $BGB$ {\it etc}.,
\ie one would end up with an infinite system of equations. Surprisingly this is not the case
and monitoring
two appropriately chosen quantities  in tandem is sufficient to close the system. 
Third, to connect the partial expectation of $G_{ii}$ with the subordination functions from free probability,
we rely on the optimal stability
result for the subordination equations obtained in~\cite{BES15}. 

One prominent application of our work concerns the {\it single ring theorem} of Guionnet, Krishnapur and Zeitouni \cite{GKZ11} on the 
eigenvalue distribution of matrices of the form $UTV$, where $T$ is a fixed positive definite matrix and 
$U$, $V$ are independent Haar distributed. 
 Via the hermitization technique, the current proof of the local law for the addition of random matrices  can also be used
 to prove a local version of the single ring theorem. This approach was demonstrated recently  by 
 Benaych-Georges~\cite{BG}, who proved a local single ring theorem on scale $(\log N)^{-1/4}$
 using Kargin's local law on scale $(\log N)^{-1/2}$. The local law on the optimal scale $N^{-1}$ and throughout the entire bulk spectrum
is a key ingredient to prove the local single ring theorem on the optimal scale. 
The details  are deferred to our separate work in preparation~\cite{BES15tri}.

\subsection{Notation}
 The following definition for high-probability estimates is suited for our purposes, which was first used in~\cite{EKY}.
\begin{definition}\label{definition of stochastic domination}
Let
\begin{align}
 X=(X^{(N)}(v)\,:\, N\in\N\,, v\in \mathcal{V}^{(N)})\,,\qquad\qquad Y=(Y^{(N)}(v)\,:\, N\in\N\,,\,v\in \mathcal{V}^{(N)})
\end{align}
be two families of nonnegative random variables where $\mathcal{V}^{(N)}$ is a possibly $N$-dependent parameter set. We say that $Y$ stochastically dominates $X$, uniformly in~$v$, if for all (small) $\epsilon>0$ and (large) $D>0$,
\begin{align}
\sup_{v\in \mathcal{V}^{(N)}} \P\,\bigg(X^{(N)}(v)>N^{\epsilon} Y^{(N)}(v) \bigg)\le N^{-D}\,, \label{080630}
\end{align} 
for sufficiently large $N\ge N_0(\epsilon,D)$. If $Y$ stochastically dominates $X$, uniformly in~$v$, we write $X \prec Y$. 
\end{definition}

 We further rely on the following notation. We use the symbols $O(\,\cdot\,)$ and $o(\,\cdot\,)$ for the standard big-O and little-o notation. We use~$c$ and~$C$ to denote strictly positive constants that do not depend on $N$. Their values may change from line to line. For $a,b\ge0$, we write $a\lesssim b$, $a\gtrsim b$ if there is $C\ge1$ such that $a\le Cb$, $a\geq C^{-1} b$ respectively.

 We use bold font for vectors in $\C^N$ and denote the components as $\mathbf{v}=(v_1,\ldots,v_N)\in\C^N$. The canonical basis of $\C^N$ is denoted by $(\mathbf{e}_i)_{i=1}^N$. For $\mathbf{v},\mathbf{w}\in\C^N$, we write $\mathbf{v}^* \mathbf{w}$ for the scalar product $\sum_{i=1}^N \overline{v}_i w_i $.
  We denote by $\|\mathbf v\|_2$ the Euclidean norm  and by $\|\mathbf{v}\|_{\infty}=\max_{i}|v_i|$ the uniform norm of $\mathbf{v}\in\C^N$. 

We denote by $M_N(\C)$ the set of $N\times N$ matrices over $\C$. For $A\in M_N(\C)$, we denote by $\|A\|$ its operator norm and by $\|A\|_2$ its Hilbert-Schmidt norm. The matrix entries of $A$ are denoted by $A_{ij}=\mathbf{e}_i^* A\mathbf{e}_j$. We denote by $\ntr\! A$ the normalized trace of $A$, \ie $\ntr\! A=\frac{1}{N}\sum_{i=1}^N\,A_{ii}$. 
For $\mathbf{v},\mathbf{w}\in\C^N$, the rank-one matrix $\mathbf{v}\mathbf{w}^*$ has elements $(\mathbf{v}\mathbf{w}^*)_{ij}=({v}_i\overline{w}_j)$.

 Let $\mathbf{g}=(g_1,\ldots, g_N)$ be a real or complex Gaussian vector. We write $\mathbf{g}\sim \mathcal{N}_{\mathbb{R}}(0,\sigma^2I_N)$ if $g_1,\ldots, g_N$ are independent and identically distributed (i.i.d.) $N(0,\sigma^2)$ normal variables; and we write $\mathbf{g}\sim \mathcal{N}_{\mathbb{C}}(0,\sigma^2I_N)$ if $g_1,\ldots, g_N$ are i.i.d.\ $N_{\mathbb{C}}(0,\sigma^2)$ variables, where $g_i\sim N_{\mathbb{C}}(0,\sigma^2)$ means that $\Re g_i$ and $\Im g_i$ are independent $N(0,\frac{\sigma^2}{2})$ normal variables. 

Finally, we use double brackets to denote index sets, \ie 
$$
\llbracket n_1, n_2 \rrbracket \deq [n_1, n_2] \cap \Z\,,
$$
for $n_1, n_2 \in \R$.

\section{Main results}

\subsection{Free additive convolution} \label{le subsection additive convolution} In this subsection, we recall the definition of the free additive convolution. This is a shortened version of Section~2.1~of~\cite{BES15} added for completeness. 

Given a probability measure\footnote{All probability measures considered will be assumed to be Borel.} $\mu$ on $\R$ its {\it Stieltjes transform}, $m_\mu$, on the complex upper half-plane $\C^+\deq\{ z\in\C\,:\, \im z>0\}$ is defined by
\begin{align}\label{le definition of stieltjes transform}
 m_\mu(z)\deq\int_\R\frac{\dd\mu(x)}{x-z}\,, \qquad\qquad z\in\C^+\,.
\end{align}
Note that $m_{\mu}\,:\,\C^+\rightarrow \C^+$ is an analytic function such that
\begin{align}\label{le limit to be a probablity measure}
 \lim_{\eta\nearrow\infty} \ii \eta\, m_\mu(\ii\eta)=-1\,.
\end{align}
Conversely, if $m\,:\, \C^+\rightarrow \C^+$ is an analytic function such that $\lim_{\eta\nearrow\infty} \ii \eta\, m(\ii\eta)=1$, then $m$ is the Stieltjes transform of a probability measure $\mu$, \ie $m(z)=m_\mu(z)$, for all $z\in\C^+$. 

We denote by $F_\mu$ the {\it negative reciprocal Stieltjes transform} of $\mu$, \ie
\begin{align}\label{le F definition}
 F_{\mu}(z)\deq -\frac{1}{m_{\mu}(z)}\,,\qquad \qquad z\in\C^+\,.
\end{align}
Observe that
 \begin{align}\label{le F behaviour at infinity}
\lim_{\eta\nearrow \infty}\frac{F_{\mu}(\ii\eta)}{\ii\eta}=1\,,
\end{align}
as follows from~\eqref{le limit to be a probablity measure}. Note, moreover, that $F_\mu$ is an analytic function on $\C^+$ with nonnegative imaginary part.

The {\it free additive convolution} is the symmetric binary operation on probability measures on $\R$ characterized by the following result.
\begin{proposition}[Theorem 4.1 in~\cite{BB}, Theorem~2.1 in~\cite{CG}]\label{le prop 1}
Given two probability measures, $\mu_1$ and $\mu_2$, on $\R$, there exist unique analytic functions, $\omega_1\,,\omega_2\,:\,\C^+\rightarrow \C^+$, such that,
 \begin{itemize}[noitemsep,topsep=0pt,partopsep=0pt,parsep=0pt]
  \item[$(i)$] for all $z\in \C^+$, $\im \omega_1(z),\,\im \omega_2(z)\ge \im z$, and
  \begin{align}\label{le limit of omega}
  \lim_{\eta\nearrow\infty}\frac{\omega_1(\ii\eta)}{\ii\eta}=\lim_{\eta\nearrow\infty}\frac{\omega_2(\ii\eta)}{\ii\eta}=1\,;
  \end{align}
  \item[$(ii)$] for all $z\in\C^+$, 
  \begin{align}\label{le definiting equations}
   F_{\mu_1}(\omega_{2}(z))=F_{\mu_2}(\omega_{1}(z))\,,\qquad\qquad \omega_1(z)+\omega_2(z)-z=F_{\mu_1}(\omega_{2}(z))\,.
  \end{align}

 \end{itemize}
\end{proposition}

It follows from~\eqref{le limit of omega} that the analytic function $F\,:\,\C^+\rightarrow \C^+$ defined by
\begin{align}\label{le kkv}
 F(z)\deq F_{\mu_1}(\omega_{2}(z))=F_{\mu_2}(\omega_{1}(z))\,,
\end{align}
satisfies the analogue of~\eqref{le F behaviour at infinity}. Thus $F$ is the negative reciprocal Stieltjes transform of a probability measure $\mu$, called the free additive convolution of $\mu_1$ and $\mu_2$,  usually denoted by $\mu\equiv\mu_1\boxplus\mu_2$. The functions $\omega_1$ and $\omega_2$ of Proposition~\ref{le prop 1} are called {\it subordination functions} and $m$ is said to be subordinated to~$m_{\mu_1}$, respectively to $m_{\mu_2}$. Moreover, observe that $\omega_1$ and~$\omega_2$ are analytic functions on $\C^+$ with nonnegative imaginary parts. Hence they admit the Nevanlinna representations
\begin{align}\label{le neva for omega}
 \omega_{j}(z)=a_{\omega_{j}}+z+\int_\R\frac{1+zx}{x-z}\,\dd\varrho_{\omega_{j}}(x)\,,\qquad\qquad j=1,2\,,\qquad z\in\C^+\,,
\end{align}
where $a_{\omega_j}\in\R$ and $\varrho_{\omega_j}$ are finite Borel measures on $\R$. For further details and historical remarks on the free additive convolution we refer to, \eg~\cite{VDN,HP}.

Choosing $\mu_1$ as a single point mass at $b\in \R$ and $\mu_2$ arbitrary, it is straightforward to check that~$\mu_{1}\boxplus\mu_2 $ is~$\mu_2$ shifted by~$b$. We exclude this uninteresting case by assuming hereafter that $\mu_1$ and~$\mu_2$ are both supported at more than one point. For general $\mu_1$ and $\mu_2$, the atoms of $\mu_1\boxplus\mu_2$ are identified as follows.  A point $c\in\R$ is an atom of $\mu_1\boxplus\mu_2$, if and only if there exist $a,b\in\R$ such that $c=a+b$ and $\mu_1(\{a\})+\mu_2(\{b\})>1$; see [Theorem~7.4,~\cite{BeV98}]. Properties of the continuous part of $\mu_1\boxplus\mu_2$ may be inferred from the  boundary behavior of the functions $F_{\mu_1\boxplus\mu_2}$, $\omega_1$ and $\omega_2$. For simplicity, we restrict the discussion to compactly supported probability measures.
\begin{proposition}[Theorem 2.3 in~\cite{Bel1}, Theorem~3.3 in~\cite{Bel}]\label{prop extension}
 Let $\mu_1$ and $\mu_2$ be compactly supported probability measures on $\R$ none of them being a single point mass. Then the functions $F_{\mu_1\boxplus\mu_2}$, $\omega_1$, $\omega_2\,:\, \C^+\to\C^+$ extend continuously to $\R$.
\end{proposition}
Belinschi further showed in Theorem~4.1 in~\cite{Bel} that the singular continuous part of $\mu_1\boxplus\mu_2$ is always zero and that the absolutely continuous part, $(\mu_1\boxplus\mu_2)^{\mathrm{ac}}$, of $\mu_1\boxplus\mu_2$ is always nonzero. We denote the density function of $(\mu_1\boxplus\mu_2)^{\mathrm{ac}}$ by $f_{\mu_1\boxplus\mu_2}$.

We are now all set to introduce our notion of {\it regular bulk}, $\mathcal{B}_{\mu_1\boxplus\mu_2}$, of $\mu_1\boxplus\mu_2$. Informally, we let $\mathcal{B}_{\mu_1\boxplus\mu_2}$ be the open set on which $\mu_1\boxplus\mu_2$ has a continuous density that is strictly positive and bounded from above. For a formal definition we first introduce the set
\begin{align}\label{le set U}
\mathcal{U}_{\mu_1\boxplus\mu_2}\deq\mathrm{int}\,\bigg\{\supp (\mu_1\boxplus\mu_2)^{\mathrm{ac}}\,\big\backslash\,\{ x\in \R\,:\, \lim_{\eta\searrow 0}F_{\mu_1\boxplus\mu_2}(x+\ii\eta)=0\} \bigg\}\,.
\end{align}
Note that $\mathcal{U}_{\mu_1\boxplus\mu_2}$ does not contain any atoms of $\mu_1\boxplus\mu_2$. By the Luzin-Privalov theorem the set $ \{ x\in \R\,:\, \lim_{\eta\searrow 0}F_{\mu_1\boxplus\mu_2}(x+\ii\eta)=0\}$ has Lebesgue measure zero. In fact, an even stronger statement applies for the case at hand. Belinschi~\cite{Bel2} showed that if $x\in\R$ is such that $\lim_{\eta\searrow 0}F_{\mu_1\boxplus\mu_2}(x+\ii\eta)=0$, then it must be of the form $x=a+b$ with $\mu_1(\{a\})+\mu_2(\{b\})\ge 1$, $a,b\in\R$. There could only be  finitely many such $x$, thus $\mathcal{U}_{\mu_1\boxplus\mu_2}$ must contain an open non-empty interval.

\begin{proposition}[Theorem~3.3 in~\cite{Bel}]\label{le real analytic prop}
Let $\mu_1$ and $\mu_2$ be as above and fix any $x\in\mathcal{U}_{\mu_1\boxplus\mu_2}$. Then $F_{\mu_1\boxplus\mu_2}$, $\omega_1$, $\omega_2\,:\,\C^+\rightarrow \C^+$ extend analytically around $x$. Thus the density function $f_{\mu_1\boxplus\mu_2}$ is real analytic in $\mathcal{U}_{\mu_1\boxplus\mu_2}$ wherever positive.
\end{proposition}
The regular bulk is obtained from $\mathcal{U}_{\mu_1\boxplus\mu_2}$ by removing the zeros of $f_{\mu_1\boxplus\mu_2}$ inside $\mathcal{U}_{\mu_1\boxplus\mu_2}$.
\begin{definition}
The regular bulk of the measure $\mu_1\boxplus\mu_2$ is the set
\begin{align}
 \mathcal{B}_{\mu_1\boxplus\mu_2}\deq\mathcal{U}_{\mu_1\boxplus\mu_2}\,\backslash\, \left\{x\in\mathcal{U}_{\mu_1\boxplus\mu_2}\,:\,f_{\mu_1\boxplus\mu_2}	(x)=0\right\}\,.
\end{align}
\end{definition}
Note that $\mathcal{B}_{\mu_1\boxplus\mu_2}$ is an open nonempty set on which $\mu_1\boxplus\mu_2$ admits the density $f_{\mu_1\boxplus\mu_2}$. The density is strictly positive and thus (by Proposition~\ref{le real analytic prop}) real analytic on $\mathcal{B}_{\mu_1\boxplus\mu_2}$.

\subsection{Definition of the model and assumptions}
Let $A\equiv A^{(N)}$ and $B\equiv B^{(N)}$ be two sequences of deterministic real diagonal matrices in $M_N(\C)$, whose empirical spectral distributions are denoted by $\mu_A$ and $\mu_B$, respectively. More precisely,
\begin{align}\label{le empirical measures of A and B}
 \mu_{A}\deq\frac{1}{N}\sum_{i=1}^N\delta_{a_i}\,,\qquad\qquad\mu_{B}\deq\frac{1}{N}\sum_{i=1}^N\delta_{b_i}\,,
\end{align}
with $A=\mathrm{diag}(a_i)$, $B=\mathrm{diag}(b_i)$. The matrices $A$ and $B$ actually depend on $N$, but we omit this from our notation. Proposition~\ref{le prop 1} asserts the existence of unique analytic functions~$\omega_A$ and $\omega_B$ satisfying the analogue of~\eqref{le limit of omega} such that, for all $z\in\C^+$,
\begin{align}\label{060101}
   F_{\mu_A}(\omega_{B}(z))=F_{\mu_B}(\omega_{A}(z))\,,\qquad\qquad \omega_A(z)+\omega_B(z)-z=F_{\mu_A}(\omega_{B}(z))\,.
\end{align}

We will assume that there are deterministic probability measures $\mu_\alpha$ and $\mu_\beta$ on $\R$, neither of them being a single point mass, such that the empirical spectral distributions $
 \mu_A$ and $\mu_B$ converge weakly, as $N\to\infty$, to $\mu_\alpha$ and $ \mu_\beta$, respectively. More precisely, we assume that 
\begin{align}\label{le assumptions convergence empirical measures}
\dL(\mu_A,\mu_\alpha)+\dL(\mu_B,\mu_\beta)\to 0\,,
\end{align}
as $N\to \infty$, where $\dL$ denotes the L\'evy distance. Proposition~\ref{le prop 1} asserts that there are unique analytic functions $\omega_\alpha$, $\omega_\beta$  satisfying the analogue of~\eqref{le limit of omega} such that, for all $z\in\C^+$,
\begin{align}\label{060102}
F_{\mu_\alpha}(\omega_{\beta}(z))=F_{\mu_\beta}(\omega_{\alpha}(z))\,,\qquad \omega_\alpha(z)+\omega_\beta(z)-z=F_{\mu_\alpha}(\omega_{\beta}(z))\,.
\end{align}
Proposition~4.13 of~\cite{BeV93} states that $\dL(\mu_A\boxplus\mu_B,\mu_\alpha\boxplus\mu_\beta)\le \dL(\mu_A,\mu_\alpha)+\dL(\mu_B,\mu_\beta)$, \ie the free additive convolution is continuous with respect to weak convergence of measures. 

Denote by $ U(N) $ the unitary group of degree $N$. Let $U\in U(N) $ be distributed according to the Haar measure (in short a {\it Haar unitary}), and consider the random matrix
\begin{eqnarray}\label{le our H}
H\equiv H^{(N)}\deq A+UBU^*\,.
\end{eqnarray}
Our results also holds for the real case when $U$ is Haar distributed on the orthogonal group, $O(N)$, of degree $N$. Throughout the main part of the paper the discussion will focus on the unitary case while the orthogonal case is addressed in Appendix~\ref{s.a.a}. 

We introduce the {\it Green function}, $G_H$, of $H$ and its normalized trace, $m_H$, by
\begin{align}\label{le true green function}
 G_H(z)\deq \frac{1}{H-z}\,,\qquad\qquad m_H(z)\deq\ntr G_H(z)\,,\qquad\qquad z\in\C^+\,.
\end{align}
For simplicity, we frequently use the notation $G(z)$ instead of $G_H(z)$ and we write $G_{ij}(z)\equiv (G_H)_{ij}(z)$ for the $(i,j)$th matrix element of $G(z)$.

\subsection{Main results}  For $a,b\ge 0$, $b\ge a$,  and $\mathcal{I}\subset \R$, let
\begin{align}\label{le domain S}
\mathcal{S}_{\mathcal{I}}(a,b)\deq \{z=E+\ii\eta\in\C^+\,:\,E\in \mathcal{I}\,,  a\le \eta\le b\}\,,
\end{align}
In addition, for brevity, we set, for any given $\gamma>0$,
\begin{align}\label{ZG eta}
\eta_{\mathrm{m}}\equiv \eta_{\mathrm{m}}(\gamma)\deq N^{-1+\gamma}. 
\end{align}

The main results of this paper are as follows.
\begin{theorem} \label{thm.091201}
Let $\mu_\alpha$ and $\mu_\beta$ be two compactly supported probability measures on $\R$, and assume that neither is supported at a single point and that at least one of them is supported at more than two points. Assume that the sequences of matrices $A$ and $B$ in~\eqref{le our H} are such that their empirical eigenvalue distributions $\mu_A$ and $\mu_B$ satisfy~\eqref{le assumptions convergence empirical measures}. Let $\mathcal{I}\subset\mathcal{B}_{\mu_\alpha\boxplus\mu_\beta}$ be a nonempty compact interval.

Then, for any fixed $\gamma>0$,  the estimates 
\begin{align}
\max_{1\le i\le N}&\Big|G_{ii}(z)- \frac{1}{a_i-\omega_B(z)}\Big|\prec \frac{1}{\sqrt{N\eta}}\,,\label{0912100}\\
\max_{ i\not= j}&\Big |G_{ij}(z)\Big|\prec \frac{1}{\sqrt{N\eta}} \label{100710}
\end{align}
and 
\begin{align}
\qquad\quad\Big|m_H(z)- m_{\mu_A\boxplus \mu_B}(z)\Big|\prec \frac{1}{\sqrt{N\eta}}\, \label{0913200}
\end{align}
 hold uniformly on $ \mathcal{S}_{\mathcal{I}}(\eta_{\mathrm{m}},1)$ (see~\eqref{le domain S}), where  $\eta\equiv \im z$ and $\eta_\mathrm{m}$ is given in~\eqref{ZG eta}.
\end{theorem}

 \begin{remark} The assumption that neither of $\mu_\alpha$ and $\mu_\beta$ is a point mass, ensures that the free additive convolution is not a simple translate. The additional assumption that at least one of them is supported at more than two points is for the brevity of the exposition here. In Appendix~\ref{s.a.b}, we present the corresponding result for the special case when $\mu_\alpha$ and $\mu_\beta$ are both convex combinations of two point masses.
\end{remark}

\begin{remark} \label{rem.091801}  We recall from Lemma~5.1 and Theorem~2.7 of~\cite{BES15} that, under the conditions of Theorem~\ref{thm.091201},  there is a finite constant~$C$ such that
\begin{align}\label{le auxiliary statment}
 \max_{z\in \mathcal{S}_{\mathcal{I}}(0,1)}\max\Big\{\big|\omega_A(z)-\omega_{\alpha}(z)\big|, \big|\omega_B(z)-\omega_{\beta}(z)\big|,\big|m_{\mu_A\boxplus\mu_B}(z)-m_{\mu_\alpha\boxplus\mu_\beta}(z)\big|\Big\}\nonumber\\
 \le C\left(\dL(\mu_A,\mu_\alpha)+\dL(\mu_B,\mu_\beta)\right)\,,
\end{align}
\ie  the L\'evy distances of the empirical eigenvalue distributions of $A$ and $B$ from their
limiting distributions control uniformly the deviations of the corresponding
subordination functions and Stieltjes transforms. Note moreover that $\max_{z\in\mathcal{S}_{\mathcal{I}}(0,1)} |m_{\mu_\alpha\boxplus\mu_\beta}(z)|<\infty$ by compactness of $\mathcal{I}$ and analyticity of $m_{\mu_1\boxplus\mu_2}$. Thus the Stieltjes-Perron inversion formula directly implies that $(\mu_{A}\boxplus\mu_{B})^{\mathrm{ac}}$ has a density, $f_{\mu_A\boxplus\mu_B}$, inside $\mathcal{I}$ and that
\begin{align}\label{le uniform pointwise}
 \max_{x\in\mathcal{I}}|f_{\mu_A\boxplus\mu_B}(x)-f_{\mu_\alpha\boxplus\mu_\beta}(x)|\le C\left(\dL(\mu_A,\mu_\alpha)+\dL(\mu_B,\mu_\beta)\right)\,,
\end{align}
for $N$ sufficiently large. In particular, since $\mathcal{I}\subset\mathcal{B}_{\mu_\alpha\boxplus\mu_\beta}$, we have $\mathcal{I}\subset\mathcal{B}_{\mu_A\boxplus\mu_B}$, for $N$ sufficiently large, \ie we use, for large $N$, $\mu_\alpha\boxplus\mu_\beta$ to locate (an interval of) the regular bulk of $\mu_A\boxplus\mu_B$. Finally, combining~\eqref{le auxiliary statment} and~\eqref{0912100}, we get
\begin{align*}
\max_{1\le i\le N}\Big|G_{ii}(z)- \frac{1}{a_i-\omega_\beta(z)}\Big|\prec \frac{1}{\sqrt{N\eta}}+\dL(\mu_A,\mu_\alpha)+\dL(\mu_B,\mu_\beta)\,,
\end{align*}
uniformly on $ \mathcal{S}_{\mathcal{I}}(\eta_{\mathrm{m}},1)$, where $\eta=\im z$. Averaging over the index $i$, we get the corresponding statement for~$|m_H-m_{\mu_A\boxplus\mu_B}|$ with the same error bound.
\end{remark}

\begin{remark}
 
 Note that assumption~\eqref{le assumptions convergence empirical measures} does not exclude that the matrix $H$ has outliers in the large~$N$ limit. In fact, the model $H=A+UBU^*$ shows a rich phenomenology when, say,~$A$ has a finite number of large spikes; we refer to the recent works in~\cite{BBCF,Capitaine13,Kargin}. 

\end{remark}

Let $\lambda_1,\ldots, \lambda_N$  be the eigenvalues of $H$, and $\mathbf{u}_1,\ldots, \mathbf{u}_N$ be the corresponding $\ell^2$-normalized eigenvectors. The following result shows complete delocalization of the bulk eigenvectors.

\begin{theorem}[Delocalization of eigenvectors] \label{thm.091501} Under the assumptions of Theorem~\ref{thm.091201} the following holds. Let $\mathcal{I}\subset\mathcal{B}_{\mu_\alpha\boxplus\mu_\beta}$ be a compact nonempty interval. Then
\begin{align}
\max_{i\,:\,\lambda_i\in \mathcal{I}} \|\mathbf{u}_i\|_{\infty}\prec\frac{1}{\sqrt{N}}\,. \label{091511}
\end{align}
\end{theorem}

\subsection{Strategy of proof} In this subsection, we informally outline the strategy of our proofs. Throughout the paper, without loss of generality, we assume 
\begin{align}
\ntr A=\ntr B=0\,.\label{091072}
\end{align}
For brevity, we use the shorthand $m_{\boxplus}\equiv m_{\mu_A\boxplus\mu_B}$ for the Stieltjes transform of $\mu_A\boxplus\mu_B$.

We consider first the unitary setting. Let 
\begin{align}
H\deq A+UBU^*\,,\qquad \mathcal{H}\deq U^*AU+B\,, \label{091960}
\end{align}
and denote their Green functions by 
\begin{align}
G(z)=(H-z)^{-1}\,,\qquad \mathcal{G}(z)=(\mathcal{H}-z)^{-1}\,,\qquad\qquad z\in\C^+\,. \label{091961}
\end{align}
We write $z=E+\ii\eta\in\C^+$, with $E\in \R$ and $\eta>0$, for the spectral parameter. In the sequel we often omit $z\in\C^+$ from the notation if no confusion can arise. Recalling~\eqref{le true green function}, we have
\begin{align*}
m_H(z)=\ntr G(z)=\ntr \mathcal{G}(z)\,,\qquad\qquad z\in\C^+\,.
\end{align*}
For brevity, we set
\begin{align*}
\wt{A}\deq U^*AU\,,\qquad \wt{B}\deq UBU^*\,.
\end{align*}

The following functions will play a key role in our proof.
\begin{definition}[Approximate subordination functions]
\begin{align}
\omega_A^c(z)\deq z-\frac{\ntr \widetilde{A}\mathcal{G}(z)}{m_H(z)}\,,\qquad\quad \omega_B^c(z)\deq z-\frac{\ntr \wt{B}G(z)}{m_H(z)}\,,\qquad\qquad z\in\C^+\,. \label{091260}
\end{align}
\end{definition}
 Notice that the role of $A$ and $B$ are not symmetric in these notations. By cyclicity of the trace, we may write
\begin{align}
\omega_A^c(z)= z-\frac{\ntr AG(z)}{m_H(z)}\,,\qquad\qquad z\in\C^+\,. \label{091835}
\end{align} 
We remark that the approximate subordination functions defined above are slightly different from the candidate subordination functions introduced in~\cite{Kargin} which were later used in~\cite{BES15}.

The functions $\omega_A^c(z)$ and $\omega_B^c(z)$ turn out to be good approximations to the subordination functions $\omega_A(z)$ and $\omega_B(z)$ of~\eqref{060101}. A direct consequence of the definition in~\eqref{091260}~is that
\begin{align}
\frac{1}{m_H(z)}=z-\omega_A^c(z)-\omega_B^c(z)\,,\qquad\qquad z\in\C^+\,.\label{091850}
\end{align}

Having set the notation, our main task is to show that
\begin{align}
G_{ii}(z)=\big(a_i-\omega_B^c(z)\big)^{-1}+\OSD \Big(\frac{1}{\sqrt{N\eta}}\Big)\,,\qquad\qquad z\in \mathcal{S}_{\mathcal{I}}(\eta_{\mathrm{m}},1)\,,\label{091811}
\end{align}
where we focus, for simplicity, on the diagonal Green function entries only.

We first heuristically explain how~\eqref{091811} leads to our main result in~\eqref{0912100}. A key input is the local stability of the system~\eqref{060101} established in~\cite{BES15}; see Subsection~\ref{s.4} for a summary. Averaging over $i$ in~\eqref{091811}, we get 
\begin{align}
m_H(z)=m_A(\omega_B^c(z))+\OSD\Big(\frac{1}{\sqrt{N\eta}}\Big)\,. \label{091840}
\end{align}
Replacing $H$ by $\mathcal{H}$, we analogously get 
\begin{align}
m_H(z)=m_B(\omega_A^c(z))+\OSD\Big(\frac{1}{\sqrt{N\eta}}\Big)\,, \label{091841}
\end{align}
according to~\eqref{091835}. Substituting ~\eqref{091850} into~\eqref{091840} and~\eqref{091841} we obtain the system
\begin{align*}
&F_{\mu_A}(\omega_B^c(z))=\omega_A^c(z)+\omega_B^c(z)-z+\OSD\Big(\frac{1}{\sqrt{N\eta}}\Big),\nonumber\\
 &F_{\mu_B}(\omega_A^c(z))=\omega_A^c(z)+\omega_B^c(z)-z+\OSD\Big(\frac{1}{\sqrt{N\eta}}\Big),
\end{align*}
which is a perturbation of~\eqref{060101}. Using the local stability of the system~\eqref{060101}, we obtain
\begin{align}
\big|\omega_A^c(z)-\omega_A(z)\big|\prec \frac{1}{\sqrt{N\eta}}\,,\qquad \qquad \big|\omega_B^c(z)-\omega_B(z)\big|\prec \frac{1}{\sqrt{N\eta}}\,.
\end{align}
Plugging this estimate back into~\eqref{091811} we get~\eqref{0912100}. The full proof of this step is accomplished in Section~\ref{s.7}.

We now return to~\eqref{091811}. Its proof relies on the following decomposition of Haar measure on the unitary group given, \eg in~\cite{DS87, Mezzadri}. For any fixed $i\in\llbracket 1,N\rrbracket$, any Haar unitary $U$ can be written as 
\begin{align}\label{le first decomposition}  
U=-\e{\mathrm{i}\theta_i}R_i\,U^{\langle i\rangle}\,.
\end{align}
Here $R_i$ is the {\it Householder reflection} (up to a sign) sending the vector $\bs{e}_i$ to $\bs{v}_i$, where $\bs{v}_i\in\C^N$ is a random vector distributed uniformly on the complex unit $(N-1)$-sphere, and $\theta_i\in [0,2\pi)$ is the argument of the $i$th coordinate of $\mathbf{v}_i$. The unitary matrix $U^{\langle i\rangle}$ has $\bs{e}_i$ as its $i$th column and its $(i,i)$-matrix minor (obtained by removing the $i$th column and $i$th row) is Haar distributed on $U(N-1)$; see Section~\ref{s.5} for more detail.

The gist of the decomposition in~\eqref{le first decomposition} is that the Householder reflection $R_i$ and the unitary $U^{\langle i\rangle}$ are independent, for each fixed $i\in\llbracket 1,N\rrbracket$. Hence, the decomposition in~\eqref{le first decomposition} allows one to split off the partial randomness of the vector~$\mathbf{v}_i$ from~$U$.

The proof of~\eqref{091811} is divided into two parts: 
\begin{itemize}[noitemsep,topsep=0pt,partopsep=0pt,parsep=0pt]
\item[$(i)$] Concentration of $G_{ii}$ around the partial expectation $\mathbb{E}_{\mathbf{v}_i}[G_{ii}]$, \ie $$|G_{ii}-\mathbb{E}_{\mathbf{v}_i}[G_{ii}]|\prec \frac{1}{\sqrt{N\eta}}\,.$$
\item[$(ii)$] Computation of the partial expectation $\mathbb{E}_{\mathbf{v}_i}[G_{ii}]$, \ie $$|\mathbb{E}_{\mathbf{v}_i}\big[G_{ii}(z)\big]-(a_i-\omega_B^c(z)\big)^{-1}|\prec \frac{1}{\sqrt{N\eta}}\,.$$
\end{itemize}

To prove part $(i)$, we resolve dependences by expansion and use concentration estimates for the vector $\bs{v}_i$. This part is accomplished in Section~\ref{s.6}.

Part~$(ii)$ is carried out in Section~\ref{section partial concentration}. We start from the Green function identity
\begin{align}
(a_i-z)G_{ii}(z)=-(\wt{B}G(z))_{ii}+1\,. \label{0922102}
\end{align}
Taking the $\mathbb{E}_{\bs{v}_i}$ expectation of~\eqref{0922102} and recalling the definition of the approximate subordination function $\omega_B^c(z)$ in~\eqref{091260}, it suffices to show that
\begin{align*}
\mathbb{E}_{\mathbf{v}_i}\big[(\wt{B}G)_{ii}\big]=\frac{\ntr \wt{B}G}{\ntr G} G_{ii}+\OSD\Big(\frac{1}{\sqrt{N\eta}}\Big)\,,
\end{align*}
 to prove~\eqref{091811}.
Denoting $\wt{B}^{\la i\ra}\deq U^{\la i\ra}B(U^{\la i\ra})^*$ and setting, for $z\in\C^+$,
\begin{align*}
S_i^{\sharp}(z)\deq \e{\mathrm{i}\theta_i}\mathbf{v}_i^*\wt{B}^{\la i\ra}G(z)\mathbf{e}_i\,,\qquad\qquad T_i^{\sharp}(z)\deq\e{\mathrm{i}\theta_i}\mathbf{v}_i^*G(z)\mathbf{e}_i\,,
\end{align*} 
we will prove that
\begin{align}
\mathbb{E}_{\mathbf{v}_i}\big[(\wt{B}G(z))_{ii}\big]=-\mathbb{E}_{\mathbf{v}_i}\big[S_i^{\sharp}(z)\big]+\OSD\Big(\frac{1}{\sqrt{N}}\Big)\,.\label{0922101}
\end{align}
Hence, it suffices to estimate $\mathbb{E}_{\mathbf{v}_i}\big[S_i^{\sharp}]$ instead. Approximating $\e{-\mathrm{i}\theta_i}\mathbf{v}_i$ by a Gaussian vector and using integration by parts for Gaussian random variables, we get the pair of equations
\begin{align*}
\mathbb{E}_{\mathbf{v}_i}\big[S_i^{\sharp}\big]&=\ntr \big( \widetilde{B}G\big)\,\big(\mathbb{E}_{\mathbf{v}_i}\big[S_i^{\sharp}\big]-b_i \, \mathbb{E}_{\mathbf{v}_i}\big[T_i^{\sharp}\big]\big)+\ntr  \big(\widetilde{B}G\widetilde{B}\big)\, \big(G_{ii}+\mathbb{E}_{\mathbf{v}_i}\big[T_i^{\sharp}\big]\big)+\OSD\Big(\frac{1}{\sqrt{N\eta}} \Big)\,,  \nonumber\\
\mathbb{E}_{\mathbf{v}_i}\big[T_i^{\sharp}\big]&=\ntr \big(G\big) \,\big(\mathbb{E}_{\mathbf{v}_i}\big[S_i^{\sharp}\big]- b_i \; \mathbb{E}_{\mathbf{v}_i}\big[T_i^{\sharp}\big]\big)+\ntr  \big(\widetilde{B}G\big)\, \big(G_{ii}+\mathbb{E}_{\mathbf{v}_i}\big[T_i^{\sharp}\big]\big)+\OSD\Big( \frac{1}{\sqrt{N\eta}} \Big)\,,  
\end{align*}
where we dropped the $z$-argument for the sake of brevity; see~\eqref{091120} and~\eqref{091121} for precise statements with slightly modified $S_i^\#$ and $T_i^\#$. Solving $\mathbb{E}_{\mathbf{v}_i}\big[S_i^{\sharp}\big]$ from the above two equations, we arrive at
\begin{align}
\mathbb{E}_{\mathbf{v}_i}\big[S_i^{\sharp}\big] &=-\frac{\ntr  \big(\widetilde{B}G\big)}{\ntr G} G_{ii}+\bigg[\frac{\ntr\big( \widetilde{B}G\big)-\big(\ntr\big(\widetilde{B}G\big)\big)^2}{\ntr G}+\ntr \big( \widetilde{B}G\widetilde{B}\big)\bigg]\big(G_{ii}+\mathbb{E}_{\mathbf{v}_i}\big[T_i^{\sharp}\big]\big)\nonumber\\ &\qquad\qquad+\OSD\Big(\frac{1}{\sqrt{N\eta}} \Big)\,. \label{0922105}
\end{align}
Returning to~\eqref{0922101}, we also obtain, using concentration estimates for $(\wt{B}G)_{ii}$ (which follow from the concentration estimates of $G_{ii}$ established in part $(i)$ and~\eqref{0922102}), that 
\begin{align}
\Big|\frac{1}{N}\sum_{i=1}^N\mathbb{E}_{\mathbf{v}_i}\big[S_i^{\sharp}\big]+\ntr \wt{B}G\Big|\prec \frac{1}{\sqrt{N\eta}}\,.\label{0922115}
\end{align}
Thus,  averaging~\eqref{0922105} over the index~$i$ and comparing with~\eqref{0922115}, we conclude that

\begin{align*}
\Big|\frac{\ntr\big( \widetilde{B}G\big)-\big(\ntr\big(\widetilde{B}G\big)\big)^2}{\ntr G}+\ntr \big( \widetilde{B}G\widetilde{B}\big)\Big|\prec \frac{1}{\sqrt{N\eta}}\,.
\end{align*}
Plugging this last estimate back into~\eqref{0922105}, we eventually find that
\begin{align*}
\mathbb{E}_{\mathbf{v}_i}\big[S_i^{\sharp}\big] =-\frac{\ntr  \big(\widetilde{B}G\big)}{\ntr G} G_{ii}+\OSD\Big(\frac{1}{\sqrt{N\eta}} \Big)\,,
\end{align*}
which together with~\eqref{0922101} and~\eqref{0922102} gives us part $(ii)$. This completes the sketch of the proof for the unitary case. The proof of the orthogonal case is similar. The necessary modifications are given in Appendix~\ref{s.a.a}.

\section{Preliminaries}
In this section, we first collect some basic tools used later on and then summarize results of~\cite{BES15}. In particular, we discuss, under the assumptions of Theorem~\ref{thm.091201}, stability properties of the system~\eqref{060101} and state essential properties of the subordination functions~$\omega_A$ and~$\omega_B$.

\subsection{Stochastic domination and large deviation properties}\label{stochastic domination section}
Recall the definition of stochastic domination in Definition~\ref{definition of stochastic domination}. The relation $\prec$ is a partial ordering: it is transitive and it satisfies the arithmetic rules of an order relation, {\it e.g.}, if $X_1\prec Y_1$ and $X_2\prec Y_2$ then $X_1+X_2\prec Y_1+Y_2$ and $X_1 X_2\prec Y_1 Y_2$. Further assume that $\Phi(v)\ge N^{-C}$ is deterministic and that~$Y(v)$ is a nonnegative random variable satisfying $\E [Y(v)]^2\le N^{C'}$ for all~$v$. Then $Y(v) \prec \Phi(v)$, uniformly in $v$, implies $\E [Y(v)] \prec \Phi(v)$, uniformly in~$v$.

Gaussian vectors have well-known large deviation properties. We will use them in the following form
whose proof is standard.  
\begin{lemma} \label{lem.091720} Let $X=(x_{ij})\in M_N(\C)$ be a deterministic matrix and let $\bs{y}=(y_{i})\in\C^N$ be a deterministic complex vector. For a Gaussian random vector $\mathbf{g}=(g_1,\ldots, g_N)\in \mathcal{N}_{\mathbb{R}}(0,\sigma^2 I_N)$ or $\mathcal{N}_{\mathbb{C}}(0,\sigma^2 I_N)$, we have
 \begin{align}\label{091731}
  |\bs{y}^* \bs{g}|\prec\sigma \|\bs{y}\|_2\,,\qquad\qquad  |\bs{g}^* X\bs{g}-\sigma^2N \ntr X|\prec \sigma^2\| X\|_2\,.
 \end{align}
\end{lemma}

\subsection{Rank-one perturbation formula}
At various places, we use the following fundamental perturbation formula: for $\bs{\alpha},\bs{\beta}\in\C^N$ and an invertible $D\in M_N(\C)$, we have
\begin{align}
\big(D+\bs{\alpha}\bs{\beta}^*\big)^{-1}=D^{-1}-\frac{D^{-1}\bs{\alpha}\bs{\beta}^*D^{-1}}{1+\bs{\beta}^*D^{-1}\bs{\alpha}}\,, \label{091002Kevin}
\end{align}
as can be checked readily. A standard application of~\eqref{091002Kevin} is recorded in the following lemma.
\begin{lemma}
Let $D\in M_N(\C)$ be Hermitian and let $Q\in M_N(\C)$ be arbitrary. Then, for any finite-rank Hermitian matrix $R\in M_N(\C)$, we have
\begin{align}
\left|\ntr \left(Q\big(D+R-z\big)^{-1}\right)-\ntr \left(Q(D-z)^{-1}\right)\right| &\leq \frac{\mathrm{rank}(R)\|Q\|}{N\eta}\,,\qquad z=E+\ii\eta\in\C^+\,.
 \label{091002}
\end{align}
\end{lemma}
\begin{proof}
Let $z\in\C^+$ and $\bs{\alpha}\in\C^N$. Then from~\eqref{091002Kevin} we have
\begin{align}\label{091002000}
\ntr \Big(Q\big(D\pm \bs{\alpha}\bs{\alpha}^*-z\big)^{-1}\Big)-\ntr \Big(Q(D-z)^{-1} \Big)=\pm\frac{1}{N}\frac{\bs{\alpha}^*(D-z)^{-1}Q(D-z)^{-1}\bs{\alpha}}{1\pm\bs{\alpha}^*(D-z)^{-1}\bs{\alpha}}\,.
\end{align}
We can thus estimate
\begin{align}
\Big|\ntr \Big(Q\big(D\pm\bs{\alpha}\bs{\alpha}^*-z\big)^{-1}\Big)-\ntr \Big(Q(D-z)^{-1} \Big) \Big|&\leq \frac{\|Q\|}{N}\frac{\|(D-z)^{-1}\bs{\alpha}\|_2^2}{\big|1\pm \bs{\alpha}^*(D-z)^{-1}\bs{\alpha}\big|}\nonumber\\&= \frac{\|Q\|}{N\eta}\frac{\bs{\alpha}^* \Im (D-z)^{-1}\bs{\alpha}}{\big|1\pm\bs{\alpha}^*(D-z)^{-1}\bs{\alpha}\big|}\nonumber\\ &\leq \frac{\|Q\|}{N\eta}\,. 
 \label{0910021111}
\end{align}
Since $R=R^*\in M_N(\C)$ has finite rank, we can write $R$ as a finite sum of rank-one Hermitian matrices of the form $\pm \bs{\alpha}\bs{\alpha}^*$. Thus iterating ~\eqref{0910021111} we get~\eqref{091002}.
\end{proof}

\subsection{Local stability of the system~\eqref{060101}} \label{s.4}
We first consider~\eqref{le definiting equations} in a general setting: For generic probability measures $\mu_1,\mu_2$, let $\PP_{\mu_1,\mu_2}\,:\, (\C^+)^{3}\rightarrow \C^2$ be given by
\begin{align}\label{le H system defs}
\PP_{\mu_1,\mu_2}(\omega_1,\omega_2,z)\deq\left(\begin{array}{cc} F_{\mu_1}(\omega_2)-\omega_1-\omega_2+z \\ F_{\mu_2}(\omega_1)-\omega_1-\omega_2+z \end{array}\right)\,,
\end{align}
where $F_{\mu_1}$, $F_{\mu_2}$ are the negative reciprocal Stieltjes transforms of $\mu_1$, $\mu_2$; see~\eqref{le F definition}. 
Considering $\mu_1,\mu_2$ as fixed, the equation
\begin{align}\label{le H system}
\PP_{\mu_1,\mu_2}(\omega_1,\omega_2,z)=0\,,
\end{align}
is equivalent to~\eqref{le definiting equations} and, by Proposition~\ref{le prop 1}, there are unique analytic functions $\omega_1,\omega_2\,:\, \C^+\rightarrow \C^+$, $z\mapsto \omega_1(z),\omega_2(z)$ satisfying~\eqref{le limit of omega} that solve~\eqref{le H system} in terms of $z$. Choosing $\mu_1=\mu_\alpha$, $\mu_2=\mu_\beta$ equation~\eqref{le H system} is equivalent to~\eqref{060102}; choosing $\mu_1=\mu_A$, $\mu_2=\mu_B$ it is equivalent to~\eqref{060101}. When no confusion can arise, we simply write~$\PP$ for $\PP_{\mu_1,\mu_2}(\omega_1,\omega_2,z)$. 

We call the system~\eqref{le H system} {\it linearly $S$-stable at} $(\omega_1,\omega_2)$ if
\begin{align}\label{le what stable means}
\Gamma_{\mu_1,\mu_2}(\omega_1,\omega_2)\deq \left\|\left(\begin{array}{cc}
-1& F_{\mu_1}'(\omega_2)-1  \\
F_{\mu_2}'(\omega_1)-1& -1\\
  \end{array}\right)^{-1} \right\|\le S\,,
\end{align}
 for some positive constant $S$.  In particular, the partial Jacobian matrix of~\eqref{le H system defs} given by
\begin{align*}
 \mathrm{D}\PP(\omega_1,\omega_2)\deq\left(\frac{\partial \PP}{\partial \omega_1}(\omega_1,\omega_2,z) \,,\,\frac{\partial \PP}{\partial\omega_2}(\omega_1,\omega_2,z) \right)=\left(\begin{array}{cc}
-1& F_{\mu_1}'(\omega_2)-1  \\
F_{\mu_2}'(\omega_1)-1 & -1\\
  \end{array}\right),
\end{align*}
has a bounded inverse at $(\omega_1,\omega_2)$. Note that $ \mathrm{D}\PP(\omega_1,\omega_2)$ is independent of $z$. The implicit function theorem reveals that, if~\eqref{le H system} is linearly $S$-stable at  $(\omega_1,\omega_2)$, then
\begin{align}\label{091460}
 \max_{z\in \mathcal{S}_{\mathcal{I}}(0,1)}|\omega_1'(z)|\le 2S\,,\qquad  \max_{z\in \mathcal{S}_{\mathcal{I}}(0,1)}|\omega_2'(z)|\le 2S\,.
\end{align}
In particular, $\omega_1$ and $\omega_2$ are Lipschitz continuous with constant $2S$. A more detailed analysis yields the following local stability result of the system $\PP_{\mu_1,\mu_2}(\omega_1,\omega_2,z)=0$.

\begin{lemma}[Proposition~4.1, \cite{BES15}]\label{le lemma perturbation of system}
Fix $z_0\in\C^+$. Assume that the functions $\widetilde\omega_1$, $\widetilde\omega_2$, $\widetilde{r}_1$, $\widetilde{r}_2\,:\,\C^+\rightarrow \C$ satisfy $\im\widetilde\omega_1(z_0)>0$, $\im\widetilde\omega_2(z_0)>0$ and 
 \begin{align}\label{la perturbed system}
 \PP_{\mu_1,\mu_2}(\widetilde\omega_1(z_0),\widetilde\omega_2(z_0),z_0)=\widetilde r(z_0)\,,
 \end{align}
where $\widetilde{r}(z)\deq(\widetilde r_1(z),\widetilde r_2(z))^\top$. Assume moreover that there is $\delta\in[0,1]$ such that
 \begin{align}\label{le apriori closeness}
 |\widetilde\omega_1(z_0)-\omega_1(z_0)|\le \delta\,,\qquad |\widetilde\omega_2(z_0)-\omega_2(z_0)|\le \delta\,,  
 \end{align}
 where $\omega_1(z)$, $\omega_2(z)$ solve the unperturbed system $\PP_{\mu_1,\mu_2}(\omega_1,\omega_2,z)=0$ with $\im \omega_1(z)\ge \im z$ and $\im \omega_2(z)\ge z$, $z\in\C^+$. 
 Assume that there is a constant $S$ such that $\PP$ is linearly $S$-stable at $(\omega_1(z_0),\omega_2(z_0))$, and assume in addition that there are strictly positive constants $K$ and $k$ with $k>\delta$ and with $k^2>\delta KS$ such that
\begin{align}\label{le stability sum ims}
k\le \im \omega_1(z_0)\le K\,,\qquad k\le \im \omega_2(z_0)\le K \,.
 \end{align}

 Then
\begin{align}\label{le conclusion of lemma}
 |\widetilde\omega_1(z_0)-\omega_1(z_0)|\le  2S \|\widetilde{r}(z_0)\|_2\,,\qquad|\widetilde\omega_2(z_0)-\omega_2(z_0)|\le  2S \|\widetilde{r}(z_0)\|_2\,.
\end{align}
\end{lemma}
In Section~\ref{s.7}, we will apply Lemma~\ref{le lemma perturbation of system} with the choices~$\mu_1=\mu_A$ and $\mu_2=\mu_B$. We thus next show that the system  $\Phi_{\mu_A,\mu_B}(\omega_A,\omega_B,z)=0$ is $S$-stable, for all $z\in \mathcal{S}_{\mathcal{I}}(0,1)$, and that~\eqref{le stability sum ims} holds uniformly on~$ \mathcal{S}_{\mathcal{I}}(0,1)$; see~\eqref{le domain S} for the definition.
\begin{lemma}[Lemma~5.1 and Corollary~5.2 of~\cite{BES15}] \label{cor.080601}
 Let $\mu_A$, $\mu_B$ be the probability measures from~\eqref{le empirical measures of A and B} satisfying the assumptions of Theorem~\ref{thm.091201}. Let $\omega_A,\omega_B$ denote the associated subordination functions of~\eqref{060101}. Let $\mathcal{I}$ be the interval in~Theorem~\ref{thm.091201}.   Then for  $N$ sufficiently large, the system $\PP_{\mu_A,\mu_B}(\omega_A,\omega_B,z)=0$ is $S$-stable with some positive constant~$S$, uniformly on $ \mathcal{S}_{\mathcal{I}}(0,1)$.  Moreover, there exist two strictly positive constants $K$ and $k$, such that for $N$ sufficiently large, we have
\begin{align}\label{le upper bound on omega AB}
 \max_{z\in \mathcal{S}_{\mathcal{I}}(0,1)}|\omega_A(z)|\le K\,,\qquad  \max_{z\in \mathcal{S}_{\mathcal{I}}(0,1)}|\omega_B(z)|\le K\,,
\end{align}
\begin{align}\label{le lower bound on omega AB}
\min_{z\in \mathcal{S}_{\mathcal{I}}(0,1)}\im\omega_A(z)\ge k\,,\qquad  \min_{z\in \mathcal{S}_{\mathcal{I}}(0,1)}\im\omega_B(z)\ge k\,.
\end{align}
\end{lemma}
\begin{remark}
 Under the assumptions of~Lemma~\ref{cor.080601}, the estimates in~\eqref{le lower bound on omega AB} can be extended as follows. There is $\tilde k>0$ such that
 \begin{align}\label{le former eta0 problem}
  \min_{z\in \mathcal{S}_{\mathcal{I}}(0,1)}(\im\omega_A(z)-\im z)\ge \tilde k\,,\qquad  \min_{z\in \mathcal{S}_{\mathcal{I}}(0,1)}(\im\omega_B(z)-\im z)\ge \tilde k\,.
 \end{align}
 This follows by combining~\eqref{le lower bound on omega AB} with the Nevanlinna representations in~\eqref{le neva for omega}.
 
\end{remark}

We conclude this section by mentioning that the general perturbation result in Lemma~\ref{le lemma perturbation of system} combined with Lemma~\ref{cor.080601}, can be used to prove~\eqref{le auxiliary statment}. We refer to~\cite{BES15} for details.

\section{Partial randomness decomposition}\label{s.5}
We use a decomposition of Haar measure on the unitary groups obtained in \cite{DS87} (see also~\cite{Mezzadri}): For a Haar distributed unitary matrix $ U\equiv U_N$, there exist a random vector $\mathbf{v}_1=(v_{11},\ldots, v_{1N})$, uniformly distributed on the complex unit $(N-1)$-sphere $\mathcal{S}_{\C}^{N-1}\deq \{\mathbf{x}\in\C^N\,:\, \mathbf{x}^*\mathbf{x}=1   \}$, and a Haar distributed unitary matrix $U^1\equiv U^1_{N-1}\in U(N-1) $, which is independent of $\mathbf{v}_1$, such that one has the decomposition
\begin{align*}
U=-\e{\mathrm{i}\theta_1}(I-\mathbf{r}_1\mathbf{r}_1^*)
\left(\begin{array}{ccc} 1 & ~\\ 
~ & U^{1}
\end{array}\right)=: -\e{\mathrm{i}\theta_1}R_1U^{\la 1\ra}\,,
\end{align*}
where 
\begin{align}
\mathbf{r}_1\deq\sqrt{2}\frac{\mathbf{e}_1+\e{-\mathrm{i}\theta_1}\mathbf{v}_1}{\|\mathbf{e}_1+\e{-\mathrm{i}\theta_1}\mathbf{v}_1\|_2}\,,     \qquad\qquad  R_1\deq I-\mathbf{r}_1\mathbf{r}_1^*\,,
 \end{align}
and where $	\theta_1$ is the argument of the first coordinate of the vector $\mathbf{v}_1$. 
More generally, for any $i\in \llbracket 1,N \rrbracket$, there exists an independent pair $(\mathbf{v}_i, U^{i})$, with $\mathbf{v}_i$ a uniformly distributed unit vector $\mathbf{v}_i$ and with  $U^{i}\in U(N-1)$ a Haar unitary, such that one has the decomposition
\begin{align}
 U=-\e{\mathrm{i}\theta_i}R_iU^{\langle i\rangle}\,,\qquad\mathbf{r}_i\deq\sqrt{2}\frac{\mathbf{e}_i+\e{-\mathrm{i}\theta_i}\mathbf{v}_i}{\|\mathbf{e}_i+\e{-\mathrm{i}\theta_i}\mathbf{v}_i\|_2}\,, \qquad R_i\deq I-\mathbf{r}_i\mathbf{r}_i^*\,,
\end{align}
where $U^{\la i\ra}$ is the unitary matrix with $\mathbf{e}_i$ as its $i$th column and $U^i$ as its $(i,i)$-matrix minor. 

With the above notation, we can write
\begin{align*}
H=A+R_i \widetilde{B}^{\la i\ra} R_i^*\,,
\end{align*}
for any $i\in \llbracket 1,N \rrbracket$, where we introduced the shorthand notation
\begin{align}
\widetilde{B}^{\la i\ra}\deq U^{\la i\ra}B \big(U^{\la i\ra}\big)^*\,. \label{0911401}
\end{align}
We further define
\begin{align}
H^{\la i\ra} \deq A+\widetilde{B}^{\la i \ra}\,,\quad\qquad G^{\la i \ra}(z)\deq (H^{\la i\ra}-z)^{-1}\,,\qquad\qquad z\in\C^+\,. \label{090820}
\end{align}
Note that $B^{\la i\ra}$, $H^{\la i\ra}$ and $G^{\la i\ra}$ are independent of $\mathbf{v}_i$. 

It is well known that for a uniformly distributed unit vector $\mathbf{v}_i\in\mathbb{C}^N$, there exists a Gaussian vector $\widetilde{\mathbf{g}}_i=(\widetilde{g}_{i1},\cdots, \widetilde{g}_{iN})\sim \mathcal{N}_{\mathbb{C}}(0,N^{-1}I)$  such that
\begin{align*}
\mathbf{v}_i=\frac{\widetilde{\mathbf{g}}_i}{\|\widetilde{\mathbf{g}}_i\|_2}\,.
\end{align*}
By definition, $\theta_i$ is also the argument of $\widetilde{g}_{ii}$. Set 
\begin{align}
g_{ik}\deq\e{-\mathrm{i}\theta_i}\widetilde{g}_{ik}\,, \qquad\qquad k\neq i\,, \label{def of g}
\end{align}
and introduce an $N_{\mathbb{C}}(0,N^{-1})$ variable $g_{ii}$ which is independent of the unitary matrix $U$ and of $\widetilde{\mathbf{g}}_{i}$. 
Then, we denote $\mathbf{g}_i\deq(g_{i1},\ldots,g_{iN})$ and note $\mathbf{g}_i\sim \mathcal{N}_{\mathbb{C}}(0,N^{-1}I)$. In addition, by definition, we have
\begin{align*}
\e{-\mathrm{i}\theta_i}\mathbf{v}_i-\mathbf{g}_i=\frac{|\widetilde{g}_{ii}|-g_{ii}}{\|\widetilde{\mathbf{g}}_i\|_2}\mathbf{e}_i+\Big(\frac{1}{\|\widetilde{\mathbf{g}}_i\|_2}-1\Big)\mathbf{g}_i\,.
\end{align*}
In subsequent estimates for $G_{ij}$, it is convenient to approximate $\mathbf{r}_i$ by 
\begin{align}
\mathbf{w}_i\deq\mathbf{e}_i+\mathbf{g}_i \label{091061}
\end{align} 
in the decomposition $U=-\e{\mathrm{i}\theta_i}R_iU^{\la i\ra}$, without changing the randomness of $U^{\la i\ra}$. To estimate the precision of this approximation, we require more notation: Let 
 \begin{align}
 W_i=W_i^*\deq I-\mathbf{w}_i\mathbf{w}_i^*\,, \qquad\qquad \widetilde{B}^{(i)}=W_i\widetilde{B}^{\la i\ra}W_i\,.  \label{091060}
 \end{align}
 Correspondingly, we  denote
 \begin{align}
 H^{(i)}\deq A+\widetilde{B}^{(i)}\,,\quad\qquad G^{(i)}(z)\deq(H^{(i)}-z)^{-1}\,,\qquad\qquad z\in\C^+\,. \label{091330}
 \end{align}

The following lemma shows that $\mathbf{r}_i$ can be replaced by $\mathbf{w}_i$ in Green function entries at the expense of an error that is below the precision we are interested in.
 \begin{lemma} \label{lem.091060} Fix $z=E+\ii\eta\in\C^+$ and choose indices $i,j,k\in\llbracket 1,N\rrbracket$. Suppose that
 \begin{align}
 &\max\Big\{|G_{kk}(z)|,  |G^{(i)}_{ij}(z)|\Big\}\prec 1\,, \nonumber\\
 & \max\Big\{|\mathbf{g}_i^*G^{(i)}{(z)}\mathbf{e}_j| ,  |\mathbf{g}_i^*\wt{B}^{\la i\ra}G^{(i)}(z)\mathbf{e}_j|\Big\}\prec 1\,, \label{0911111}
 \end{align}
hold. Then
 \begin{align}
\big|G_{kj}(z)-G^{(i)}_{kj}(z)\big|\prec\frac{1}{\sqrt{N\eta}} \label{091057}
 \end{align}
 holds, too.
 \end{lemma}
 
 \begin{proof}[Proof of Lemma \ref{lem.091060}] Fix $i,j,k\in\llbracket 1,N\rrbracket$. We first note that
 \begin{align*}
 \mathbf{r}_i=\mathbf{w}_i+\delta_{1i}\mathbf{e}_i+\delta_{2i}\mathbf{g}_i\,,
 \end{align*}
where
 \begin{align}\label{le delta with the norms}
 &\delta_{1i}\deq \bigg(\frac{\sqrt{2}}{\| \mathbf{e}_i+\e{-\mathrm{i}\theta_i}\mathbf{v}_i\|_2}-1\bigg)+\frac{\sqrt{2}}{\| \mathbf{e}_i+\e{-\mathrm{i}\theta_i}\mathbf{v}_i\|_2}\frac{|\widetilde{g}_{ii}|-g_{ii}}{\|\widetilde{\mathbf{g}}_i\|_2}\,,\nonumber\\
 &\delta_{2i}\deq \frac{\sqrt{2}}{\| \mathbf{e}_i+\e{-\mathrm{i}\theta_i}\mathbf{v}_i\|_2}\frac{1}{\|\widetilde{\mathbf{g}}_i\|_2}-1\,.
 \end{align}
 By the strong concentration of the norms in~\eqref{le delta with the norms} and $g_{ii},\widetilde{g}_{ii}\sim N_{\mathbb{C}}(0,N^{-1})$, we have
 \begin{align}
 |\delta_{1i}|\prec \frac{1}{\sqrt{N}}\,,\qquad\quad |\delta_{2i}|\prec \frac{1}{\sqrt{N}}\,. \label{091055}
 \end{align}
 Denote
 \begin{align*}
 \Delta_i\deq\mathbf{w}_i\mathbf{w}_i^*-\mathbf{r}_i\mathbf{r}_i^*\,.
 \end{align*}
 Fix now $z\in\C^+$. Dropping $z$ from the notation, a first order Neumann expansion of the resolvent yields
 \begin{align}
 G_{kj}=G^{(i)}_{kj}-\big(G(\Delta_i \wt{B}^{\la i\ra}W_i+W_i\wt{B}^{\la i\ra}\Delta_i+\Delta_i\wt{B}^{\la i\ra}\Delta_i)G^{(i)}\big)_{kj}\,. \label{091050}
 \end{align}
 Observe that the second term on the right side of~\eqref{091050} is a polynomial in the terms
 \begin{align}\def\arraystretch{1.5}\begin{array}{lllll}
 G^{(i)}_{ij}\,,  & \mathbf{g}_i^*G^{(i)}\mathbf{e}_j\, & \mathbf{g}_i^*\wt{B}^{\la i\ra}G^{(i)}\mathbf{e}_j\,,  &\mathbf{e}_i^*\wt{B}^{\la i\ra}G^{(i)}\mathbf{e}_j\,,\\
 G_{ki}\,, &\mathbf{e}_k^*G\mathbf{g}_i \, , &\mathbf{e}_k^*G\wt{B}^{\la i\ra} \mathbf{g}_i\, , &\mathbf{e}_k^*G\wt{B}^{\la i\ra} \mathbf{e}_i\, ,\\
 \mathbf{g}_i^*\wt{B}^{\la i\ra} \mathbf{e}_i\, , &\mathbf{e}_i^*\wt{B}^{\la i\ra}\mathbf{g}_i\, , & \mathbf{g}_i^*\wt{B}^{\la i\ra}\mathbf{g}_i, &\mathbf{e}_i^*\wt{B}^{\la i\ra}\mathbf{e}_i\,,
   \end{array}\label{0911100}
 \end{align}
 with coefficients of the form $\delta_{1i}^{k_1}\delta_{2i}^{k_2}$, for some nonnegative integers $k_1,k_2$ such that $k_1+k_2\geq 1$.
By assumption~\eqref{0911111} and the fact $\wt{B}^{\la i\ra}\mathbf{e}_i=b_i\mathbf{e}_i$, we further observe that the first four terms in~\eqref{0911100} are stochastically dominated by one. The last four terms are also stochastically dominated by one as follows from the trivial fact $\mathbf{e}_i^*\wt{B}^{\la i\ra}\mathbf{e}_i=b_i$ and Lemma~\ref{lem.091720}. The terms in the second line of~\eqref{0911100} are stochastically dominated by 
 \begin{align}
 |\mathbf{e}_k^*GQ^{\la i\ra}{\mathbf{x}_i}|\prec \|Q^{\la i\ra}\| \|G \mathbf{e}_k\|_2\lesssim \sqrt{(G^*G)_{kk}}=\sqrt{\frac{\Im G_{kk}}{\eta}}\prec {\frac{1}{\sqrt\eta}}\,,\label{091301}
 \end{align}
 with $Q^{\la i\ra}=I$ or $\wt{B}^{\la i\ra}$, and with ${\mathbf{x}_i}=\mathbf{e}_i$ or $\mathbf{g}_i$, where the last step follows from~\eqref{0911111}. Note that the terms in the second line of ~\eqref{0911100} appear only linearly in~\eqref{091050}. 
 Hence,~\eqref{091055},~\eqref{091301} and the order one bound for the first and last four terms in~\eqref{0911100} lead to ~\eqref{091057}.\qedhere
 \end{proof}

\section{Concentration with respect to the vector $\mathbf{g}_i$}\label{s.6}
In this section, we show that $G_{ii}^{(i)}$ concentrates around the partial expectation $\mathbb{E}_{\mathbf{g}_i}[G_{ii}^{(i)}]$, where $\mathbb{E}_{\mathbf{g}_i}[\,\cdot\,]$ is the expectation with respect to the collection $(\Re g_{ij}, \Im g_{ij})_{j=1}^N$. Besides the diagonal Green function entries $G^{(i)}_{ii}=\mathbf{e}_i^* G^{(i)}\mathbf{e}_i$ the following combinations are of importance
\begin{align}\label{le T and S}
 T_i(z)\deq \bs{g}_i^*G^{(i)}(z) \bs{e}_i\,,\qquad S_i(z)\deq   \bs{g}_i^* \widetilde B^{\la i\ra}G^{(i)}\bs{e}_i\,,\qquad\quad z\in\C^+\,.
\end{align}
The estimation of $\E_{\bs{g}_i}[G_{ii}^{(i)}]$, carried out in the Sections~\ref{section partial concentration} and~\ref{s.7}, involves the quantities~$T_i$ and~$S_i$. From a technical point of view, it is convenient to be able to go back and forth between~$T_i$, $S_i$ and their expectations~$\mathbb{E}_{\mathbf{g}_i} [T_i]$, $\mathbb{E}_{\mathbf{g}_i} [S_i]$. Thus after establishing concentration estimates for $G_{ii}^{(i)}$ in Lemma~\ref{lem090501} below,  we establish in Corollary~\ref{cor090501} concentration estimates for~$T_i$ and ~$S_i$ where we also give a rough bounds on $T_i$, $S_i$ and related quantities. We need some more notation: for a general random variable $X$ we define
\begin{align}\label{the former Q}
 \IE_{\bs{g}_i} X\deq X-\E_{\bs{g}_i} X\,.
\end{align}
The main task in this section is to prove the following lemma.

\begin{lemma}  \label{lem090501} Suppose that the assumptions of Theorem~\ref{thm.091201} are satisfied and let $\gamma>0$. Fix $z=E+\ii\eta\in \mathcal{S}_{\mathcal{I}}(\eta_{\mathrm{m}},1)$ and assume that 
\begin{align}
\big| G_{ii}(z)-(a_i-\omega_B(z))^{-1}\big| \prec N^{-\frac{\gamma}{4}}\,,\qquad \big| G_{ii}^{(i)}(z)-(a_i-\omega_B(z))^{-1}\big| \prec N^{-\frac{\gamma}{4}}\,, \label{090830}
\end{align}
uniformly in $i\in\llbracket 1,N\rrbracket$. Then
\begin{align}
&\max_{i\in\llbracket 1,N\rrbracket}\big|\IE_{\mathbf{g}_i}[G_{ii}^{(i)}(z)]\big|\prec \frac{1}{\sqrt{N\eta}}\,.\label{0907100}
\end{align} 
\end{lemma}

\begin{proof}[Proof of Lemma \ref{lem090501}]
In this proof we fix $z\in \mathcal{S}_{\mathcal{I}}(\eta_{\mathrm{m}},1)$.  Recall the definition of $G^{\la i\ra}(z)$ in~\eqref{090820} and note that $G^{\la i\ra}(z)$ is independent of $\mathbf{v}_i$ (or $\mathbf{g}_i$). It is therefore natural to expand $G^{(i)}(z)$ around $G^{\la i\ra}(z)$ and to use the independence between $G^{\la i\ra}(z)$ and $\mathbf{g}_i$ in order to verify the concentration estimates. However, by construction, we have
\begin{align}
G^{\la i\ra}_{ii}(z)=\frac{1}{a_i+b_i-z}\,,  \label{091405}
\end{align}
which may be as large as $1/\eta$, depending on $a_i$, $b_i$ and $z$. To circumvent problems coming from instabilities in $G^{\la i\ra}_{ii}(z)$, we may use a  ``regularization" trick to enhance stability in the $\mathbf{e}_i$-direction: instead of considering the Green function of $H^{(i)}=A+\widetilde{B}^{(i)}$ directly, we first consider the ($z$-dependent) matrix 
\begin{align}
H^{\{i\}}(z)\deq A+\widetilde{B}^{(i)}-(b_i+\omega_B(z)-z)\mathbf{e}_i\mathbf{e}_i^*\,, \label{092002}
\end{align}
and define $G^{\{i\}}(z)\deq(H^{\{i\}}(z)-z)^{-1}$. Note that $H^{\{ i\}}(z)$ is not symmetric, yet since $\im \omega_B(z)\ge \im z$ by Proposition~\ref{le prop 1}, $ G^{\{i\}}(z)\ $ is in fact well-defined on the whole upper-half plane.
  Fix any $j\in \llbracket 1,N\rrbracket$. Using the rank-one perturbation formula~\eqref{091002Kevin}, we get
\begin{align}
G_{ij}^{(i)}(z)=G^{\{i\}}_{ij}(z)-\frac{(b_i+\omega_B(z)-z)G^{\{i\}}_{ii}(z)G^{\{i\}}_{ij}(z)}{1+(b_i+\omega_B(z)-z)G^{\{i\}}_{ii}(z)}=\frac{G^{\{i\}}_{ij}(z)}{1+(b_i+\omega_B(z)-z)G^{\{i\}}_{ii}(z)}\,.\label{091001}
\end{align}
Some algebra then reveals that
\begin{align}
G_{ij}^{(i)}(z)=\frac{G_{ij}^{\{i\}}(z)-(b_i+\omega_B(z)-z)\IE_{\mathbf{g}_i}[G^{\{i\}}_{ii}(z)]\, G_{ij}^{(i)}(z)}{1+(b_i+\omega_B(z)-z)\mathbb{E}_{\mathbf{g}_i}[G^{\{i\}}_{ii}(z)]}\,.\label{091004}
\end{align}
By assumption~\eqref{090830} and identity~\eqref{091001}, we have
\begin{align}
\Big|G^{(i)}_{ii}(z)-\frac{1}{a_i-\omega_B(z)}\Big|\prec N^{-\frac{\gamma}{4}}\, ,\qquad \Big|G^{\{i\}}_{ii}(z)-\frac{1}{a_i-b_i-2\omega_B(z)+z}\Big|\prec N^{-\frac{\gamma}{4}}\,.\label{091003}
\end{align}
Note that $|\im (a_i-b_i-2\omega_B(z)+z)|\ge \im \omega_B(z)>0$,
 thus both denominators are  well separated away from 0 by Lemma~\ref{cor.080601},  in particular
$G^{(i)}_{ii}(z)$ and $G^{\{i\}}_{ii}(z)$ are uniformly bounded.
We will prove below that
\begin{align}
\big|\IE_{\mathbf{g}_i}[G^{\{i\}}_{ii}(z)]\big|\prec \frac{1}{\sqrt{N\eta}}\,. \label{091012}
\end{align}
Setting $j=i$ in~\eqref{091004} and expressing the denominator on the right side 
 by using \eqref{091003}-\eqref{091012},
we get
\begin{align}
\Big|1+(b_i+\omega_B(z)-z)\mathbb{E}_{\mathbf{g}_i}[G^{\{i\}}_{ii}(z)]
-\frac{a_i-\omega_B(z)}{a_i-b_i-2\omega_B(z)+z}\Big|\prec N^{-\frac{\gamma}{4}}\,. \label{100505}
\end{align}
In particular, together with Lemma~\ref{cor.080601} and $\Im \omega_B(z)\geq \Im z$, this implies that the absolute value of the denominator on the right side of~\eqref{091004} is bounded from below by some strictly positive constant.
Thus, applying $\IE_{\mathbf{g}_i}$ on both sides of \eqref{091012}, we obtain the concentration estimate in~\eqref{0907100}. 

In the rest of the proof, we verify~\eqref{091012}. Consider next the matrix
\begin{align}
H^{[i]}(z)\deq A+\widetilde{B}^{\la i\ra}-(b_i+\omega_B(z)-z)\mathbf{e}_i\mathbf{e}_i^*\,, \label{091401}
\end{align}
and let $G^{[i]}(z)\deq (H^{[i]}(z)-z)^{-1}$. Note that $H^{[i]}(z)$ depends on $z$ and $\omega_B(z)$ and is thus not symmetric, yet since $\im \omega_B(z)\ge \im z$, $G^{[i]}(z)$ is, similar to $G^{\{i\}}$, well-defined on the whole upper-half plane. Note that 
\begin{align}
|G^{[i]}_{ii}(z)|=\left|\frac{1}{a_i-\omega_B(z)}\right| \lesssim 1  \label{091031}
\end{align} 
 since $\im \omega_B$ is uniformly bounded from below on $ \mathcal{S}_{\mathcal{I}}(\eta_{\mathrm{m}},1)$ by Lemma \ref{cor.080601}.

  We now expand $G^{\{i\}}(z)$ around $G^{[i]}(z)$ and use the independence among~$G^{[i]}(z)$ and $\mathbf{g}_i$.  For simplicity, we hereafter drop the $z$-dependence from the notation. We start with noticing that
\begin{align}
H^{\{i\}}-H^{[i]}&=-\mathbf{w}_i\mathbf{w}_i^* \widetilde{B}^{\la i\ra}- \widetilde{B}^{\la i\ra}\mathbf{w}_i\mathbf{w}_i^*+\mathbf{w}_i\mathbf{w}_i^* \widetilde{B}^{\la i\ra}\mathbf{w}_i\mathbf{w}_i^*\nonumber\\
&=-\mathbf{w}_i\mathbf{w}_i^* \widetilde{B}^{\la i\ra}- \big(\widetilde{B}^{\la i\ra}-\mathbf{w}_i^* \widetilde{B}^{\la i\ra}\mathbf{w}_i\; I\big)\mathbf{w}_i\mathbf{w}_i^*\nonumber\\
&=\mathbf{w}_i\mathbf{s}_i^*+\mathbf{t}_i\mathbf{w}_i^*\,,\label{092001}
\end{align}
where we introduced
\begin{align}
\bs{s}_i\deq-\widetilde{B}^{\la i\ra}\mathbf{w}_i\,,  \qquad \bs{t}_i\deq-\big(\widetilde{B}^{\la i\ra}-\mathbf{w}_i^* \widetilde{B}^{\la i\ra}\mathbf{w}_i\; I\big)\mathbf{w}_i\,. \label{090701}
\end{align}
Iterating the rank-one perturbation formula~\eqref{091002Kevin} once, we obtain
\begin{align}
G^{\{i\}}=X^{[i]}-\frac{X^{[i]}\bs{t}_i\mathbf{w}_i^*X^{[i]}}{1+\mathbf{w}_i^*X^{[i]}\bs{t}_i}\,,\qquad\quad X^{[i]}\deq G^{ [i]}-\frac{G^{[ i]}\mathbf{w}_i\bs{s}_i^* G^{[i]}}{1+\bs{s}_i^*G^{[i]}\mathbf{w}_i}\,.\label	{090755}
\end{align}
Taking the $(i,j)$th matrix entry in~\eqref{090755}, we have
\begin{align}
G_{ij}^{\{i\}}=G_{ij}^{[i]}+\frac{\Psi_{i,j}}{1+\Xi_i}\,, \label{090610}
\end{align}  
where we introduced
\begin{multline}
\Xi_i\deq (\bs{s}_i^*G^{[i]} \mathbf{w}_i)+(\mathbf{w}_i^*G^{[i]}\bs{t}_i)+(\bs{s}_i^*G^{[i]} \mathbf{w}_i)(\mathbf{w}_i^*G^{[i]}\bs{t}_i)-(\mathbf{w}_i^*G^{[i]}\mathbf{w}_i)( \bs{s}_i^*G^{[i]}\bs{t}_i)\, , \label{090837}
\end{multline}
and
\begin{align}
\Psi_{i,j}&\deq-(\mathbf{e}_i^*G^{[i]}\bs{t}_i)\big((\mathbf{w}_i^* G^{[i]}\mathbf{e}_j)+(\bs{s}_i^*G^{[i]} \mathbf{w}_i)( \mathbf{w}_i^* G^{[i]}\mathbf{e}_j)-( \mathbf{w}_i^*G^{[i]}\mathbf{w}_i)( \bs{s}_i^*G^{[i]}\mathbf{e}_j)\big)\nonumber\\
&-(\mathbf{e}_i^*G^{[i]}\mathbf{w}_i)\big( (\bs{s}_i^* G^{[i]}\mathbf{e}_j)+(\mathbf{w}_i^*G^{[i]}\bs{t}_i)( \bs{s}_i^* G^{[i]}\mathbf{e}_j)-(\bs{s}_i^* G^{[i]}\bs{t}_i)( \mathbf{w}_i^*G^{[i]}\mathbf{e}_j)\big)\,. \label{090811}
\end{align}
We now rewrite~\eqref{090610} as
\begin{align}
 G_{ij}^{\{i\}}=G_{ij}^{[i]}+\frac{\Psi_{i,j}-\IE_{\mathbf{g}_i}[\Xi_i]\,(G_{ii}^{\{i\}}-G_{ij}^{[i]})}{1+\mathbb{E}_{\mathbf{g}_i}[\Xi_i]}\,. \label{0907110}
\end{align}
Since $|G_{ii}^{\{i\}}|\prec 1$ (\cf~\eqref{091003}) and $|G_{ii}^{[i]}|\prec 1$ (\cf~\eqref{091031}), it suffices to verify the following statements to show~\eqref{091012}:
\begin{align}
\text{(i):}\quad \big|\IE_{\mathbf{g}_i}[\Xi_i]\big|\prec\frac{1}{\sqrt{N\eta}}\,,\quad\; \text{(ii):}  \quad \frac{1}{1+\mathbb{E}_{\mathbf{g}_i}[\Xi_i]}\prec 1\,,\quad\; \text{(iii):} \quad \big|\IE_{\mathbf{g}_i}[\Psi_{i,j}]\big|\prec \frac{1}{\sqrt{N\eta}}\,. \label{092010}
\end{align}

We first show claim $(i)$. Substituting the definitions in~\eqref{090701} into~\eqref{090837}, we have
\begin{align}
\Xi_i = &-\mathbf{w}_i^*\widetilde{B}^{\la i\ra} G^{ [i]} \mathbf{w}_i -\mathbf{w}_i^*G^{[ i]} \widetilde{B}^{\la i\ra}\mathbf{w}_i+\mathbf{w}_i^*\widetilde{B}^{\la i\ra}G^{[i]}\mathbf{w}_i\,\mathbf{w}_i^*G^{[ i]} \widetilde{B}^{\la i\ra}\mathbf{w}_i\nonumber\\
&\qquad-\mathbf{w}_i^* G^{[i]}\mathbf{w}_i \,\mathbf{w}_i^*\wt{B}^{\la i\ra} G^{[ i]} \wt{B}^{\la i\ra}\mathbf{w}_i+\mathbf{w}_i^*\wt{B}^{\la i\ra}\mathbf{w}_i\,\mathbf{w}_i^*G^{[ i]}\mathbf{w}_i\nonumber\\
&\qquad-\mathbf{w}_i^*\wt{B}^{\la i\ra}\mathbf{w}_i\,\mathbf{w}_i^*\wt{B}^{\la i\ra} G^{[i]}\mathbf{w}_i-\mathbf{w}_i^*\wt{B}^{\la i\ra}\mathbf{w}_i\,\mathbf{w}_i^*G^{[ i]} \mathbf{w}_i\,\mathbf{w}_i^*\widetilde{B}^{\la i\ra}G^{[i]}\mathbf{w}_i\,. \label{091040}
\end{align}
Let $Q_1^{\la i\ra}$ and $Q_2^{\la i\ra}$ each stand for either $I$ or $\widetilde{B}^{\la i\ra}$. Recalling that $\mathbf{w}_i=\mathbf{e}_i+\mathbf{g}_i$ and that $\mathbf{g}_i\sim \mathcal{N}_{\C}(0,N^{-1}I)$ is a complex Gaussian vector, we compute
\begin{align}\label{gaussian estimates auxiliary}
 \E_{\mathbf{g}_i}\mathbf{w}_i^* Q_1^{\la i\ra}G^{[i]}Q_2^{\la i\ra}\mathbf w_i=(Q_1^{\la i\ra}G^{[i]}Q_2^{\la i\ra})_{ii}+\ntr Q_1^{\la i\ra}G^{[i]}Q_2^{\la i\ra}
 \,.\end{align}
 To bound the right side of~\eqref{gaussian estimates auxiliary} we observe that $ |(Q_1G^{[i]}Q_2)_{ii}|\prec |G_{ii}^{[i]}|\prec 1$, where we used that $\mathbf{e}_i$ is an eigenvector of $\widetilde{B}^{\la i\ra}$ and~\eqref{091031}.
   (Notice that, for simplicity,  here and at several other places  we consistently use the notation $\prec$
  even when the stronger  $\leq$ or $\lesssim$ relations would also hold, \ie we use the concept stochastic domination
  even for estimating    almost surely bounded or  deterministic quantities.)

  To control the second term on the right side of~\eqref{gaussian estimates auxiliary}, we note that a first order Neumann expansion of the resolvents yields 
 \begin{align}\label{on the way}
   |\ntr Q_2^{\la i\ra}Q_1^{\la i\ra}G^{[i]}-\ntr Q_2^{\la i\ra}Q_1^{\la i\ra}G^{\la i\ra}|&=|\ntr Q_2^{\la i\ra}Q_1^{\la i\ra}G^{[i]}(b_i+\omega_B(z)-z)\mathbf{e}_i\mathbf{e}_i^*G^{\la i\ra}|\nonumber\\
   &\prec \frac{1}{N}\|Q_2^{\la i\ra}Q_1^{\la i\ra}G^{[i]}\mathbf{e}_i\mathbf{e}_i^*\|_2\|\mathbf{e}_i\mathbf{e}_i^*G^{\la i\ra}\|_2\nonumber\\
   &\prec \frac{1}{N}\|G^{[i]}\mathbf{e}_i\mathbf{e}_i^*\|_2\|\mathbf{e}_i\mathbf{e}_i^*G^{\la i\ra}\|_2\nonumber\\
   &\prec\frac{1}{N}{|(\mathbf{e}_i^*|G^{[i]}|^2\mathbf{e}_i)|}^{1/2}\,{|(\mathbf{e}_i^*|G^{\la i\ra}|^2\mathbf{e}_i)|}^{1/2}\,,
 \end{align}
where we used the boundedness of $b_i$, $\omega_B(z)$, $\|Q_1^{\la i\ra}\|$ and $\|Q_2^{\la i\ra}\|$. Notice next the identities
\begin{align}\label{wward}
 (|G^{[i]}(z)|^2)_{jj}=\frac{\im G^{[i]}_{jj}(z)}{(1-\delta_{ij})\eta+\delta_{ij}\im \omega_B(z)}\,,\qquad (|G^{\la i\ra}(z)|^2)_{jj}=\frac{\im G^{\la i\ra}_{jj}(z) }{\eta}\,,
\end{align}
for $j\in\llbracket 1,N\rrbracket$, with $z=E+\ii\eta$ and $|G|^2=G^*G$. The identities in~\eqref{wward} follow directly from the definitions in~\eqref{091401},~\eqref{090820} and the definition of the Green function. Since $|G_{ii}^{\la i\ra}|\prec \frac1\eta$ (\cf~\eqref{091405}), we obtain combining~\eqref{on the way} and~\eqref{wward} with~\eqref{090830} that
 \begin{align}\label{GGG}
  |\ntr Q_2^{\la i\ra}Q_1^{\la i\ra}G^{[i]}-\ntr Q_2^{\la i\ra}Q_1^{\la i\ra}G^{\la i\ra}|&\prec \frac{1}{N\eta}\,.
 \end{align}
Since $H^{\la i\ra }$ is a Hermitian finite-rank perturbation of $H$, we can apply~\eqref{091002} to conclude that 
\begin{equation}\label{QQQ}
|\ntr Q_2^{\la i\ra}Q_1^{\la i\ra}G-\ntr Q_2^{\la i\ra}Q_1^{\la i\ra}G^{\la i\ra}|\prec \frac1{N\eta}\, .
\end{equation}
 We will now show that $\ntr Q_1^{\la i\ra}G^{\la i \ra} Q_2^{\la i\ra}$ is bounded.
 Using the resolvent identities and $\ntr B^{\la i \ra} =\ntr B=0$,  we get
$$
   \ntr B^{\la i \ra} G^{\la i\ra} = 1- \ntr (A-z) G^{\la i\ra}, \qquad 
    \ntr B^{\la i \ra} G^{\la i\ra} B^{\la i \ra} = z+ \ntr (A-z) G^{\la i\ra} (A-z)\,,
$$ 
thus to control $\ntr Q_2^{\la i\ra}Q_1^{\la i\ra}G^{\la i\ra}$ we need to bound $\ntr (A-z)^k G^{\la i\ra}$
for $k=0,1,2$.
 Since $H^{\la i\ra }$ is a Hermitian finite-rank perturbation of $H$, we can apply~\eqref{091002} to conclude that 
\begin{align}\label{trGG}
|\ntr  (A-z)^k G^{\la i\ra} -\ntr  (A-z)^k G | & \prec \frac1{N\eta}\,,\quad \qquad k=0,1,2.
\end{align} 
Since  $A$ is diagonal with bounded matrix elements, we have $\ntr (A-z)^k G \lesssim \max_j |G_{jj}|\prec 1$,
where the last bound comes from \eqref{090830}.
This directly controls  $\ntr Q_2^{\la i\ra}Q_1^{\la i\ra}G^{\la i\ra}$ and then, using
\eqref{GGG} and  \eqref{QQQ},  we have 
\begin{equation}\label{trQGQ} 
   |\ntr Q_2^{\la i\ra}Q_1^{\la i\ra}G^{\la i \ra}| + |\ntr Q_2^{\la i\ra}Q_1^{\la i\ra}G| +  |\ntr Q_2^{\la i\ra}Q_1^{\la i\ra}G^{[i]}|\prec 1\, .
\end{equation}
 Thus, returning to~\eqref{gaussian estimates auxiliary}, we showed
\begin{align}\label{le expecation is order one}
 \E_{\mathbf{g}_i}\mathbf{w}_i^* Q_1^{\la i\ra}G^{[i]}Q_2^{\la i\ra}\mathbf w_i \prec   1 \, .
\end{align}

Using the Gaussian concentration estimates in~\eqref{091731} and $\mathbf{w}_i=\mathbf{e}_i+\mathbf{g}_i$, we obtain
 \begin{align}\label{le mini fluctuations}
 |\IE_{\mathbf{g}_i}\mathbf{w}_i^* Q_1^{\la i\ra}G^{[i]}(z)Q_2^{\la i\ra}\mathbf w_i |&\prec \left({\frac{(|Q_1^{\la i\ra}G^{[i]}(z)Q_2^{\la i\ra}|^2)_{ii}}{N}}\right)^{\frac12}+\left({\frac{\| Q_1^{\la i\ra}G^{[i]}(z)Q_2^{\la i\ra}\|_2^2}{N^2}}\right)^{\frac12}\nonumber\\
 &\prec \left({\frac{\im G^{[i]}_{ii}(z)}{N\im\omega_B(z)}}\right)^{\frac12}+\left({\frac{\im \ntr G^{[i]}(z)}{N\eta}}\right)^{\frac12} \prec \frac{1}{\sqrt{N\eta}} \,,
\end{align}
where we also used that $\mathbf{e}_i$ is an eigenvector of $\widetilde{B}^{\la i\ra}$, that $\widetilde{B}^{\la i\ra}$ is bounded and~\eqref{wward}. In the last step \eqref{091031} and \eqref{trQGQ} were used.
Combined with \eqref{le expecation is order one} we thus proved
\begin{equation}\label{le hura1}
\mathbf{w}_i^* Q_1^{\la i\ra}G^{[i]}(z)Q_2^{\la i\ra}\mathbf w_i \prec 1.
\end{equation}

 For a later use we remark that, combining~\eqref{QQQ} and~\eqref{le mini fluctuations}, we also proved
\begin{align}\label{le hura}
 \mathbf{w}_i^* Q_1^{\la i\ra}G^{[i]}(z)Q_2^{\la i\ra}\mathbf w_i =(Q_1^{\la i\ra}G^{[i]}Q_2^{\la i\ra})_{ii}+
  \ntr Q_1^{\la i\ra}GQ_2^{\la i\ra}+\OSD\Big(\frac{1}{\sqrt{N\eta}}\Big)\,.
\end{align}

In a very similar way we get, recalling that $\ntr B=0$ and $\|B\|\prec 1$, that
\begin{align}\label{le B}
\mathbf{w}_i^*\wt{B}^{\la i\ra}\mathbf{w}_i=b_i+\IE_{\mathbf{g}_i} \mathbf{w}_i^*\wt{B}^{\la i\ra}\mathbf{w}_i=b_i+\OSD\Big(\frac{1}{\sqrt{N}}\Big)\,.
\end{align}

To deal with terms containing four or six factors of $\mathbf{w}_i$ in $\IE_{\mathbf{g}_i}[\Xi_i]$ (see~\eqref{091040}), we use the following rough bound. For general random variables $X$ and $Y$ satisfying $|X|,|Y|\prec 1$, we have
\begin{align}\label{le IE correlations}
 \IE_{\mathbf{g}_i}[XY]=\IE_{\mathbf{g}_i}[\IE_{\mathbf{g}_i}[X] \IE_{\mathbf{g}_i}[Y]]+\IE_{\mathbf{g}_i}[\IE_{\mathbf{g}_i}[X]\E_{\mathbf{g}_i} [Y] ]+\IE_{\mathbf{g}_i}[\E_{\mathbf{g}_i} [X]\IE_{\mathbf{g}_i}[Y]]\,.
\end{align}
In particular we have $|\IE_{\mathbf{g}_i}[ XY]|\prec |\IE_{\mathbf{g}_i} X|+|\IE_{\mathbf{g}_i} Y|$, where we used basic properties of stochastic domination outlined in Subsection~\ref{stochastic domination section}.

Then, recalling the explicit form of $\Xi_i$ in~\eqref{091040} and  using~\eqref{le mini fluctuations},~\eqref{le hura1},~\eqref{le B} and~\eqref{le IE correlations}, a straightforward estimate shows that $|\IE_{\mathbf{g}_i}\Xi_i|\prec \frac{1}{\sqrt{N\eta}}$, and claim $(i)$ in~\eqref{092010} is thus proved.

We next show statement $(ii)$ of~\eqref{092010}. To compute the expectation $\E_{\mathbf{g}_i}[\Xi_i]$, we are going to use the identities
\begin{align}\label{le computation of traces}
 \ntr\wt{B}G=1-\ntr(A-z)G\,,\qquad\qquad \ntr\wt{B}G\wt{B}=z+\ntr(A-z)G(A-z)\,,
\end{align}
that follow from $(H-z)G(z)=1$ and $\ntr A=\ntr B=0$.
Invoking assumption~\eqref{090830} we have
\begin{align*}
\ntr \big((A-z)^kG\big)=\frac{1}{N}\sum_{i=1}^N \frac{(a_i-z)^k}{a_i-\omega_B}+\OSD(N^{-\frac{\gamma}{4}})\,,
\end{align*}
with~$k\in\N$. Recalling further the shorthand notation $m_\boxplus\equiv m_{\mu_A\boxplus\mu_B}$ and from~\eqref{060101} that
\begin{align}\label{le m shorthand notation Kevin}
 m_\boxplus=\frac{1}{N}\sum_{i=1}^N\frac{1}{a_i-\omega_B}\,,
\end{align}
we get from the above that
\begin{align}
\ntr G&=m_{\boxplus}+\OSD(N^{-\frac{\gamma}{4}})\,,\nonumber\\
 \ntr \big((A-z)G\big)&=1+(\omega_B-z)m_{\boxplus}+\OSD(N^{-\frac{\gamma}{4}})\,,\nonumber\\
 \ntr \big((A-z)^2G\big)&=\omega_B-2z+(\omega_B-z)^2m_{\boxplus}+\OSD(N^{-\frac{\gamma}{4}})\,. \label{091901}
\end{align}
Thus from~\eqref{le hura} we obtain
\begin{align}
\mathbf{w}_i^*G^{[ i]}\mathbf{w}_i&=m_{\boxplus}+\frac{1}{a_i-\omega_B}+\OSD(N^{-\frac{\gamma}{4}})\,,\nonumber\\
\mathbf{w}_i^*\widetilde{B}^{\la i\ra} G^{ [i]} \mathbf{w}_i&=-(\omega_B-z)m_{\boxplus}+\frac{b_i}{a_i-\omega_B}+\OSD(N^{-\frac{\gamma}{4}})\,,\nonumber\\
\mathbf{w}_i^*G^{[ i]} \widetilde{B}^{\la i\ra}\mathbf{w}_i&=-(\omega_B-z)m_{\boxplus}+\frac{b_i}{a_i-\omega_B}+(N^{-\frac{\gamma}{4}})\,,\nonumber\\
\mathbf{w}_i^*\wt{B}^{\la i\ra} G^{[ i]} \wt{B}^{\la i\ra}\mathbf{w}_i&=(\omega_B-z)+(\omega_B-z)^2 m_{\boxplus}+\frac{b_i^2}{a_i-\omega_B}+\OSD(N^{-\frac{\gamma}{4}})\,.\label{092021}
 \end{align}
Plugging~\eqref{092021} into ~\eqref{091040}, using  the identity $\omega_A+\omega_B=z-1/m_{\boxplus}$ and taking the expectation, a straightforward computation shows that 
\begin{align}
1+\mathbb{E}_{\mathbf{g}_i}[\Xi_i]=\frac{(\omega_A-b_i)(2\omega_B-a_i+b_i-z)m_\boxplus}{a_i-\omega_B}+\OSD (N^{-\frac{\gamma}{4}})\,. \label{0911300}
\end{align}
Then from Lemma~\ref{cor.080601} one observes that statement $(ii)$ of~\eqref{092010} holds. In fact, the first term on the right side of~\eqref{0911300} is bounded away from zero uniformly on $z\in \mathcal{S}_{\mathcal{I}}(\eta_{\mathrm{m}},1)$.

We move on to statement $(iii)$ of~\eqref{092010}. Let $Q_1^{\la i\ra}$ and $Q_2^{\la i\ra}$ each stand again for either~$I$ or~$\widetilde{B}^{\la i\ra}$. Then we note that
\begin{align}\label{le bubu}
 \mathbf{e}_i^* Q_1^{\la i\ra}G^{[i]}Q_2^{\la i\ra}\mathbf{w}_i=(Q_1^{\la i\ra}G^{[i]}Q_2^{\la i\ra})_{ii}+\OSD\Big(\frac{1}{\sqrt{N}}\Big)\,,
\end{align}
as follows from the Gaussian large deviation estimates in~\eqref{091731}, assumption~\eqref{090830} and the fact that $\mathbf{e}_i$ is an eigenvector of $Q_1^{\la i\ra}$, $Q_2^{\la i\ra}$ and $G^{[i]}$. Having established~\eqref{le bubu}, it suffices to recall~\eqref{le hura} and~\eqref{le B} to conclude that $|\IE_{\mathbf{g}_i}[\Psi_{i,i}]|\prec \frac{1}{\sqrt{N\eta}}$. This proves claim $(iii)$ in~\eqref{092010} and thus completes the proof of Lemma~\ref{lem090501}.
\end{proof}

\begin{corollary}  \label{cor090501}  Suppose that the assumptions of Theorem~\ref{thm.091201} are satisfied and let $\gamma>0$. Fix $z=E+\ii\eta\in \mathcal{S}_{\mathcal{I}}(\eta_{\mathrm{m}},1)$ and assume that 
\begin{align}
\big| G_{ii}^{(i)}(z)-(a_i-\omega_B(z))^{-1}\big|\prec N^{-\frac{\gamma}{4}}\,,\quad   \big| G_{ii}(z)-(a_i-\omega_B(z))^{-1}\big| \prec N^{-\frac{\gamma}{4}},\label{09083000}
\end{align}
hold for all $i\in\llbracket 1,N\rrbracket$. Letting $Q_1^{\la i \ra}$, $Q_2^{\la i \ra}$ stand for $I$ or $\wt{B}^{\la i\ra}$, and letting $\mathbf{x}_i$, $\mathbf{y}_i$ stand for~$\mathbf{g}_i$ or~$\mathbf{e}_i$, we have the bound
\begin{align}
\max_{i\in\llbracket 1,N\rrbracket}\big|\mathbf{x}_i^*Q_1^{\la i \ra}G^{(i)}(z) Q_2^{\la i \ra}\mathbf{y}_i|&\prec 1\,.\label{091090}
\end{align}
In particular, $|S_i(z)|, |T_i(z)|\prec1$, for all $i\in \llbracket 1,N\rrbracket$. Moreover, we have
\begin{align}
\max_{i\in\llbracket 1,N\rrbracket}\big|\IE_{\mathbf{g}_i}[T_i(z)]\big|\prec \frac{1}{\sqrt{N\eta}}\,,\qquad \qquad \max_{i\in\llbracket 1,N\rrbracket}\big|\IE_{\mathbf{g}_i}[S_i(z)]\big|\prec \frac{1}{\sqrt{N\eta}}\,.\label{0907101}
\end{align}
\end{corollary}

\begin{proof}
Using once more~\eqref{091002Kevin}, we can write
\begin{align*}
\mathbf{x}_i^*Q_1^{\la i \ra}G^{(i)} Q_2^{\la i \ra}\mathbf{y}_i=\mathbf{x}_i^*Q_1^{\la i \ra}G^{\{i\}}Q_2^{\la i \ra}\mathbf{y}_i-\frac{(b_i+\omega_B-z)\;\mathbf{x}_i^*Q_1^{\la i \ra}G^{\{i\}}\mathbf{e}_i\; \mathbf{e}_i^*G^{\{i\}}Q_2^{\la i \ra}\mathbf{y}_i}{1+(b_i+\omega_B-z)G^{\{i\}}_{ii}}.
\end{align*}
Hence to prove the bound in~\eqref{091090} it suffices to bound $\mathbf{x}_i^*Q_1^{\la i \ra}G^{\{i\}}Q_2^{\la i \ra}\mathbf{y}_i$ and $\mathbf{x}_i^*Q_1^{\la i \ra}G^{\{i\}}\mathbf{e}_i$ with the choices $Q_1^{\la i \ra},Q_2^{\la i \ra}=I$ or $\wt{B}^{\la i\ra}$ and $\mathbf{x}_i,\mathbf{y}_i=\mathbf{g}_i$ or $\mathbf{e}_i$. To do so, we expand~$G^{\{i\}}$ around~$G^{[i]}$. It turns out that $\mathbf{x}_i^*Q_1^{\la i \ra}G^{\{i\}}Q_2^{\la i \ra}\mathbf{y}_i$ and $\mathbf{x}_i^*Q_1^{\la i \ra}G^{\{i\}}\mathbf{e}_i$  both are of the form $\widetilde\Psi_i/(1+\Xi_i)$, where $\Xi_i$ is given in~\eqref{090837} and  $\widetilde\Psi_i$ is a polynomial of the quantities appearing in~\eqref{le hura},~\eqref{le B} and~\eqref{le bubu}. Then $(i)$ and $(ii)$ of~\eqref{092010} imply that $(1+\Xi_i)^{-1}\prec 1$, which together with the bounds in~\eqref{le hura} and~\eqref{le bubu} leads to the conclusion~\eqref{091090}.

 To prove~\eqref{0907101}, we follow, mutatis mutandis, the proof of~\eqref{0907100} by replacing $G_{ii}^{(i)}$ by $T_i=\mathbf{g}_i^*G^{(i)}\mathbf{e}_i$ or  $S_i=\mathbf{g}_i^*\wt{B}^{\la i\ra}G^{(i)}\mathbf{e}_i$. For instance, for $T_i$ the counterpart of~\eqref{091004} is
 \begin{align*}
 \mathbf{g}_i^*G^{(i)}\mathbf{e}_i=\frac{\mathbf{g}_i^*G^{\{i\}}\mathbf{e}_i-(b_i+\omega_B-z)\IE_{\mathbf{g}_i}\big[G^{\{i\}}_{ii}\big]\,\mathbf{g}_i^*G^{(i)}\mathbf{e}_i}{1+(b_i+\omega_B-z)\mathbb{E}_{\mathbf{g}_i}\big[G^{\{i\}}_{ii}\big]}\,.
 \end{align*}
 Now, according to~\eqref{100505},~\eqref{091012} and the bound $|T_i|\prec 1$ (\cf~\eqref{091090}), it suffices to show 
 \begin{align}
 \big|\IE_{\mathbf{g}_i}\big[\mathbf{g}_i^*G^{\{i\}}\mathbf{e}_i\big]\big|\prec \frac{1}{\sqrt{N\eta}}\,. \label{100601}
 \end{align}
The proof of~\eqref{100601} is nearly the same as the one of~\eqref{091012}. One can also use a similar argument for $S_i$ by using the bound $|S_i|\prec 1$ from~\eqref{091090}. We omit the details.
\end{proof}

\section{Identification of the partial expectation $\mathbb{E}_{\mathbf{g}_i}\big[G_{ii}^{(i)}\big]$} \label{section partial concentration}
In this section, we estimate the partial expectation $\mathbb{E}_{\mathbf{g}_i}\big[G_{ii}^{(i)}\big]$, which together with the concentration inequalities in Lemma~\ref{lem090501} lead to the following lemma. 
Recall the definition of $S_i$ and $T_i$ in~\eqref{le T and S}.
\begin{proposition}\label{lem.0910111}Suppose that the assumptions of Theorem~\ref{thm.091201} are satisfied and let $\gamma>0$. Fix $z=E+\ii\eta\in \mathcal{S}_{\mathcal{I}}(\eta_{\mathrm{m}},1)$. Assume that 
\begin{align}
\Big| G_{ii}^{(i)}(z)-\big(a_i-\omega_B(z)\big)^{-1}\Big| \prec N^{-\frac{\gamma}{4}}\,,\quad   \Big| G_{ii}(z)-\big(a_i-\omega_B(z)\big)^{-1}\Big| \prec N^{-\frac{\gamma}{4}} \,, \label{091071}
\end{align}
hold uniformly in $i\in\llbracket 1,N\rrbracket$. Then,
\begin{align}
\max_{i\in\llbracket 1,N\rrbracket}\Big|G_{ii}^{(i)}(z)-\big(a_i-{ \omega_B^c(z)}\big)^{-1}\Big|\prec \frac{1}{\sqrt{N\eta}}\,, \label{091151}
\end{align}
and
\begin{align}
\max_{i\in\llbracket 1,N\rrbracket}\bigg|S_i(z)+\frac{z-\omega_B^c(z)}{a_i-\omega_B^c(z)}\bigg|\prec\frac{1}{\sqrt{N\eta}}\,,\qquad\qquad \max_{i\in\llbracket 1,N\rrbracket}\big|T_i(z)|\prec\frac{1}{\sqrt{N\eta}}\,. \label{091152}
\end{align}
\end{proposition}

In the proof of Proposition~\ref{lem.0910111} we will need the following auxiliary lemma whose proof is postponed to the very end of this section.

\begin{lemma} \label{lem.0910115} Under the assumption of Proposition~\ref{lem.0910111}, the estimates
\begin{align}
\big|\ntr \big(\widetilde{B}^{\la i\ra}G^{(i)}(z)-\widetilde{B}G(z)\big)\big|\leq \frac{C}{N\eta}\,,\qquad \big|\ntr \big(\widetilde{B}^{\langle i\rangle}G^{(i)}(z)\widetilde{B}^{\langle i\rangle}-\widetilde{B}G(z)\widetilde{B}\big)\big|\leq \frac{C}{N\eta}\,, \label{091130}
\end{align}
and the bounds
\begin{align}
\big| \ntr  \big(\widetilde{B}^{\langle i\rangle}G^{(i)}(z)\big)\big|\prec 1\,,\qquad  \big|\ntr \big(\widetilde{B}^{\langle i\rangle}G^{(i)}(z)\widetilde{B}^{\langle i\rangle}\big)\big|\prec 1\,, \label{0910113}
\end{align}
hold uniformly in $i\in\llbracket 1,N \rrbracket$. Furthermore the estimates
\begin{align}
\big| \IE_{\mathbf{g}_i} \big[\ntr \big(\widetilde{B}^{\langle i\rangle}G^{(i)}(z)\big)\big]\big| \leq \frac{C}{N\eta}\,,\qquad
\big|\IE_{\mathbf{g}_i}\big[\ntr \big(\widetilde{B}^{\langle i\rangle}G^{(i)}(z)\widetilde{B}^{\langle i\rangle}\big)\big]\big|\leq \frac{C}{N\eta}\,, \label{090901}
\end{align}
hold uniformly in $i\in\llbracket 1,N \rrbracket$.
\end{lemma}

\begin{proof}[Proof of Proposition~\ref{lem.0910111}]
Fix $i\in\llbracket 1, N\rrbracket$. By the concentration results of Lemma~\ref{lem090501} and Corollary~\ref{cor090501}, it suffices to estimate $\mathbb{E}_{\mathbf{g}_i}[G_{ii}^{(i)}(z)]$, $\mathbb{E}_{\mathbf{g}_i}\big[S_i(z)\big]$ and  $\mathbb{E}_{\mathbf{g}_i}[T_i(z)]$ to establish~\eqref{091151} and~\eqref{091152}. Recall the definition of $H^{(i)}$ and $G^{(i)}$ from~\eqref{091330}. We start with the identity
\begin{align}
(A-z)G^{(i)}(z)=-\widetilde{B}^{(i)}G^{(i)}(z)+I\,,\qquad\qquad z\in\C^+\,. \label{091089}
\end{align}
Since $A$ is diagonal, we have
\begin{align}
(a_i-z)G_{ii}^{(i)}(z)=-\big(\widetilde{B}^{(i)}G^{(i)}(z)\big)_{ii}+1\,,\qquad\qquad z\in\C^+\,. \label{091102}
\end{align}
Therefore, to estimate $\mathbb{E}_{\mathbf{g}_i}[G_{ii}^{(i)}(z)]$, it suffices to estimate $\mathbb{E}_{\mathbf{g}_i}[(\widetilde{B}^{(i)}G^{(i)}(z))_{ii}]$ instead. Recalling the definitions in~\eqref{091061} and~\eqref{091060}, we have
\begin{align}
\big(\widetilde{B}^{(i)}G^{(i)}\big)_{ii} &=\mathbf{e}_i^*\big( I-\mathbf{e}_i\mathbf{e}_i^* -\mathbf{e}_i\mathbf{g}_i^*-\mathbf{g}_i\mathbf{e}_i^*-\mathbf{g}_i\mathbf{g}_i^*\big) \widetilde{B}^{\langle i\rangle}\nonumber\\ &\qquad \times \big( I-\mathbf{e}_i\mathbf{e}_i^* -\mathbf{e}_i\mathbf{g}_i^*-\mathbf{g}_i\mathbf{e}_i^*-\mathbf{g}_i\mathbf{g}_i^*\big) G^{(i)}\mathbf{e}_i\nonumber\\
&= - \mathbf{e}_i^*\big(\mathbf{e}_i\mathbf{g}_i^*+\mathbf{g}_i\mathbf{e}_i^*+\mathbf{g}_i\mathbf{g}_i^*\big) \widetilde{B}^{\langle i\rangle} \big( I-\mathbf{e}_i\mathbf{e}_i^* -\mathbf{e}_i\mathbf{g}_i^*-\mathbf{g}_i\mathbf{e}_i^*-\mathbf{g}_i\mathbf{g}_i^*\big) G^{(i)}\mathbf{e}_i\,. \label{091073}
\end{align}
Since $\mathbf{e}_i$ is an eigenvector of $\wt{B}^{\la i\ra}$ (\cf~\eqref{0911401}), we have $(\widetilde{B}^{\langle i\rangle}G^{(i)}\big)_{ii}=b_iG^{(i)}_{ii}$.
Since moreover~$B$ is traceless by assumption~\eqref{091072}, we have $\ntr \wt{B}^{\la i\ra}=\ntr B=0$.  Thus
the {\it apriori} estimates in~\eqref{091071}, the bound in~\eqref{091090}, and the following concentration estimates (\cf Lemma~\ref{lem.091720})
\begin{align}
|\mathbf{e}_j^*\mathbf{g}_i|\prec \frac{1}{\sqrt{N}}\,,\qquad |\mathbf{e}^*_j\widetilde{B}^{\langle i\rangle}\mathbf{g}_i|\prec \frac{1}{\sqrt{N}}\,, \qquad \big|\mathbf{g}_i^*\widetilde{B}^{\la i\ra}\mathbf{g}_i\big|\prec \frac{1}{\sqrt{N}}\,, \label{091080}
\end{align}
for all $j\in\llbracket 1,N\rrbracket$, imply that $\mathbf{g}_i^* \widetilde{B}^{\langle i\rangle}G^{(i)}\mathbf{e}_i$ is the only relevant term in~\eqref{091073}. Thus recalling from definition~\eqref{le T and S} that $S_i=\mathbf{g}_i^* \widetilde{B}^{\langle i\rangle}G^{(i)}\mathbf{e}_i$ we arrive at
\begin{align}
 \big|(\widetilde{B}^{(i)}G^{(i)})_{ii}+S_i\big|\prec \frac{1}{\sqrt{N}}\,. \label{090401bis}
\end{align}

Using integration by parts for complex Gaussian random variables, we compute~$\mathbb{E}_{\mathbf{g}_i}[S_i]$ next. Regarding $g$ and $\overline{g}$ as independent variables for computing $\partial_g f(g,\overline{g})$, we have
\begin{align}
\int_{\C} \overline{g} f(g,\overline{g})\,\e{-\frac{|g|^2}{\sigma^2}} \dd g\wedge \dd \overline{g}=\sigma^2\int_{\C}\partial_g f(g,\overline{g})\,\e{-\frac{|g|^2}{\sigma^2}}\dd g\wedge \dd\overline g\,, \label{091350}
\end{align}
for differentiable functions $f\,:\, \C^2\to\C$. Using~\eqref{091350} with $\sigma^2=1/N$ for each component of $\mathbf{g}_i=(g_{i1},\ldots, g_{iN})$, we have
\begin{align}
\mathbb{E}_{\mathbf{g}_i}[S_i]= \sum_{k=1}^N \mathbb{E}_{\mathbf{g}_i} \big[\overline{g}_{ik} (\widetilde{B}^{\langle i\rangle}G^{(i)})_{ki}\big]&= \frac{1}{N}\sum_{k=1}^N \mathbb{E}_{\mathbf{g}_i} \bigg[\frac{\partial (\widetilde{B}^{\langle i\rangle}G^{(i)})_{ki}}{\partial g_{ik}}\bigg]\,. \label{090370}
\end{align}
Using the definitions in~\eqref{091061},~\eqref{091060} and regarding $g_{ik}$, $\overline{g}_{ik}$ as independent variables, we have
\begin{align}
\frac{\partial W_i}{\partial g_{ik}}=-\mathbf{e}_k\mathbf{e}_i^*-\mathbf{e}_k\mathbf{g}_i^*\,, \label{091755}
\end{align}
so that
\begin{align}
&\frac{\partial \big(\widetilde{B}^{\langle i\rangle}G^{(i)}\big)_{ki}}{\partial g_{ik}}=\mathbf{e}_k^*\widetilde{B}^{\langle i\rangle}G^{(i)}\big(\mathbf{e}_k\mathbf{e}_i^*+\mathbf{e}_k\mathbf{g}_i^*\big)\widetilde{B}^{\langle i\rangle}\big(I-\mathbf{e}_i\mathbf{e}_i^* -\mathbf{e}_i\mathbf{g}_i^*-\mathbf{g}_i\mathbf{e}_i^*-\mathbf{g}_i\mathbf{g}_i^*\big) G^{(i)}\mathbf{e}_i\nonumber\\
&\qquad+\mathbf{e}_k^*\widetilde{B}^{\langle i\rangle}G^{(i)}\big(I-\mathbf{e}_i\mathbf{e}_i^* -\mathbf{e}_i\mathbf{g}_i^*-\mathbf{g}_i\mathbf{e}_i^*-\mathbf{g}_i\mathbf{g}_i^*\big)\widetilde{B}^{\langle i\rangle}\big(\mathbf{e}_k\mathbf{e}_i^*+\mathbf{e}_k\mathbf{g}_i^*\big) G^{(i)}\mathbf{e}_i\,. \label{091758}
\end{align} 
Since $\mathbf{e}_i$ is an eigenvector of $\widetilde B^{\la i \ra}$ with eigenvalue $b_i$, we further get from~\eqref{091758} that 
\begin{align}
\frac{\partial \big(\widetilde{B}^{\langle i\rangle}G^{(i)}\big)_{ki}}{\partial g_{ik}}&= \big(\widetilde{B}^{\langle i\rangle}G^{(i)}\big)_{kk}\big( \mathbf{g}_i^*\widetilde{B}^{\langle i\rangle}G^{(i)}\mathbf{e}_i-b_i\, \mathbf{g}_i^*G^{(i)}\mathbf{e}_i\big)\nonumber\\
&\qquad+\big(\widetilde{B}^{\langle i\rangle}G^{(i)}\widetilde{B}^{\langle i\rangle}\big)_{kk} \big(G_{ii}^{(i)}+\mathbf{g}_i^*G^{(i)}\mathbf{e}_i\big)\nonumber\\
&\qquad-\big(\widetilde{B}^{\langle i\rangle}G^{(i)}\big)_{kk}\,\big(G^{(i)}_{ii}+\mathbf{g}_i^*G^{(i)}\mathbf{e}_i\big)\, \big(\mathbf{e}_i^*\wt{B}^{\la i\ra}\mathbf{g}_i+\mathbf{g}_i^*\wt{B}^{\la i\ra}\mathbf{e}_i+\mathbf{g}_i^*\wt{B}^{\la i\ra}\mathbf{g}_i\big)\nonumber\\
&\qquad-\big(G^{(i)}_{ii}+\mathbf{g}_i^*G^{(i)}\mathbf{e}_i\big)\big(\mathbf{e}_i^*\wt{B}^{\la i\ra}\mathbf{e}_k\big)\big(\mathbf{e}_k^*\widetilde{B}^{\langle i\rangle}G^{(i)}\mathbf{e}_i\big)\nonumber\\
&\qquad-\big(G^{(i)}_{ii}+\mathbf{g}_i^*G^{(i)}\mathbf{e}_i\big)\big(\mathbf{g}_i^*\wt{B}^{\la i\ra}\mathbf{e}_k\big)\big(\mathbf{e}_k^*\widetilde{B}^{\langle i\rangle}G^{(i)}\mathbf{e}_i\big)\nonumber\\
&\qquad-\big(G^{(i)}_{ii}+\mathbf{g}_i^*G^{(i)}\mathbf{e}_i\big)\big(\mathbf{e}_i^*\wt{B}^{\la i\ra}\mathbf{e}_k\big)\big(\mathbf{e}_k^*\widetilde{B}^{\langle i\rangle}G^{(i)}\mathbf{g}_i\big) \nonumber\\ 
&\qquad-\big(G^{(i)}_{ii}+\mathbf{g}_i^*G^{(i)}\mathbf{e}_i\big)\big(\mathbf{g}_i^*\wt{B}^{\la i\ra}\mathbf{e}_k\big)\big(\mathbf{e}_k^*\widetilde{B}^{\langle i\rangle}G^{(i)}\mathbf{g}_i\big)\,.\label{0911450}
\end{align}
Plugging~\eqref{0911450} into~\eqref{090370} and rearranging, we get
\begin{align}
\mathbb{E}_{\mathbf{g}_i}[S_i]&= \mathbb{E}_{\mathbf{g}_i}\Big[\ntr  \big(\widetilde{B}^{\langle i\rangle}G^{(i)}\big)\big( \mathbf{g}_i^*\widetilde{B}^{\langle i\rangle}G^{(i)}\mathbf{e}_i-b_i \; \mathbf{g}_i^*G^{(i)}\mathbf{e}_i\big)\Big]\nonumber\\
&\quad+\mathbb{E}_{\mathbf{g}_i}\Big[\ntr  \big(\widetilde{B}^{\langle i\rangle}G^{(i)}\widetilde{B}^{\langle i\rangle} \big)\big(G^{(i)}_{ii}+\mathbf{g}_i^*G^{(i)}\mathbf{e}_i\big)\Big]\nonumber\\
&\quad-\mathbb{E}_{\mathbf{g}_i}\Big[\ntr \big(\widetilde{B}^{\langle i\rangle}G^{(i)}\big)\;\big(\mathbf{e}_i^*\wt{B}^{\la i\ra}\mathbf{g}_i+\mathbf{g}_i^*\wt{B}^{\la i\ra}\mathbf{e}_i+\mathbf{g}_i^*\wt{B}^{\la i\ra}\mathbf{g}_i\big)\big(G^{(i)}_{ii}+\mathbf{g}_i^*G^{(i)}\mathbf{e}_i\big)\Big]\nonumber\\
&\quad-\frac{1}{N} \mathbb{E}_{\mathbf{g}_i}\Big[\Big(b_i^2G^{(i)}_{ii}+\mathbf{g}_i^*\big(\wt{B}^{\la i\ra}\big)^2G^{(i)}\mathbf{e}_i+\mathbf{e}_i^*(\wt{B}^{\la i\ra})^2G^{(i)}\mathbf{g}_i+\mathbf{g}_i^*\big(\wt{B}^{\la i\ra}\big)^2G^{(i)}\mathbf{g}_i\Big)\nonumber\\
&\qquad \qquad\qquad\times\big(G^{(i)}_{ii}+\mathbf{g}_i^*G^{(i)}\mathbf{e}_i\big)\Big]\,. \label{091101}
\end{align}

We next claim that the last two terms on the right of~\eqref{091101} are small. Using the boundedness of $G_{ii}^{(i)}$ (following from the {\it apriori} estimate~\eqref{091071}), the bound \eqref{091090}, the concentration estimates in~\eqref{091080}, and estimate~\eqref{0910113} of the auxiliary Lemma~\ref{lem.0910115}, and the trivial bounds
\begin{align}
\big|\mathbf{x}_i^*(\wt{B}^{\la i\ra})^2G^{(i)}\mathbf{y}_i\big|\prec\frac{1}{\eta}\,,\qquad\quad \mathbf{x}_i,\,\mathbf{y}_i=\mathbf{e}_i\;\, \text{or}\;\, \mathbf{g}_i\,, \label{100701}
\end{align}
we see that the last two terms on the right side of~\eqref{091101} are indeed negligible, \ie 
\begin{align}
\mathbb{E}_{\mathbf{g}_i}[S_i]&= \mathbb{E}_{\mathbf{g}_i}\Big[\ntr  \big(\widetilde{B}^{\langle i\rangle}G^{(i)}\big)\big(S_i- b_i T_i\big)\Big]\nonumber\\
&\qquad\quad+\mathbb{E}_{\mathbf{g}_i}\Big[\ntr  \big(\widetilde{B}^{\langle i\rangle}G^{(i)}\widetilde{B}^{\langle i\rangle} \big)\big(G^{(i)}_{ii}+T_i\big)\Big]+\OSD\Big(\frac{1}{\sqrt{N}} \Big)+\OSD\Big(\frac{1}{N\eta} \Big)\,,  \label{091105}
\end{align} 
where we also used the definitions of $T_i$ and $S_i$ in~\eqref{le T and S}. From assumption \eqref{091071} and  Corollary~\ref{cor090501}, we have the bounds 
\begin{align}
\max_{j\in \llbracket 1,N\rrbracket}|G_{jj}^{(j)}|\prec 1\,, \qquad \max_{j\in \llbracket 1,N\rrbracket}|T_j|\prec 1\,,\qquad \max_{j\in \llbracket 1,N\rrbracket} |S_j|\prec 1\,. \label{0921200}
\end{align} 
We hence obtain from~\eqref{091105}, ~\eqref{0910113}, and the concentration estimates in~\eqref{090901},~\eqref{0907100} that
\begin{align}
\mathbb{E}_{\mathbf{g}_i}\big[S_i\big]&=\ntr \big( \widetilde{B}^{\langle i\rangle}G^{(i)}\big)\,\big(\mathbb{E}_{\mathbf{g}_i}\big[S_i\big]- b_i \, \mathbb{E}_{\mathbf{g}_i}\big[T_i\big]\big)\nonumber\\
&\qquad+\ntr  \big(\widetilde{B}^{\langle i\rangle}G^{(i)}\widetilde{B}^{\langle i\rangle}\big)\, \big(G^{(i)}_{ii}+\mathbb{E}_{\mathbf{g}_i}\big[T_i\big]\big)+\OSD\Big(\frac{1}{\sqrt{N\eta}} \Big)\,.   \label{091110}
\end{align} 
Repeating the above computations for $\mathbb{E}_{\mathbf{g}_i}[\mathbf{g}_i^* G^{(i)}\mathbf{e}_i]=\mathbb{E}_{\mathbf{g}_i}[T_i]$, we similarly obtain
\begin{align}
\mathbb{E}_{\mathbf{g}_i}\big[T_i\big]&=\ntr G^{(i)} \,\big(\mathbb{E}_{\mathbf{g}_i}\big[S_i\big]- b_i \; \mathbb{E}_{\mathbf{g}_i}\big[T_i\big]\big)\nonumber\\
&\qquad+\ntr  \big(\widetilde{B}^{\langle i\rangle}G^{(i)}\big)\, \big(G^{(i)}_{ii}+\mathbb{E}_{\mathbf{g}_i}\big[T_i\big]\big)+\OSD\Big( \frac{1}{\sqrt{N\eta}} \Big)\,.   \label{091111}
\end{align} 
Now, using the bounds in~\eqref{0921200}, the estimates~\eqref{091130} and $|\ntr G^{(i)}-\ntr G|\prec\frac{1}{N\eta}$ (following from~\eqref{091002}), we obtain from~\eqref{091110} and~\eqref{091111} the equations
\begin{align}
\mathbb{E}_{\mathbf{g}_i}\big[S_i\big]-\ntr \big(\wt{B}G\wt{B}\big)\big(G_{ii}^{(i)}+\mathbb{E}_{\mathbf{g}_i}\big[T_i\big]\big)&=\ntr\big( \wt{B}G\big)\,\big(\mathbb{E}_{\mathbf{g}_i}
\big[S_i\big]-b_i \mathbb{E}_{\mathbf{g}_i}\big[T_i\big]\big)+\OSD\Big(\frac{1}{\sqrt{N\eta}} \Big)\,, \label{091120}
\end{align}
and
\begin{align}
\mathbb{E}_{\mathbf{g}_i}\big[T_i\big]-\ntr \big(\wt{B}G\big)\,\big(G_{ii}^{(i)}+\mathbb{E}_{\mathbf{g}_i}\big[T_i\big]\big)&=\ntr \big(G\big) \,\big(\mathbb{E}_{\mathbf{g}_i}\big[S_i\big]-b_i \mathbb{E}_{\mathbf{g}_i}\big[T_i\big]\big) 
+\OSD\Big(\frac{1}{\sqrt{N\eta}} \Big)\,. \label{091121}
\end{align}

We first approximately solve~\eqref{091121} for $\mathbb{E}_{\mathbf{g}_i}[T_i]$ to show, under the assumptions of Proposition~\ref{lem.0910111}, that $|\E_{\bs{g}_i} T_i|\prec N^{-\frac{\gamma}{4}}$. To see this, we recall~\eqref{091102} and~\eqref{090401bis} which together with assumption~\eqref{091071} imply that 
\begin{align}
S_i=(a_i-z)G_{ii}^{(i)}-1+\OSD\Big(\frac{1}{\sqrt{N}}\Big)=-\frac{z-\omega_B}{a_i-\omega_B}+\OSD\big(N^{-\frac{\gamma}{4}}\big)\,. \label{0920137}
\end{align}
By the concentration estimate~\eqref{0907101}, we also have 
\begin{align}
\mathbb{E}_{\mathbf{g}_i}\big[S_i\big]=-\frac{z-\omega_B}{a_i-\omega_B}+\OSD\big( N^{-\frac{\gamma}{4}}\big)\,. \label{0914300}
\end{align}
In addition, by the identity $\widetilde{B}G=I-(A-z)G$, assumption~\eqref{091071} and equality~\eqref{le m shorthand notation Kevin}, we have, using the shorthand notation $m_\boxplus\equiv m_{\mu_A\boxplus\mu_B}$,
\begin{align}
\ntr G= m_{\boxplus}+\OSD\big(N^{-\frac{\gamma}{4}}\big)\,,\qquad \ntr \big(\wt{B}G\big)=(z-\omega_B)m_{\boxplus}+\OSD\big(N^{-\frac{\gamma}{4}}\big)\,. \label{0920130}
\end{align}
Substituting~\eqref{0914300} and assumption~\eqref{091071} into~\eqref{091121}, and using $|T_i|, |S_i|\prec 1$, we obtain 
\begin{align}
\big|\big(1-\ntr \big(\wt{B}G\big)+b_i\ntr G\big) \mathbb{E}_{\mathbf{g}_i}\big[T_i\big]\big|\prec N^{-\frac{\gamma}{4}}\,. \label{0920135}
\end{align}
Using~\eqref{0920130} and the second equation of~\eqref{060101}, we have
\begin{align}
\big|\big(1-\ntr \big(\wt{B}G\big)+b_i\ntr G\big)\big|&=\big| 1+(\omega_B-z+b_i)m_{\boxplus}\big|+\OSD\big(N^{-\frac{\gamma}{4}}\big)\nonumber\\
&=|(-\omega_A+b_i)m_{\boxplus}|+\OSD\big(N^{-\frac{\gamma}{4}}\big)\,.\label{091231}
\end{align}
Since $|(-\omega_A+b_i)m_{\boxplus}|\gtrsim 1$ by~\eqref{le lower bound on omega AB}, we have from~\eqref{0920135} that $\mathbb{E}_{\mathbf{g}_i}[T_i]\prec N^{-\frac{\gamma}{4}}$. Hence from~\eqref{0907101}, $ |T_i|\prec N^{-\frac{\gamma}{4}}$. Then solving~\eqref{091120} and~\eqref{091121} for $\mathbb{E}_{\mathbf{g}_i}\big[S_i\big]$, we obtain
\begin{align}
  \mathbb{E}_{\mathbf{g}_i} \big[S_i\big]&=-\frac{\ntr  \big(\widetilde{B}G\big)}{\ntr G} G_{ii}^{(i)}+\Big[\frac{\ntr\big( \widetilde{B}G\big)-\big(\ntr\big(\widetilde{B}G\big)\big)^2}{\ntr G}+\ntr \big( \widetilde{B}G\widetilde{B}\big)\Big]\big(G_{ii}^{(i)}+\mathbb{E}_{\mathbf{g}_i}\big[T_i\big]\big)\nonumber\\ &\qquad\qquad+\OSD\Big(\frac{1}{\sqrt{N\eta}}\Big)\,.\label{091145}
\end{align}
Averaging over the index $i$ and reorganizing, we get 
\begin{align}
\Big|\frac{\ntr  \big(\widetilde{B}G\big)-\big(\ntr \big(\widetilde{B}G\big)\big)^2}{\ntr  G}+\ntr  \big(\widetilde{B}G\widetilde{B}\big)\Big|=\bigg|\frac{\frac{1}{N}\sum_{i=1}^N\big(\frac{\ntr  (\widetilde{B}G)}{\ntr  G} G_{ii}^{(i)}+\mathbb{E}_{\mathbf{g}_i}\big[S_i\big]\big)+\OSD(\frac{1}{\sqrt{N\eta}})}{\frac{1}{N}\sum_{i=1}^N (G_{ii}^{(i)}+\mathbb{E}_{\mathbf{g}_i}\big[T_i\big])}\bigg|\,. \label{091132}
\end{align}
Now, recalling the concentration of $S_i$ in~\eqref{0907101} and estimate~\eqref{090401bis}, we have
\begin{align}
\big|\mathbb{E}_{\mathbf{g}_i} \big[S_i\big]+ (\widetilde{B}^{(i)}G^{(i)})_{ii}\big|\prec\frac{1}{\sqrt{N\eta}}\,.  \label{091225}
\end{align} 
Note that under assumption~\eqref{091071}, we can use Corollary~\ref{cor090501}
 to get~\eqref{091090}, which together with~\eqref{091071} implies that the assumptions in Lemma~\ref{lem.091060} in the case of $i=j=k$ are satisfied. Then, by~\eqref{091057} with $i=j=k$ and~\eqref{091102}, we get
 \begin{align}
 G_{ii}^{(i)}=G_{ii}+\OSD\Big(\frac{1}{\sqrt{N\eta}}\Big)\,,\qquad (\widetilde{B}^{(i)}G^{(i)})_{ii}=(\widetilde{B}G)_{ii}+\OSD\Big(\frac{1}{\sqrt{N\eta}}\Big)\,, \label{0920150}
 \end{align}
 for all $i\in \llbracket1,N\rrbracket$. 
Using~\eqref{091225} and~\eqref{0920150} we obtain
\begin{align}
\Big|\frac{1}{N}\sum_{i=1}^N G_{ii}^{(i)}-\ntr G\Big|\prec \frac{1}{\sqrt{N\eta}}\,,\qquad \Big|\frac{1}{N}\sum_{i=1}^N \mathbb{E}_{\mathbf{g}_i}\big[S_i\big]+\ntr \big(\wt{B} G\big)\Big|\prec\frac{1}{\sqrt{N\eta}}\,. \label{091220}
\end{align}
Substituting~\eqref{091220} and assumption~\eqref{091071} into the right side of~\eqref{091132}, and using $|\ntr G|\gtrsim 1$ (following from~\eqref{0920130}) and $|T_i|\prec N^{-\frac{\gamma}{4}}$, we obtain 
\begin{align}
\Big|\frac{\ntr  \big(\widetilde{B}G\big)-\big(\ntr \big(\widetilde{B}G\big)\big)^2}{\ntr  G}+\ntr \big( \widetilde{B}G\widetilde{B}\big)\Big|\prec\frac{1}{\sqrt{N\eta}}\,. \label{091221}
\end{align}
 Now, plugging~\eqref{091221} back into~\eqref{091145} gives
\begin{align}
\mathbb{E}_{\mathbf{g}_i}\big[S_i\big]=-\frac{\ntr  \big(\widetilde{B}G\big)}{\ntr  G} G_{ii}^{(i)}+\OSD\Big(\frac{1}{\sqrt{N\eta}}\Big)\,, \label{091155}
\end{align}
which together with~\eqref{091102} and~\eqref{091225} implies that
\begin{align}
\big(a_i-\omega_B^c\big)G_{ii}^{(i)}=1+\OSD\Big(\frac{1}{\sqrt{N\eta}}\Big)\,, \label{0920160}
\end{align}
in light of the definition of $\omega_B^c(z)$ in~\eqref{091260}. By assumption~\eqref{091071} we see that $\omega_B^c(z)=\omega_B(z)+\OSD(N^{-\frac{\gamma}{4}})$. Hence by~\eqref{le lower bound on omega AB}, we also have $\Im \omega_B^c(z)\geq c$ for some positive constant~$c$. Therefore, we get~\eqref{091151} from~\eqref{0920160}.

Then~\eqref{091155} and~\eqref{091151}, together with the definition of $\omega_B^c(z)$ in~\eqref{091260} and the concentration of $S_i$ in~\eqref{cor090501}, imply the estimate of $S_i$ in~\eqref{091152}.

Substituting~\eqref{091155} into~\eqref{091121}, we strengthen~\eqref{0920135} to
\begin{align}
\big|\big(1-\ntr \big(\wt{B}G\big)+b_i\ntr G\big) \mathbb{E}_{\mathbf{g}_i}\big[T_i\big]\big|\prec\frac{1}{\sqrt{N\eta}}\,. \label{091232}
\end{align}
Using~\eqref{091231} again, we obtain from~\eqref{091232} that
\begin{align*}
\big|\mathbb{E}_{\mathbf{g}_i}\big[T_i\big]\big|\prec \frac{1}{\sqrt{N\eta}}\,,
\end{align*}
which together with the concentration inequality in~\eqref{0907101} implies~\eqref{091152}. Therefore, we complete the proof of Lemma~\ref{lem.0910111}. 
\end{proof} 

We conclude this section with the proof of Lemma~\ref{lem.0910115}.

\begin{proof}[Proof of Lemma \ref{lem.0910115}]  We start by invoking the finite-rank perturbation formula~\eqref{091002} to get 
\begin{align*}
\big|\ntr Q_1^{\la i\ra}G^{(i)}Q_2^{\la i\ra}-\ntr Q_1^{\la i\ra}GQ_2^{\la i\ra}\big|\leq \frac{2\| Q_1^{\la i\ra}Q_2^{\la i\ra}\|}{N\eta}\,,\qquad Q_1^{\la i\ra},Q_2^{\la i\ra}=I\,\,\text{or}\,\,\wt{B}^{\la i\ra}\,.
\end{align*}
Hence, it suffices to verify~\eqref{091130} and~\eqref{0910113} with $G^{(i)}$ replaced by~$G$. 
 Recalling from Section~\ref{s.5} that $R_i=I-\mathbf{r}_i\mathbf{r}_i^*$ and using the fact that $R_i$ is a Householder reflection (in fact $\|\mathbf{r}_i\|_2^2=2$ by construction), we have $\wt{B}^{\la i\ra}=R_i\wt{B} R_i$. Then we write
\begin{align}
 \ntr \big(\widetilde{B}^{\langle i\rangle}G\big)=\ntr \big(R_i\wt{B} R_iG\big)=\ntr \big(\widetilde{B}G\big)+d_i\,, \label{0913120}
 \end{align} 
 with 
\begin{align*}
d_i\deq -\frac{1}{N}{\mathbf{r}}_i^*\wt{B}G{\mathbf{r}}_i-\frac{1}{N}\mathbf{r}_i^* G\wt{B}\mathbf{r}_i+\frac{1}{N}(\mathbf{r}_i^*\wt{B}\mathbf{r}_i)(\mathbf{r}_i G\mathbf{r}_i)\,.
\end{align*}
Using that $\|G\|\le 1/\eta$, we immediately get the deterministic bound $|d_i|\le C/N\eta$, for some numerical constant $C$. Together with~\eqref{0913120} this implies the first estimate in~\eqref{091130}. The second estimate in~\eqref{091130} is obtained in the similar way.

The bounds in~\eqref{0910113} follow by combining the sharp formulas for $\ntr(\wt{B}G)$
and $\ntr(\wt{B}G\wt{B})$ from 
 \eqref{0920130}, \eqref{091221} 
with the estimates in~\eqref{091130}. 

To prove~\eqref{090901}, we set $Q^{\la i \ra}=\wt{B}^{\la i\ra}$ or $(\wt{B}^{\la i\ra})^2$ and note that
\begin{align*}
\big|\IE_{\mathbf{g}_i}[\ntr \big(Q^{\la i \ra} G^{(i)}\big)]\big|=\big|\IE_{\mathbf{g}_i}[\ntr \big(Q^{\la i \ra} (G^{(i)}-G^{\la i\ra})\big)]\big|\leq \frac{2\|Q^{\la i \ra}\|}{N\eta}\,,
\end{align*}
where we used that $\mathbf{g}_i$ and $G^{\la i\ra}$ are independent, and once more~\eqref{091002}.
\end{proof}

\section{Proof of Theorem~\ref{thm.091201}: Inequalities~\eqref{0912100} and~\eqref{0913200}} \label{s.7}
In this section, we prove the estimates~\eqref{0912100} and~\eqref{0913200} of Theorem~\ref{thm.091201} via a continuity argument. We also prove Theorem~\ref{thm.091501}.

First, let us recall the matrix $\mathcal{H}$ and its Green function $\mathcal{G}$ defined in~\eqref{091960} and~\eqref{091961}, these are the natural counterparts of $H$ and $G$ with the roles of $A$ and $B$ interchanged. We can apply a similar partial randomness decomposition to the unitary $U^*$ in $\mathcal{H}$ as we did for $U$ in $H$ in Section~\ref{s.5}.
This means that, for any $i\in \llbracket 1,N \rrbracket$, there exists an independent pair $(\widehat{\mathbf{v}}_i,\mathcal{U}^{i})$, uniformly distributed  on $\mathcal{S}_{\C}^{N-1}$ and $ U(N-1) $, respectively, such that with $\widehat{\mathbf{r}}_i\deq\sqrt{2}(\mathbf{e}_i+\e{-\mathrm{i}\widehat{\theta}_i}\widehat{\mathbf{v}}_i)/\|\mathbf{e}_i+\e{-\mathrm{i}\widehat{\theta}_i}\widehat{\mathbf{v}}_i\|_2$, we have the decomposition $U^*=-\e{\mathrm{i}\widehat{\theta}_i}\mathcal{R}_i\,\mathcal{U}^{\langle i\rangle}$, where~$\widehat{\theta}_i$ is the argument of the $i$th coordinate of $\widehat{\mathbf{v}}_i$; where $\mathcal{R}_i\deq(I-\widehat{\mathbf{r}}_i\widehat{\mathbf{r}}_i^*)$ and $\mathcal{U}^{\la i\ra}$ is the unitary matrix with $\mathbf{e}_i$ as its $i$th column and $\mathcal{U}^i$ as its $(i,i)$-matrix minor. Analogously to $\mathbf{g}_i$ defined in~\eqref{def of g}, we define a Gaussian vector $\widehat{\mathbf{g}}_i=(\widehat{g}_{i1},\ldots,\widehat{g}_{iN})\sim \mathcal{N}_{\mathbb{C}}(0,N^{-1}I)$, to approximate $\e{-\mathrm{i}\widehat{\theta}_i}\widehat{\mathbf{v}}_i$. Setting $\widehat{\mathbf{w}}_i\deq\mathbf{e}_i+\widehat{\mathbf{g}}_i$ and $\mathcal{W}_i\deq I-\widehat{\mathbf{w}}_i\widehat{\mathbf{w}}_i^*$, we define
\begin{align*}
\mathcal{H}^{(i)}\deq B+\mathcal{W}_i\,\mathcal{U}^{\la i\ra} A\, (\mathcal{U}^{\la i\ra})^*\mathcal{W}_i\,,
\end{align*}
for all $i\in\llbracket 1,N\rrbracket$. Calligraphic letters are used to distinguish the decompositions of $\mathcal{H}$ from the decompositions of $H$.

Next, we introduce the $z$-dependent random variable
\begin{align}
\Lambda_{\mathrm{d}}(z)\deq &\max_{i\in\llbracket 1,N\rrbracket}| G_{ii}^{(i)}(z)-\big(a_i-\omega_B(z)\big)^{-1}|+\max_{i\in\llbracket 1,N\rrbracket}| G_{ii}(z)-\big(a_i-\omega_B(z)\big)^{-1}| \nonumber\\
&+\max_{i\in\llbracket 1,N\rrbracket}|\mathcal{G}^{(i)}_{ii}(z)-\big(b_i-\omega_A(z)\big)^{-1}|+\max_{i\in\llbracket 1,N\rrbracket}|\mathcal{G}_{ii}(z)-\big(b_i-\omega_A(z)\big)^{-1}|\,. \label{def of Lambda}
\end{align}
Moreover, for any $\delta\in [0,1]$  and $z\in \mathcal{S}_{\mathcal{I}}(\eta_{\mathrm{m}},1)$, we  define the following event  
\begin{align}
\Theta_{\mathrm{d}}(z, \delta)\deq \{ \Lambda_{\mathrm{d}}(z)\leq \delta\}\, . \label{def of Xi} 
\end{align}
 The subscript $\mathrm{d}$ refers to ``diagonal" matrix elements. 
With the above notation, we have the following lemma.
\begin{lemma} \label{lem.091282}Suppose that the assumptions of Theorem~\ref{thm.091201} are satisfied and 
fix $\gamma>0$.  For any $\varepsilon$ with  $0<\varepsilon \le \frac{\gamma}{8}$ and for any
 $D>0$ there  exists a positive integer  $N_2(D,\varepsilon)$ 
  such that the following holds:
 For any  fixed $z=E+\mathrm{i}\eta\in \mathcal{S}_{\mathcal{I}}(\eta_{\mathrm{m}},1)$ 
  there exists an event $\Omega_{\mathrm{d}}(z)\equiv \Omega_{\mathrm{d}}(z,D,\varepsilon)$ with 
\begin{align*}
\mathbb{P}\big(\Omega_{\mathrm{d}}(z)\big)\geq 1-N^{-D}, \qquad \qquad\forall N\ge N_2(D,\varepsilon)\,,
\end{align*}
such that if the estimate
\begin{align}
\mathbb{P}\big(\Theta_{\mathrm{d}}(z, N^{-\frac{\gamma}{4}})\big)\geq 1-N^{-D}\big(1+N^5(1-\eta)\big)  \label{a priori bound on lambda_d}
\end{align}
holds for all $D>0$ and $N\geq N_1(D,\gamma,\varepsilon)$,  for some threshold $N_1(D,\gamma,\varepsilon)$, 
 then we also have
\begin{align}
\Theta_{\mathrm{d}}(z,N^{-\frac{\gamma}{4}})\cap \Omega_{\mathrm{d}}(z)\subset \Theta_{\mathrm{d}}\Big(z, \frac{N^\varepsilon}{\sqrt{N\eta}}\Big) \label{improved bound for lambda_d}
\end{align}
for all $N\geq N_3(D,\gamma,\varepsilon)\deq \max\{ N_1(D,\gamma,\varepsilon), N_2(D,\varepsilon)\}$.
\end{lemma}

\begin{proof} In this proof we fix $z\in \mathcal{S}_{\mathcal{I}}(\eta_{\mathrm{m}},1)$. By the definition of $\prec$ in Definition \ref{definition of stochastic domination}, we see that  assumption~\eqref{a priori bound on lambda_d} implies
\begin{align}
\big|G_{ii}^{(i)}(z)-(a_i-\omega_B^c(z))^{-1}\big|\prec N^{-\frac{\gamma}{4}}\,,\qquad \big|G_{ii}(z)-(a_i-\omega_B^c(z))^{-1}\big|\prec N^{-\frac{\gamma}{4}}\,,\label{091250}
\end{align}
and
\begin{align}
\big|\mathcal{G}_{ii}^{(i)}(z)-(b_i-\omega_A^c(z))^{-1}\big|\prec N^{-\frac{\gamma}{4}}\,,\qquad \big|\mathcal{G}_{ii}(z)-(b_i-\omega_A^c(z))^{-1}\big|\prec N^{-\frac{\gamma}{4}}\,. \label{091251}
\end{align}
Hence, we can use Corollary~\ref{cor090501} to get~\eqref{091090}. Together with the boundedness of~$G_{ii}^{(i)}$ and~$G_{ii}$ (\cf~\eqref{091250} and~\eqref{le lower bound on omega AB}) this implies that the assumptions in~\eqref{0911111} of Lemma~\ref{lem.091060} are satisfied when $i=j=k$. 
Thus~\eqref{091057} holds when $i=j=k$.
 Hence, invoking, (\ref{091250}) and Proposition~\ref{lem.0910111}, we get
\begin{align}
\big|G_{ii}^{(i)}(z)-(a_i-\omega_B^c(z))^{-1}\big|\prec \frac{1}{\sqrt{N\eta}}\,,\qquad \big|G_{ii}(z)-(a_i-\omega_B^c(z))^{-1}\big|\prec \frac{1}{\sqrt{N\eta}}\,. \label{091270}
\end{align}
Switching the roles of $A$ and $B$ as well as $U$ and $U^*$, and further using~\eqref{091835}, we also get 
\begin{align}
\big|\mathcal{G}_{ii}^{(i)}(z)-(b_i-\omega_A^c(z))^{-1}\big|\prec \frac{1}{\sqrt{N\eta}}\,,\qquad \big|\mathcal{G}_{ii}(z)-(b_i-\omega_A^c(z))^{-1}\big|\prec \frac{1}{\sqrt{N\eta}}\,, \label{091271}
\end{align}
under~\eqref{091251}. 

Now, we state the conclusions~\eqref{091270} and~\eqref{091271} in a more explicit quantitative form
 assuming~\eqref{a priori bound on lambda_d} which is a quantitative form of~\eqref{091250}-\eqref{091251}.
Namely, we show that the inequalities
\begin{align}
\big|G_{ii}^{(i)}(z)-(a_i-\omega_B^c(z))^{-1}\big|\leq \frac{N^{\frac{\varepsilon}{2}}}{\sqrt{N\eta}}\,,\qquad \big|G_{ii}(z)-(a_i-\omega_B^c(z))^{-1}\big|\leq \frac{N^{\frac{\varepsilon}{2}}}{\sqrt{N\eta}}\,, \nonumber\\
\big|\mathcal{G}_{ii}^{(i)}(z)-(b_i-\omega_A^c(z))^{-1}\big|\leq  \frac{N^{\frac{\varepsilon}{2}}}{\sqrt{N\eta}}\,,\qquad \big|\mathcal{G}_{ii}(z)-(b_i-\omega_A^c(z))^{-1}\big|\leq \frac{N^{\frac{\varepsilon}{2}}}{\sqrt{N\eta}}\,. 
\label{bound of G_ii on the event Xi and Omega}
\end{align}
hold on the event $\Theta_{\mathrm{d}}(z,N^{-\frac{\gamma}{4}})\cap \Omega_{\mathrm{d}}(z)$, when $N\geq N_3(D,\gamma, \varepsilon)$.
Here $\Omega_{\mathrm{d}}(z)$ is an event determined as the intersection of the ``typical" events in all
 the concentration estimates in Sections~\ref{s.5}-\ref{section partial concentration}. 

 To see this more precisely,  we go back to the proofs in these sections. The concentration estimates
  always involved quantities of the form  $\mathbb{IE}_{\mathbf{g}_i}[\mathbf{g}_i^*Q\mathbf{x}]$ with 
  $\mathbf{x}=\mathbf{g}_i, \mathbf{e}_i$ and some explicit matrix $Q$ 
  that is independent of $\mathbf{g}_i$ but often $z$-dependent. The total number of such estimates
  was linear in $N$. Thus,  according to Lemma \ref{lem.091720}, for any (small) $\varepsilon>0$ and (large) $D>0$, there exists an event $\Omega_{\mathrm{d}}(z,D, \varepsilon)$ with
\begin{align}
\mathbb{P}\big(\Omega_{\mathrm{d}}(z,D,\varepsilon)\big)\geq 1-N^{-D} \label{probability of omega_d}
\end{align}
such that all estimates of the form
\begin{align}
|\mathbb{IE}_{\mathbf{g}_i}[\mathbf{g}_i^*Q\mathbf{e}_i]|\leq \frac{N^{\frac{\varepsilon}{4}}}{\sqrt{N}}\|Q\mathbf{e}_i\|_2\,, \qquad\qquad  |\mathbb{IE}_{\mathbf{g}_i}[\mathbf{g}_i^*Q\mathbf{g}_i]|\leq \frac{N^{\frac{\varepsilon}{4}}}{N}\| Q\|_2 \label{bound of IE on Omega_d}
\end{align}
in Sections~\ref{s.5}-\ref{section partial concentration} hold on $\Omega_{\mathrm{d}}(z,D,\varepsilon)$ for all $N\geq N_2(D,\varepsilon)$. In addition, the threshold $N_2(D,\varepsilon)$ is independent~of the spectral parameter~$z$.

We  now follow the proofs in Sections~\ref{s.5}--\ref{section partial concentration} to the letter but we 
  use~\eqref{probability of omega_d},~\eqref{bound of IE on Omega_d} and~\eqref{a priori bound on lambda_d} instead of the $\prec$ relation. Instead of~\eqref{091270} and~\eqref{091271}, we find that the analogous but more quantitative bounds ~\eqref{bound of G_ii on the event Xi and Omega} hold
    on the intersection of the events $\Theta_{\mathrm{d}}(z, N^{-\frac{\gamma}{4}})$ and $\Omega_{\mathrm{d}}(z,D,\varepsilon)$.

It remains to show that on the event $\Theta_{\mathrm{d}}(z,N^{-\frac{\gamma}{4}})\cap \Omega_{\mathrm{d}}(z)$, 
\begin{align}
|\omega_A^c(z)-\omega_A(z)|\leq \frac{CN^{\frac{\varepsilon}{2}}}{\sqrt{N\eta}}\,,\qquad  |\omega_B^c(z)-\omega_B(z)|\leq \frac{CN^{\frac{\varepsilon}{2}}}{\sqrt{N\eta}}\, \label{0912801}
\end{align}
hold when $N\geq N_3(D,\gamma, \varepsilon)$.

To this end, we use the stability of the system $\Phi_{\mu_A,\mu_B}(\omega_A,\omega_B,z)=0$ as formulated in Lemma~\ref{le lemma perturbation of system}. By the definition of the approximate subordination functions $\omega_A^c(z)$ and~$\omega_B^c(z)$ in~\eqref{091260}, by the identity~\eqref{091850} and by taking the average over the index $i$ in the estimates in~\eqref{bound of G_ii on the event Xi and Omega}, we get the system of equations
\begin{align}
&m_H(z)=m_A(\omega_B^c(z))+r_A(z)\,,\nonumber\\
&m_H(z)=m_B(\omega_A^c(z))+r_B(z)\,,\nonumber\\
&\omega_A^c(z)+\omega_B^c(z)=z-\frac{1}{m_H(z)}\,, \label{091280}
\end{align}
where the error terms $r_A$ and $r_B$ satisfy
\begin{align*}
|r_A(z)|\leq \frac{CN^{\frac{\varepsilon}{2}}}{\sqrt{N\eta}}\,,\qquad |r_B(z)|\leq\frac{CN^{\frac{\varepsilon}{2}}}{\sqrt{N\eta}}\,,
\end{align*}
on the event $\Theta_{\mathrm{d}}(z,N^{-\frac{\gamma}{4}})\cap \Omega_{\mathrm{d}}(z)$ when $N\geq N_3(D,\gamma, \varepsilon)$.
Using the definition of $\Theta_{\mathrm{d}}(z,\delta)$ in~\eqref{def of Xi},~\eqref{bound of G_ii on the event Xi and Omega} and the fact that $z\in \mathcal{S}_{\mathcal{I}}(\eta_{\mathrm{m}},1)$,  so $\omega_A(z)$ and $\omega_B(z)$ 
are well separated from the real axis, we have
\begin{align}
|\omega_A^c(z)-\omega_A(z)|\leq  CN^{-\frac{\gamma}{4}}\,,\qquad  |\omega_B^c(z)-\omega_B(z)|\leq CN^{-\frac{\gamma}{4}}\,. \label{091715}
\end{align}
on the event $\Theta_{\mathrm{d}}(z,N^{-\frac{\gamma}{4}})\cap \Omega_{\mathrm{d}}(z)$ when $N\geq N_3(D,\gamma, \varepsilon)$. 
Hence, plugging the third equation of~\eqref{091280} into the first two and using~\eqref{le lower bound on omega AB} together with~\eqref{091715}, we get
\begin{align*}
\Phi_{\mu_A,\mu_B}(\omega_A^c(z),\omega_B^c(z),z)=\widetilde{r}(z)\,,
\end{align*}
where $\widetilde{r}(z)=(\widetilde{r}_A(z),\widetilde{r}_B(z))$ with
\begin{align}
|\widetilde{r}_A(z)|\leq \frac{CN^{\frac{\varepsilon}{2}}}{\sqrt{N\eta}}\,,\qquad |\widetilde{r}_B(z)|\leq \frac{CN^{\frac{\varepsilon}{2}}}{\sqrt{N\eta}}\, \label{091701}
\end{align}
on the event $\Theta_{\mathrm{d}}(z,N^{-\frac{\gamma}{4}})\cap \Omega_{\mathrm{d}}(z)$ when $N\geq N_3(D,\gamma, \varepsilon)$.
Therefore, by  Lemma~\ref{le lemma perturbation of system}, we get~\eqref{0912801}. Hence, we completed the proof of Lemma~\ref{lem.091282}. 
\end{proof}

Given Lemma~\ref{lem.091282}, we next prove Theorem \ref{thm.091201} via a continuity argument similarly to \cite{EYY}. 
 \begin{proof}[Proof of~\eqref{0912100} of Theorem~\ref{thm.091201}] 
 From Theorem~1.2 of~\cite{Kargin}, we see that for $\eta=1$, we have
\begin{align}
\max_{ i\in\llbracket 1,N\rrbracket}\big| G_{ii}(z)-(a_i-\omega_B(z))^{-1}\big|\prec N^{-\frac{\gamma}{2}},\qquad \max_{ i\in\llbracket 1,N\rrbracket}\big| \mathcal{G}_{ii}(z)-(b_i-\omega_A(z))^{-1}\big|\prec N^{-\frac{\gamma}{2}}  \label{0912105}
\end{align}
if $0<\gamma\leq 1/7$ (say). In addition, owing to the estimate $\|G\|\leq 1/\eta$, assumption~\eqref{0911111} obviously holds for $\eta=1$. Hence, by Lemma~\ref{lem.091060} in the case of $i=j=k$ and its analogue for $\mathcal{G}_{ii}^{(i)}$, we have
\begin{align}
\max_{ i\in\llbracket 1,N\rrbracket}\big| G_{ii}^{(i)}(z)-(a_i-\omega_B(z))^{-1}\big|\prec N^{-\frac{\gamma}{2}}\,,\qquad \max_{ i\in\llbracket 1,N\rrbracket}\big| \mathcal{G}_{ii}^{(i)}(z)-(b_i-\omega_A(z))^{-1}\big|\prec N^{-\frac{\gamma}{2}}. \label{0912106}
\end{align}
Hence, for any $E\in\mathcal{I}$ and $D>0$,
\begin{align}
\mathbb{P}\big(\Theta_{\mathrm{d}}(E+\mathrm{i}, N^{-\frac{3\gamma}{8}})\big)\geq 1-N^{-D}, \label{initial step for diagonal case}
\end{align}
holds for all $N\geq N_0(D,\gamma)$ with some $N_0(D,\gamma)>0$.  In the sequel we will apply Lemma~\ref{lem.091282}
with the choice 
\begin{align*}
N_1(D,\gamma,\varepsilon)\deq\max\big\{N_0(D,\gamma), N_2(D,\varepsilon)\big\}\,;
\end{align*}
 in particular we have $N_3(D,\gamma,\varepsilon) = N_1(D,\gamma,\varepsilon)$.

Next, we define the lattice
\begin{align}
\widehat{\mathcal{S}}_{\mathcal{I}}(\eta_{\mathrm{m}},1)\deq \mathcal{S}_{\mathcal{I}}(\eta_{\mathrm{m}},1)\cap N^{-5}\{\Z\times \ii \Z\}\,. \label{100715}
\end{align}
Thanks to the Lipschitz continuity of the Green function, \ie $\|G(z)-G(z')\|\leq N^2|z-z'|$ for any $z,z'\in \mathcal{S}_{\mathcal{I}}(\eta_{\mathrm{m}},1)$, it suffices to show~\eqref{0912100} on the lattice $ \widehat{\mathcal{S}}_{\mathcal{I}}(\eta_{\mathrm{m}},1)$. We now fix $E\in \mathcal{I}\cap N^{-5}\mathbb{Z}$ and decrease $\eta$ from  $\eta=1$ down to $N^{-1+\gamma}$ in steps of size $N^{-5}$. Recall the events $\Theta_{\mathrm{d}}(z,\delta)$ and $\Omega_{\mathrm{d}}(z)$ in Lemma \ref{lem.091282}, and choose the same $\varepsilon<\frac{\gamma}{8}$ in $ \Omega_{\mathrm{d}}(z,D,\varepsilon)$ for all $z$. For simplicity, we omit the real  part $E$ from the notation and rewrite
\begin{align*}
\Theta_{\mathrm{d}}(\eta,\delta)\deq\Theta_{\mathrm{d}}(E+\mathrm{i}\eta,\delta),\qquad  \Omega_{\mathrm{d}}(\eta)\deq\Omega_{\mathrm{d}}(E+\mathrm{i}\eta).
\end{align*}
Our aim is to show that for any $\eta\in [\eta_{\mathrm{m}},1]$, 
\begin{align}
\Theta_{\mathrm{d}}\big(\eta, N^{-\frac{3\gamma}{8}}\big)\cap \Omega_{\mathrm{d}} (\eta-N^{-5})\subset \Theta_{\mathrm{d}}\big(\eta-N^{-5}, N^{-\frac{3\gamma}{8}}\big). \label{key inclusion}
\end{align}
To see (\ref{key inclusion}), we first notice that by the Lipschitz continuity of the Green function and of the subordination functions $\omega_A(z)$ and $\omega_B(z)$ (see~\eqref{091460}),  we have 
\begin{align}
\Theta_{\mathrm{d}}\big(\eta, N^{-\frac{3\gamma}{8}}\big)\subset \Theta_{\mathrm{d}}\big(\eta-N^{-5}, N^{-\frac{3\gamma}{8}}+
 C N^{-3}\big)\subset \Theta_{\mathrm{d}}\big(\eta-N^{-5}, N^{-\frac{\gamma}{4}}\big), \label{one step continuity}
\end{align}
where the last step is obtained by choosing $\gamma>0$ sufficiently small. Now, we start from (\ref{initial step for diagonal case}). By (\ref{one step continuity}), we get
\begin{align*}
\mathbb{P}\big(\Theta_{\mathrm{d}}\big(1-N^{-5}, N^{-\frac{\gamma}{4}}\big)\big)\geq \mathbb{P}\big(\Theta_{\mathrm{d}}\big(1, N^{-\frac{3\gamma}{8}}\big)\big)\geq 1-N^{-D}.
\end{align*}
Hence, we can use Lemma~\ref{lem.091282} to get 
\begin{align}
\Theta_{\mathrm{d}}\big(1-N^{-5}, N^{-\frac{\gamma}{4}}\big)\cap \Omega_{\mathrm{d}}(1-N^{-5})&\subset \Theta_{\mathrm{d}}\bigg(1-N^{-5}, \frac{N^\varepsilon}{\sqrt{N(1-N^{-5})}}\bigg)\nonumber\\ &\subset \Theta_{\mathrm{d}}\Big(1-N^{-5}, N^{-\frac{3\gamma}{8}}\Big)\,, \label{apply lemma 7.1 to smaller eta}
\end{align}
which together with (\ref{one step continuity}) implies~\eqref{key inclusion} with $\eta=1$. Now, replacing $1$ by $1-N^{-5}$, we get from~\eqref{apply lemma 7.1 to smaller eta}, ~\eqref{initial step for diagonal case} and the fact $\mathbb{P}(\Omega_{\mathrm{d}}(1-N^{-5}))\geq 1-N^{-D}$ for $N\geq N_2(D,\varepsilon)$ that
\begin{align}
\mathbb{P}\big(\Theta_{\mathrm{d}}\big(1-N^{-5}, N^{-\frac{3\gamma}{8}}\big)\big)\geq 1-2N^{-D} \label{second step for diagonal case}
\end{align}
holds for all $N\geq N_3(D,\gamma,\varepsilon)$. 
Now, using~\eqref{second step for diagonal case} instead of~\eqref{initial step for diagonal case}, we get~\eqref{key inclusion} for $\eta=1-N^{-5}$. Iterating this argument, we obtain for any $\eta\in [\eta_\mathrm{m},1]\cap N^{-5}\mathbb{Z}$ that
\begin{align*}
\Theta_{\mathrm{d}}\big(1, N^{-\frac{3\gamma}{8}}\big)\cap \Omega_{\mathrm{d}} (1-N^{-5})\cap\ldots \cap\Omega_{\mathrm{d}} (\eta)\subset \Theta_{\mathrm{d}}\big(\eta, N^{-\frac{3\gamma}{8}}\big)\,.
\end{align*}
Hence, we have
\begin{align*}
\mathbb{P}\big(\Theta_{\mathrm{d}}\big(\eta, N^{-\frac{3\gamma}{8}}\big)\big)\geq 1-N^{-D}\big(1+N^5(1-\eta)\big)
\end{align*}
for all $N\geq N_3(D,\gamma,\varepsilon)$, which further implies  
\begin{align*}
\mathbb{P}\big(\Theta_{\mathrm{d}}\big(\eta-N^{-5}, N^{-\frac{\gamma}{4}}\big)\big)\geq 1-N^{-D}\big(1+N^5(1-\eta)\big)
\end{align*}
for all $N\geq N_3(D,\gamma,\varepsilon)$, by using~\eqref{one step continuity}. Then, using Lemma~\ref{lem.091282} again, we obtain 
\begin{align}
\mathbb{P}\Big(\Theta_{\mathrm{d}}\Big(\eta, \frac{N^\varepsilon}{\sqrt{N\eta}}\Big)\Big)\geq 1-N^{-D}\big(2+N^5(1-\eta)\big)
\end{align}
uniformly for all $\eta\in [\eta_\mathrm{m},1]\cap N^{-5}\mathbb{Z}$, when $N\geq N_3(D,\gamma,\varepsilon)$. Finally, by continuity, we can  extend the bounds from $z$ in the discrete lattice to the entire domain
$\mathcal{S}_{\mathcal{I}}(\eta_{\mathrm{m}},1)$. We then~get
\begin{align}
\max_{i\in\llbracket 1,N\rrbracket}&\Big|G_{ii}(z)- \frac{1}{a_i-\omega_B(z)}\Big|\prec \frac{1}{\sqrt{N\eta}}\,,\qquad \max_{i\in\llbracket 1,N\rrbracket}\Big|G_{ii}^{(i)}(z)- \frac{1}{a_i-\omega_B(z)}\Big|\prec \frac{1}{\sqrt{N\eta}}\,, \label{estimate for both G_ii and G_ii^i}
\end{align}
uniformly on $\mathcal{S}_{\mathcal{I}}(\eta_{\mathrm{m}},1)$, where we used the definitions of $\Theta_{\mathrm{d}}(z,\delta)$ in~\eqref{def of Xi} and of $\prec$ in Definition \ref{definition of stochastic domination}. This concludes the proof of~\eqref{0912100}.
\end{proof}

Having established~\eqref{0912100},~\eqref{0913200} of Theorem~\ref{thm.091201} and Theorem~\ref{thm.091201} are direct consequences.
\begin{proof}[Proof of~\eqref{0913200} of Theorem~\ref{thm.091201}] 
 It suffices to note that~\eqref{0913200} is a direct consequence of~\eqref{0912100} and the facts $m_H(z)=N^{-1}\sum_{i=1}^NG_{ii}(z)$ and $m_{A\boxplus B}(z)=N^{-1}\sum_{i=1}^N(a_i-\omega_B)^{-1}$. 
\end{proof}
\begin{proof}[Proof of Theorem \ref{thm.091501}] Using the spectral decomposition of the Green function $G$, we have
\begin{eqnarray}
\max_{j\in \llbracket 1,N\rrbracket}\Im G_{jj}(z)=\max_{j\in \llbracket 1,N\rrbracket}\sum_{i=1}^N\frac{|u_{ij}|^2\eta}{|\lambda_i-E|^2+\eta^2}=\sum_{i=1}^N\frac{||\mathbf{u}_i||_\infty^2\eta}{|\lambda_i-E|^2+\eta^2}\,,\qquad z\in\C^+\,. \label{031331}
\end{eqnarray}
Fix a small $\gamma>0$. For any $\lambda_i\in \mathcal{I}$, we set $E=\lambda_i$ on the right side of~\eqref{031331} and use~\eqref{0912100} to bound the left side of it with $z=\lambda_i+\ii\eta$, $\eta=N^{-1+\gamma}$. Then we obtain
\begin{eqnarray*}
{||\mathbf{u}_{i}||^2_\infty}\prec \eta=N^{-1+\gamma}\,.
\end{eqnarray*}
Since $\gamma>0$ is arbitrarily small, we get~\eqref{091511}. This completes the proof of Theorem~\ref{thm.091501}.\qedhere
\end{proof}

\section{Proof of Theorem~\ref{thm.091201}: Inequalities~\eqref{100710}} \label{s.8}
In this section, we prove~\eqref{100710} of Theorem~\ref{thm.091201}. Note that, from (\ref{estimate for both G_ii and G_ii^i}) in the proof of~\eqref{0912100} in Theorem~\ref{thm.091201}, we know that the following estimates hold uniformly on $ \mathcal{S}_{\mathcal{I}}(\eta_{\mathrm{m}},1)$,
\begin{align}
\Big| G_{ii}^{(i)}(z)-\big(a_i-\omega_B(z)\big)^{-1}\Big| \prec \frac{1}{\sqrt{N\eta}}\,,\quad  \Big| G_{ii}(z)-\big(a_i-\omega_B(z)\big)^{-1}\Big| \prec \frac{1}{\sqrt{N\eta}}\,. \label{100501}
\end{align}
  Taking~\eqref{100501} as an input, we follow the discussion in Sections~\ref{s.6}-\ref{s.7} to prove the estimate~\eqref{100710} with the following modifications. We introduce the quantities
\begin{align}\label{le off-diagonal T and S}
T_{i,j}(z)\deq\mathbf{g}_i^*G^{{ (i)}}(z)\mathbf{e}_j\,,\qquad  S_{i,j}(z)\deq\mathbf{g}_i^*\wt{B}^{\la i\ra}G^{{ (i)}}(z) \mathbf{e}_j\,,\qquad\qquad z\in\C^+\,.
\end{align}
that generalize $T_i(z)$ and $S_i(z)$ defined in~\eqref{le T and S}. In particular, $T_i(z)\equiv T_{i,i}(z)$ and $S_i(z)\equiv S_{i,i}(z)$, but we henceforth implicitly assume that $i\not=j$. (We use a comma in the subscripts of $T_{i,j}$, $S_{i,j}$ since they are not the entries of some matrix.) We often abbreviate $T_{i,j}\equiv T_{i,j}(z)$ and $S_{i,j}\equiv S_{i,j}(z)$.

We first establish the concentration estimates for $G_{ij}^{(i)}$ (see Lemma~\ref{lem093011}), and $T_{i,j}$ and $S_{i,j}$; see Lemma~\ref{cor1000101}. In Proposition~\ref{pro.100501} we then derive self-consistent equations for $\E_{\mathbf{g}_i}\big[T_{i,j}\big]$ and $\E_{\mathbf{g}_i}\big[S_{i,j}\big]$ that will show, together with concentration estimates, that $|G_{ij}^{(i)}|$, $|T_{i,j}|$, $|S_{i,j}|\prec \frac{1}{\sqrt{N\eta}}$, provided that $|G_{ij}^{(i)}| \prec 1$. We then close the argument via continuity.

We start with the analogue of Lemma \ref{lem090501} for the off-diagonal entries of $G^{{ (i)}}$. 
\begin{lemma}  \label{lem093011} Suppose that the assumptions of Theorem~\ref{thm.091201} are satisfied and let $\gamma>0$. Fix $z=E+\ii\eta\in \mathcal{S}_{\mathcal{I}}(\eta_{\mathrm{m}},1)$ and assume that 
\begin{align}
|G_{ij}^{(i)}(z)|\prec 1 \label{093022}\,,
\end{align}
for all $i,j\in\llbracket 1,N\rrbracket$, $i\not=j$. Then
\begin{align}
&\big|\IE_{\mathbf{g}_i}[G_{ij}^{(i)}(z)]\big|\prec \frac{1}{\sqrt{N\eta}}\,,\label{093013}
\end{align} 
for all $i,j\in\llbracket 1,N\rrbracket$, $i\not=j$. 
\end{lemma}
\begin{proof}[Proof of Lemma \ref{lem093011}] 
Recall $H^{[i]}$ and $H^{\{i\}}$ defined in~\eqref{091401} and~\eqref{092002}, as well as their Green functions $G^{[i]}$ and $G^{\{i\}}$. In the proof of Lemma~\ref{lem090501} we derived the identity
\begin{align}
G_{ij}^{(i)}(z)=\frac{G_{ij}^{\{i\}}(z)-(b_i+\omega_B(z)-z)\IE_{\mathbf{g}_i}[G^{\{i\}}_{ii}(z)]\,G_{ij}^{(i)}(z)}{1+(b_i+\omega_B(z)-z)\mathbb{E}_{\mathbf{g}_i}[G^{\{i\}}_{ii}(z)]}\,; \label{093020}
\end{align}
see~\eqref{091004}. With~\eqref{100501}, we see that assumption~\eqref{090830} is satisfied. Hence, we can use all the conclusions in the proof of Lemma~\ref{lem090501}.  Therefore, according to~\eqref{100505},~\eqref{091012} and assumption~\eqref{093022}, it suffices to show the concentration estimate $|\IE_{\mathbf{g}_i}[G^{\{i\}}_{ij}(z)]|\prec \frac{1}{\sqrt{N\eta}}$. To this end we expand $G^{\{i\}}$ around $G^{[i]}$. Recall from~\eqref{090610} that
\begin{align*}
G_{ij}^{\{i\}}=G^{[i]}_{i,j}+\frac{\Psi_{i,j}}{1+\Xi_i}\,,
\end{align*}
where~$\Xi_i$ is defined in~\eqref{090837} and $\Psi_{i,j}$ is defined in~\eqref{090811}. Recalling statements $(i)$ and $(ii)$ in~\eqref{092010}, it suffices to establish that $|\IE_{\mathbf{g}_i}[\Psi_{i,j}]|\prec\frac{1}{\sqrt{N\eta}}$. Note that $\Psi_{i,j}$ contains the terms listed in~\eqref{le hura},~\eqref{le B} and~\eqref{le bubu}, as well as the terms
\begin{align}
\mathbf{e}_i^*\wt{B}^{\la i\ra}G^{[i]}\mathbf{e}_j\,, \qquad\quad \mathbf{g}_i^*G^{[i]}\mathbf{e}_j\,,\qquad\quad \mathbf{g}_i^*\wt{B}^{\la i\ra}G^{[i]}\mathbf{e}_j\,. \label{100507}
\end{align}
Since $\mathbf{e}_i$ is an eigenvector of $\wt{B}^{\la i\ra}$  and of $H^{[i]}$, we have $\mathbf{e}_i^*\wt{B}^{\la i\ra}G^{[i]}\mathbf{e}_j=\delta_{ij} b_iG^{[i]}_{ii}$.
Moreover, using Lemma \ref{lem.091720} with $Q^{\la i \ra}=I$ or $\wt{B}^{\la i\ra}$, we have
\begin{align}
\big|\mathbf{g}_i^*Q^{\la i \ra}G^{[i]}\mathbf{e}_j\big|\prec \frac{1}{\sqrt{N}} \|Q^{\la i \ra}G^{[i]}\mathbf{e}_j\|_2\prec \frac{1}{\sqrt{N}} \|G^{[i]}\mathbf{e}_j\|_2=\bigg(\frac{\Im G_{jj}^{[i]}}{N\eta}\bigg)^{\frac12}\,. \label{100509}
\end{align}
To control $  G_{jj}^{[i]}$, we recall from~\eqref{091401} that the matrix $H^{[i]}$ is block-diagonal and we thus have, for $j\neq i$
\begin{align*}
G_{jj}^{[i]}=(A^i+U^iB^i(U^i)^*-zI_{N-1})^{-1}_{\ell\ell}\,,\qquad \ell\deq j\mathbf{1}(j<i)+(j-1)\mathbf{1}(j>i)\,,
\end{align*}
where $A^i$ and $B^i$ the are $(i,i)$-matrix minors of $A$ and $B$ respectively (obtained by removing the $i$th column and $i$th row) and $U^i\in U(N-1)$ is the $(i,i)$-matrix minor of $U^{\la i \ra}$ which is Haar distributed as seen at the beginning of Section~\ref{s.5}. Note that the matrix $A^i+U^iB^i(U^i)^*$ satisfies the assumptions of Theorem~\ref{thm.091201}. We thus have the estimate
\begin{align}
\max_{i\neq j}\big|G^{[i]}_{jj}(z)\big|\prec 1\,,\qquad\qquad z\in \mathcal{S}_{\mathcal{I}}(\eta_{\mathrm{m}},1)\,.  \label{100510}
\end{align}
Plugging this bound into~\eqref{100509} we get $|\mathbf{g}_i^*Q^{\la i \ra}G^{[i]}\mathbf{e}_j\big| \prec \frac{1}{\sqrt{N\eta}}$. The remaining part of the proof is nearly the same as the one of Lemma \ref{lem090501}. We omit the details.
\end{proof}

We have the following analogue of Corollary~\ref{cor090501}.
\begin{lemma}  \label{cor1000101}  Suppose that the assumptions of Theorem~\ref{thm.091201} are satisfied and let $\gamma>0$. For all $z=E+\ii\eta\in \mathcal{S}_{\mathcal{I}}(\eta_{\mathrm{m}},1)$, we have the bounds 
\begin{align}
\max_{i\not=j}|T_{i,j}(z)|\prec 1\,,\qquad \max_{i\not=j} |S_{i,j}(z)|\prec 1 \label{100155}
\end{align}
and the concentration estimates
\begin{align}
\max_{i\not=j}\big|\IE_{\mathbf{g}_i}[T_{i,j}(z)]\big|\prec \frac{1}{\sqrt{N\eta}}\,, \qquad 
\max_{i\not=j}\big|\IE_{\mathbf{g}_i}[S_{i,j}(z)]\big|\prec \frac{1}{\sqrt{N\eta}}\,.\label{100152}
\end{align}
\end{lemma}
\begin{proof} With the estimates in~\eqref{100501} and~\eqref{100510}, the proof is analogous to that of Corollary~\ref{cor090501}. Here we get the conclusions for all $z=E+\ii\eta\in \mathcal{S}_{\mathcal{I}}(\eta_{\mathrm{m}},1)$ at once, since we use the uniform estimate~\eqref{100501} instead of assumption~\eqref{09083000} for one fixed $z$. We omit the details.
\end{proof}
 Finally, we have the following counterpart to Proposition~\ref{lem.0910111}. 
 \begin{proposition} \label{pro.100501} Fix $z=E+\ii\eta\in \mathcal{S}_{\mathcal{I}}(\eta_{\mathrm{m}},1)$. Under the assumptions of Lemma \ref{lem093011}, we have 
 \begin{align}
\max_{i\not=j}|G_{ij}^{(i)}(z)|\prec \frac{1}{\sqrt{N\eta}} \label{100170}
 \end{align}
 and 
 \begin{align}
\max_{i\not=j}|T_{i,j}(z)|\prec \frac{1}{\sqrt{N\eta}}\,,\qquad \qquad \max_{i\not=j}|S_{i,j}(z)|\prec \frac{1}{\sqrt{N\eta}} \,.\label{100520}
 \end{align}
 \end{proposition}
 \begin{proof} The proof is similar to that of Proposition \ref{lem.0910111}.  Having established the concentration inequalities in~\eqref{093013}, it suffices to estimate $\mathbb{E}_{\mathbf{g}_i}\big[ G_{ij}^{(i)}\big]$ to prove~\eqref{100170}. We then start with
 \begin{align}\label{100160}
 (a_i-z)G_{ij}^{(i)}=-(\wt{B}^{(i)}G^{(i)})_{ij}+\delta_{ij}\,.
 \end{align}
 Choosing henceforth $i\not=j$, mimicking the reasoning from~\eqref{091073} to~\eqref{090401bis} and using ~\eqref{100155},
 we arrive at 
 \begin{align}
 (\wt{B}^{(i)}G^{(i)})_{ij}=-\mathbf{g}_i^* \widetilde{B}^{\langle i\rangle}G^{(i)}\mathbf{e}_j+\OSD\Big(\frac{1}{\sqrt{N}}\Big)=-S_{i,j}+\OSD\Big(\frac{1}{\sqrt{N}}\Big)\,. \label{100190}
 \end{align}
Then, instead of~\eqref{091101}, we obtain
\begin{align}
\mathbb{E}_{\mathbf{g}_i}[S_{i,j}]&= \mathbb{E}_{\mathbf{g}_i}\Big[\ntr  \big(\widetilde{B}^{\langle i\rangle}G^{(i)}\big)\big(S_{i,j}- b_i \,T_{i,j}\big)\Big]\nonumber\\
&\quad+\mathbb{E}_{\mathbf{g}_i}\Big[\ntr  \big(\widetilde{B}^{\langle i\rangle}G^{(i)}\widetilde{B}^{\langle i\rangle} \big)\big(G^{(i)}_{ij}+T_{i,j}\big)\Big]\nonumber\\
&\quad-\mathbb{E}_{\mathbf{g}_i}\Big[\ntr \big(\widetilde{B}^{\langle i\rangle}G^{(i)}\big)\;\big(\mathbf{e}_i^*\wt{B}^{\la i\ra}\mathbf{g}_i+\mathbf{g}_i^*\wt{B}^{\la i\ra}\mathbf{e}_i+\mathbf{g}_i^*\wt{B}^{\la i\ra}\mathbf{g}_i\big)\big(G^{(i)}_{ij}+T_{i,j}\big)\Big]\nonumber\\
&\quad-\frac{1}{N} \mathbb{E}_{\mathbf{g}_i}\Big[\Big(b_i^2G^{(i)}_{ii}+\mathbf{g}_i^*\big(\wt{B}^{\la i\ra}\big)^2G^{(i)}\mathbf{e}_i+\mathbf{e}_i^*(\wt{B}^{\la i\ra})^2G^{(i)}\mathbf{g}_i+\mathbf{g}_i^*\big(\wt{B}^{\la i\ra}\big)^2G^{(i)}\mathbf{g}_i\Big)\nonumber\\
&\qquad \qquad\qquad\times\big(G^{(i)}_{ij}+T_{i,j}\big)\Big]\,, \label{100730}
\end{align}
where we directly used the definitions in~\eqref{le off-diagonal T and S}. Then, similarly to~\eqref{091120}, using the concentration estimates in Lemma~\ref{lem093011} and in Lemma~\ref{cor1000101}, as well as the Gaussian concentration estimates in~\eqref{091080}, the bound~\eqref{100701} and Lemma~\ref{lem.0910115} for tracial quantities, we obtain
\begin{align}
\mathbb{E}_{\mathbf{g}_i}\big[S_{i,j}\big]-\ntr \big(\wt{B}G\wt{B}\big)\big(G_{ij}^{(i)}+\mathbb{E}_{\mathbf{g}_i}\big[T_{i,j}\big]\big)&=\ntr\big( \wt{B}G\big)\,\big(\mathbb{E}_{\mathbf{g}_i}\big[S_{i,j}-b_i \mathbb{E}_{\mathbf{g}_i}\big[T_{i,j}\big]\big]\big)
+\OSD\Big(\frac{1}{\sqrt{N\eta}} \Big)\,. \label{100280}
\end{align}
Analogously, we also have
\begin{align}
\mathbb{E}_{\mathbf{g}_i}\big[T_{i,j}\big]-\ntr \big(\wt{B}G\big)\,\big(G_{ij}^{(i)}+\mathbb{E}_{\mathbf{g}_i}\big[T_{i,j}\big]\big)&=\ntr \big(G\big) \,\big(\mathbb{E}_{\mathbf{g}_i}\big[S_{i,j}\big]-b_i \mathbb{E}_{\mathbf{g}_i}\big[T_{i,j}\big]\big)
+\OSD\Big(\frac{1}{\sqrt{N\eta}} \Big)\,. \label{100281}
\end{align}
Solving $\mathbb{E}_{\mathbf{g}_i}\big[S_{i,j}\big]$ from~\eqref{100280} and~\eqref{100281}, we have
\begin{align*}
\mathbb{E}_{\mathbf{g}_i}\big[S_{i,j}\big]&=-\frac{\ntr  \big(\widetilde{B}G\big)}{\ntr G} G_{ij}^{(i)}+\bigg[\frac{\ntr\big( \widetilde{B}G\big)-\big(\ntr\big(\widetilde{B}G\big)\big)^2}{\ntr G}+\ntr \big( \widetilde{B}G\widetilde{B}\big)\bigg]\big(G_{ij}^{(i)}+\mathbb{E}_{\mathbf{g}_i}\big[T_{i,j}\big]\big)\nonumber\\ &\qquad\qquad+\OSD\Big(\frac{1}{\sqrt{N\eta}} \Big)\,.
\end{align*}
Using~\eqref{091221}, the assumption $|G_{ij}^{(i)}|\prec 1$ and the bound $|T_{i,j}|\prec 1$ of~\eqref{100155}, we have
\begin{align}
\mathbb{E}_{\mathbf{g}_i}\big[S_{i,j}\big]=-\frac{\ntr  \big(\widetilde{B}G\big)}{\ntr G} G_{ij}^{(i)}+\OSD\Big(\frac{1}{\sqrt{N\eta}} \Big)\,, \label{100702}
\end{align}
which together with~\eqref{100160},~\eqref{100190}, the concentration estimate~\eqref{100152} implies that
\begin{align}
(a_i-\omega_B^c) G_{ij}^{(i)}=\OSD\Big(\frac{1}{\sqrt{N\eta}}\Big)\,. \label{100521}
\end{align}
This proves the estimate in~\eqref{100170}. 

Next, we bound $S_{i,j}$. Starting from~\eqref{100702} we directly get the second estimates in~\eqref{100520} from the Green function bound~\eqref{100170} and the concentration estimate~\eqref{100152}. 

It remains to estimate $T_{i,j}$. Plugging the bound on $G_{ij}$ in~\eqref{100170} and the bound on $S_{i,j}$ in~\eqref{100520} into the equation~\eqref{100281}, we obtain
\begin{align}
\big|\big(1-\ntr \big(\wt{B}G\big)+b_i\ntr G\big) \mathbb{E}_{\mathbf{g}_i}\big[T_{i,j}\big]\big|\prec\frac{1}{\sqrt{N\eta}}\,. \label{100526}
\end{align}
Invoking the estimate~\eqref{091231} we get $\big|\mathbb{E}_{\mathbf{g}_i}\big[T_{i,j}\big]\big|\prec\frac{1}{\sqrt{N\eta}}$. Then the first estimate in~\eqref{100520} follows from the concentration estimate for $T_{i,j}$ in~\eqref{100152}. This completes the proof.\qedhere
 \end{proof}

 Having established Lemma \ref{lem093011} and Proposition \ref{pro.100501}, we next prove~\eqref{100710} of Theorem~\ref{thm.091201} via a continuity argument similar to the proof of~\eqref{0912100}.

 \begin{proof}[Proof of~\eqref{100710} of Theorem~\ref{thm.091201}]  Fixing any 
 $z\in \mathcal{S}_{\mathcal{I}}(\eta_{\mathrm{m}},1)$ and 
  using Proposition \ref{pro.100501}, under the assumption
 \begin{align}
\max_{i\neq j} |G_{ij}^{(i)}(z)|\prec 1\,, \label{restate the bound for off-diagonal}
 \end{align} 
 we have 
  \begin{align}
\max_{i\not=j}|G_{ij}^{(i)}(z)|\prec \frac{1}{\sqrt{N\eta}}\,,\qquad 
\max_{i\not=j}|T_{i,j}(z)|\prec \frac{1}{\sqrt{N\eta}}\,,\qquad  \max_{i\not=j}|S_{i,j}(z)|\prec \frac{1}{\sqrt{N\eta}} \,. \label{collection of the bounds in Prop 8.3}
 \end{align}
 Then, by~\eqref{collection of the bounds in Prop 8.3} and~\eqref{100501}, we can use Lemma ~\ref{lem.091060}
 to get 
  \begin{align}
 |G_{ij}(z)|\prec \frac{1}{\sqrt{N\eta}}. \label{restate of the estimate of off-diagonal entry}
 \end{align} 
 Hence, in principle, it suffices  to conduct a  continuity  argument from $\eta=1$ to $\eta=\eta_{\mathrm{m}}$ 
 (similar to the proof  of~\eqref{0912100} of Theorem~\ref{thm.091201}) 
 to show that the bound~\eqref{restate the bound for off-diagonal} holds uniformly for $z\in \mathcal{S}_{\mathcal{I}}(\eta_{\mathrm{m}},1)$. However, in order to show that~\eqref{restate of the estimate of off-diagonal entry} also holds uniformly for $z\in \mathcal{S}_{\mathcal{I}}(\eta_{\mathrm{m}},1)$ quantitatively, we  monitor $G_{ij}$ in the  
 continuity argument as well.  To this end, we introduce the $z$-dependent random variable
 \begin{align*}
  \Lambda_{\mathrm{o}}\equiv\Lambda_{\mathrm{o}}(z)\deq \max_{i\neq j}| G_{ij}^{(i)}(z)|+\max_{i\neq j}| G_{ij}(z)|\,,
 \end{align*}
 and, for any $\delta\in [0,1]$ and $z\in \mathcal{S}_{\mathcal{I}}(\eta_{\mathrm{m}},1)$, we  define the event
 \begin{align*}
 \Theta_{\mathrm{o}}(z,\delta)\deq\{\Lambda_{\mathrm{o}}(z)\leq \delta\}\,;
 \end{align*}
 \cf~\eqref{def of Lambda} and~\eqref{def of Xi}. The subscript $\mathrm{o}$ refers to ``off-diagonal".
 
 We will mimic the proof of ~\eqref{0912100}. 
 Analogously, using Lemma~\ref{lem.091060} and Proposition~\ref{pro.100501}, 
 one shows that there exists an event $\Omega_{\mathrm{o}}(z)\equiv \Omega_{\mathrm{o}}(z,D,\varepsilon)$ such that the conclusions in Lemma~\ref{lem.091282} still hold  when we replace $\Theta_{\mathrm{d}}(z,\delta)$ by $\Theta_{\mathrm{o}}(z,\delta)$, $\Omega_{\mathrm{d}}(z)$
 by $\Omega_{\mathrm{o}}(z)$ and $N^{-\frac{\gamma}{4}}$ by $1$. 
  We also set $\delta=1$ in this proof. 
 This is a quantitative description of the derivation of the first bound in~\eqref{collection of the bounds in Prop 8.3} and (\ref{restate of the estimate of off-diagonal entry}) from~\eqref{restate the bound for off-diagonal}. The main difference is that here $\Omega_{\mathrm{o}}(z)$ is 
 the  event defined  as the intersection of the ``typical" events in all
 the concentration estimates in
 Sections~\ref{s.5}--\ref{section partial concentration}, in the proofs of Lemma~\ref{lem093011}
 and Proposition~\ref{pro.100501},
and the event on which the following bounds hold
 \begin{align}
& \Big| G_{ii}^{(i)}(z)-\big(a_i-\omega_B(z)\big)^{-1}\Big| \leq  \frac{N^\varepsilon}{\sqrt{N\eta}}\,,\nonumber\\
& \Big| G_{ii}(z)-\big(a_i-\omega_B(z)\big)^{-1}\Big| \leq \frac{N^\varepsilon}{\sqrt{N\eta}}\,, \qquad\quad \max_{i\neq j}\big|G^{[i]}_{jj}(z)\big|\leq  N^{\varepsilon}\,. \label{bound of G and minor of G}
 \end{align} 
Note that, by (\ref{100501}) and (\ref{100510}), we know that  (\ref{bound of G and minor of G}) holds with high probability uniformly on $ \mathcal{S}_{\mathcal{I}}(\eta_{\mathrm{m}},1)$.

 With the analogue of Lemma \ref{lem.091282} for $\Theta_{\mathrm{o}}(z, \delta=1)$ and $\Omega_{\mathrm{o}}(z)$, we conduct a continuity argument  similar to the one in the proof of~\eqref{0912100}. Again, by Lipschitz continuity of the Green function it suffices to show estimate~\eqref{100710} on the lattice $ \widehat{\mathcal{S}}_{\mathcal{I}}(\eta_{\mathrm{m}},1)$ defined in~\eqref{100715}. We fix $E\in \mathcal{I}\cap N^{-5}\mathbb{Z}$, write $z=E+\ii\eta$ and decrease $\eta$ from  $\eta=1$ down to $N^{-1+\gamma}$ in steps of size~$N^{-5}$. The initial estimate for $\eta=1$, \ie $\Lambda_{\mathrm{o}}(E+\mathrm{i})\leq 1$ follows directly from the trivial fact $\|G^{(i)}(z)\|, \|G(z)\|\leq 1/\eta$. Then one can show step by step that for any $\eta\in [\eta_{\mathrm{m}},1]$, say, 
\begin{align}
\Theta_{\mathrm{o}}\big(\eta,  1 \big)\cap \Omega_{\mathrm{o}} (\eta-N^{-5})\subset \Theta_{\mathrm{o}}\Big(\eta-N^{-5},  1\Big)\,, \label{key inclusion for off-diagonal}
\end{align}
which is the analogue of~\eqref{key inclusion}. The remaining proof is nearly the same as the counterpart in the proof of of~\eqref{0912100}. We thus omit the details.
\end{proof}

\begin{appendix} 
\section{Orthogonal case}\label{s.a.a}

In this appendix, we show that Theorem~\ref{thm.091201} and Theorem~\ref{thm.091501} also hold in the orthogonal setup where $U$ is Haar distributed on the orthogonal group $O(N)$. From the proof of Theorem \ref{thm.091501}, we see that it is implied by Theorem~\ref{thm.091201}. Hence, it suffices to discuss the latter. We outline the necessary changes in the discussion of Sections~\ref{s.5}-\ref{s.8} to adapt our proof to the orthogonal case. We mainly show the modification for the proof of (\ref{0912100}) in detail, and (\ref{100710}) will be discussed briefly at the end.

First, we modify some notation.  We start with the decomposition for the Haar measure on the orthogonal group analogous to~\eqref{le first decomposition}. For all $i\in\llbracket 1,N\rrbracket$, according to ~\cite{Mezzadri}, there exist a random vector $\mathbf{v}_i=(v_{i1},\ldots, v_{iN})$, uniformly distributed on the real unit $(N-1)$-sphere $\mathcal{S}_{\R}^{N-1}\deq \{\mathbf{x}\in\R^N\,:\, \mathbf{x}^*\mathbf{x}=1   \}$, and a Haar distributed orthogonal matrix $U^i\in O(N-1)$, which is independent of $\mathbf{v}_i$, such that one has the decomposition
\begin{align*}
U=-\sgn(v_{ii})(I-\mathbf{r}_i\mathbf{r}_i^*)U^{\la i\ra}\deq -\sgn(v_{ii})R_i\,U^{\la i\ra}\,,
\end{align*}
where 
\begin{align}
\mathbf{r}_i\deq\sqrt{2}\frac{\mathbf{e}_i+\sgn(v_{ii})\mathbf{v}_i}{\|\mathbf{e}_i+\sgn(v_{ii})\mathbf{v}_i\|_2}\,,     \qquad\qquad  R_i\deq I-\mathbf{r}_i\mathbf{r}_i^*\,,
\end{align}
 and $U^{\la i\ra}$  is the orthogonal matrix with $\mathbf{e}_i$ as its $i$th column and $U^i$ as its $(i,i)$-matrix minor. Moreover, there is a real Gaussian vector $\mathbf{g}_i\sim\mathcal{N}_\R(0, N^{-1}I)$ such that
 \begin{align*}
 \mathbf{v}_i=\frac{\widetilde{\mathbf{g}}_i}{\|\widetilde{\mathbf{g}}_i\|_2}\,,
 \end{align*}
 Similarly to (\ref{def of g}), we define
 \begin{align*}
 g_{ik}\deq\mathrm{sgn}(v_{ii})\,\widetilde{g}_{ik}\,,\qquad\qquad k\neq i\,,
 \end{align*}
 and introduce an $N(0,N^{-1})$ variable $g_{ii}$, which is independent of the orthogonal matrix~$U$ and of $\widetilde{\mathbf{g}}_i$. 
Let $\mathbf{g}_i\deq(g_{i1},\ldots,g_{iN})$ and note that $\mathbf{g}_i\sim \mathcal{N}_{\mathbb{R}}(0,N^{-1}I)$.
Then we set $\mathbf{w}_i\deq\mathbf{e}_i+\mathbf{g}_i$ and $W_i\deq I-\mathbf{w}_i\mathbf{w}_i^*$ as before.
With these modifications, we follow the proofs in Sections~\ref{s.5}-\ref{s.7} verbatim. The only difference is the derivation of~\eqref{091105}. Instead of~\eqref{091350}, we use the following integration by parts formula for real Gaussian random variables
\begin{align}
\int_{\mathbb{R}} g f(g)\,\e{-\frac{g^2}{2\sigma^2}} {\rm d} g=\sigma^2\int_{\mathbb{R}} f'(g)\,\e{-\frac{g^2}{2\sigma^2}} {\rm d}  g \,,\label{091380}
\end{align}
for differentiable functions $f\,:\, \R\to\R$.
Correspondingly, instead of~\eqref{091755}, we have 
\begin{align*}
\frac{\partial W_i}{\partial g_{ik}}=-\mathbf{e}_i\mathbf{e}_k^*-\mathbf{e}_k\mathbf{e}_i^*-\mathbf{e}_k\mathbf{g}_i^*-\mathbf{g}_i\mathbf{e}_k^*\,.
\end{align*}
Hence, we get 
\begin{align*}
&\frac{\partial \big(\widetilde{B}^{\langle i\rangle}G^{(i)}\big)_{kj}}{\partial g_{ik}}\nonumber\\ &\qquad=\mathbf{e}_k^*\widetilde{B}^{\langle i\rangle}G^{(i)}\big(\mathbf{e}_i\mathbf{e}_k^*+\mathbf{e}_k\mathbf{e}_i^*+\mathbf{e}_k\mathbf{g}_i^*+\mathbf{g}_i\mathbf{e}_k^*\big)\widetilde{B}^{\langle i\rangle}\big(I-\mathbf{e}_i\mathbf{e}_i^* -\mathbf{e}_i\mathbf{g}_i^*-\mathbf{g}_i\mathbf{e}_i^*-\mathbf{g}_i\mathbf{g}_i^*\big) G^{(i)}\mathbf{e}_j\nonumber\\
&\qquad+\mathbf{e}_k^*\widetilde{B}^{\langle i\rangle}G^{(i)}\big(I-\mathbf{e}_i\mathbf{e}_i^* -\mathbf{e}_i\mathbf{g}_i^*-\mathbf{g}_i\mathbf{e}_i^*-\mathbf{g}_i\mathbf{g}_i^*\big)\widetilde{B}^{\langle i\rangle}\big(\mathbf{e}_i\mathbf{e}_k^*+\mathbf{e}_k\mathbf{e}_i^*+\mathbf{e}_k\mathbf{g}_i^*+\mathbf{g}_i\mathbf{e}_k^*\big) G^{(i)}\mathbf{e}_j
\end{align*}
instead of~\eqref{091758}. Substitution into the identity
 \begin{align*}
\mathbb{E}_{\mathbf{g}_i}[\mathbf{g}_i^* \widetilde{B}^{\langle i\rangle}G^{(i)}\mathbf{e}_j]=\sum_{k=1}^N \mathbb{E}_{\mathbf{g}_i} \big[g_{ik} (\widetilde{B}^{\langle i\rangle}G^{(i)})_{kj}\big]= \frac{1}{N}\sum_{k=1}^N \mathbb{E}_{\mathbf{g}_i} \bigg[\frac{\partial (\widetilde{B}^{\langle i\rangle}G^{(i)})_{kj}}{\partial g_{ik}}\bigg]\,,
 \end{align*} 
yields
\begin{align}
\mathbb{E}_{\mathbf{g}_i}[\mathbf{g}_i^* \widetilde{B}^{\langle i\rangle}G^{(i)}\mathbf{e}_j]&= (\text{r.h.s. of~\eqref{100730}})+\frac{1}{N}\mathbb{E}_{\mathbf{g}_i}\Big[\big(G^{(i)}(\wt{B}^{\la i\ra})^2G^{(i)}\big)_{ji}\Big]\nonumber\\
&\quad+\frac{1}{N}\mathbb{E}_{\mathbf{g}_i}\Big[\mathbf{e}_j^*G^{(i)}(\wt{B}^{\la i\ra})^2G^{(i)}\mathbf{g}_i\Big] +\frac{1}{N}\mathbb{E}_{\mathbf{g}_i}\Big[\big( G^{(i)}\wt{B}^{\la i\ra} G^{(i)}\wt{B}^{\la i\ra}\big)_{ji}\Big]\nonumber\\
&\quad+\frac{1}{N}\mathbb{E}_{\mathbf{g}_i}\Big[\mathbf{e}_j^*G^{(i)}\wt{B}^{\la i\ra} G^{(i)}\wt{B}^{\la i\ra}\mathbf{g}_i\Big]-\frac{1}{N}\mathbb{E}_{\mathbf{g}_i}\Big[ \widehat b_i\big(G^{(i)}\wt{B}^{\la i\ra}G^{(i)}\big)_{ji}\Big]\nonumber\\
&\quad-\frac{1}{N}\mathbb{E}_{\mathbf{g}_i}\Big[\widehat{b}_j\,\mathbf{e}_j^*G^{(i)}\wt{B}^{\la i\ra}G^{(i)}\mathbf{g}_i\Big] \nonumber\\
&\quad-\frac{1}{N} \mathbb{E}_{\mathbf{g}_i}\Big[(G^{(i)}_{ij}+\mathbf{g}_i^*G^{(i)}\mathbf{e}_j)
\Big(\mathbf{e}_i^*(\wt{B}^{\la i\ra})^2G^{(i)}\mathbf{g}_i+\mathbf{g}_i^*(\wt{B}^{\la i\ra})^2G^{(i)}\mathbf{g}_i\Big)\Big]\nonumber\\
&\quad-\frac{1}{N} \mathbb{E}_{\mathbf{g}_i}\Big[(G^{(i)}_{ij}+\mathbf{g}_i^*G^{(i)}\mathbf{e}_j)\Big(b_i^2G^{(i)}_{ii}+\mathbf{g}_i^*(\wt{B}^{\la i\ra})^2G^{(i)}\mathbf{e}_i \Big)\Big]
\,,\label{091361}
\end{align}
where we introduced $\widehat b_i\deq \mathbf{w}_i\widetilde B^{\la i \ra} \mathbf{w}_i$. Note that the last two terms were discussed in the unitary setup, and they were shown to be negligible. Therefore, to get~\eqref{100280} also in the orthogonal case,  we rely on the following lemma to discard the supplementary small terms in~\eqref{091361}. At first, let us discuss the case of $i=j$, which suffices for the proof of~\eqref{0912100}.
\begin{lemma} \label{rem.091090} 
Under the assumption of Proposition~\ref{lem.0910111}, we have the following bounds
\begin{align}\def\arraystretch{1.5}\begin{array}{ll}
\big|\big(G^{(i)}(z)\wt{B}^{\la i\ra}G^{(i)}(z)\big)_{ii}\big|\prec \frac{1}{\eta}\,, \qquad\qquad & \big|\big(G^{(i)}(z)(\wt{B}^{\la i\ra})^2G^{(i)}(z)\big)_{ii}\big|\prec \frac{1}{\eta}\,,\\
 \big|\mathbf{e}_i^*G^{(i)}(z)\wt{B}^{\la i\ra}G^{(i)}(z)\mathbf{g}_i\big|\prec \frac{1}{\eta}\,, &\big|\mathbf{e}_i^*G^{(i)}(z)(\wt{B}^{\la i\ra})^2G^{(i)}(z)\mathbf{g}_i\big|\prec \frac{1}{\eta}\,,\\
\big|\big( G^{(i)}(z)\wt{B}^{\la i\ra} G^{(i)}(z)\wt{B}^{\la i\ra}\big)_{ii}\big|\prec \frac{1}{\eta}\,,&\big|\mathbf{e}_i^*G^{(i)}(z)\wt{B}^{\la i\ra} G^{(i)}(z)\wt{B}^{\la i\ra}\mathbf{g}_i\big|\prec \frac{1}{\eta}\,,\end{array} \label{091099}
\end{align}
for all $i\in \llbracket 1,\ldots, N\rrbracket$.
\end{lemma}
\begin{proof} We drop $z$ from the notation. For the first two terms, we have
\begin{align}
\big|\big(G^{(i)}(\wt{B}^{\la i\ra})^kG^{(i)}\big)_{ii}\big|\leq \|\wt{B}^{\la i\ra}\|^k\|G^{(i)}\mathbf{e}_i\|_2^2\lesssim \big((G^{(i)})^*G^{(i)}\big)_{ii}=\frac{\Im G^{(i)}_{ii}}{\eta}\prec \frac{1}{\eta}\,, \label{100731}
\end{align}
for $k=1,2$, where in the last step we used assumption~\eqref{091071}.
For the third and fourth terms, we have, for $k=1,2$,
\begin{align*}
\big|\mathbf{e}_i^*G^{(i)}(\wt{B}^{\la i\ra})^kG^{(i)}\mathbf{g}_i\big| &\leq \|\wt{B}^{\la i\ra}\|^k \|G^{(i)}\mathbf{e}_i\|_2\|G^{(i)}\mathbf{g}_i\|_2\lesssim \big(\big((G^{(i)})^*G^{(i)}\big)_{ii}\big)^{\frac12}\big(\mathbf{g}_i^*(G^{(i)})^*G^{(i)}\mathbf{g}_i\big)^{\frac12}\nonumber\\
&=\frac{1}{\eta}\big(\Im G_{ii}^{(i)}\big)^{\frac12}\big(\Im \mathbf{g}^*_i G^{(i)}\mathbf{g}_i\big)^{\frac12}\prec\frac{1}{\eta}\,,
\end{align*}
where in the last step we used assumption~\eqref{091071} and estimate~\eqref{091090}. For the fifth term we note that $( G^{(i)}\wt{B}^{\la i\ra} G^{(i)}\wt{B}^{\la i\ra})_{ii}=b_i\big( G^{(i)}\wt{B}^{\la i\ra} G^{(i)})_{ii}$ and the bound follows from~\eqref{100731}. For the last term, we have
  \begin{align}
\big|\mathbf{e}_i^*G^{(i)}\wt{B}^{\la i\ra} G^{(i)}\wt{B}^{\la i\ra}\mathbf{g}_i\big| &\leq \|\wt{B}^{\la i\ra}\|\|\mathbf{e}_i^*G^{(i)}\|_2\|G^{(i)}\wt{B}^{\la i\ra}\mathbf{g}_i\|_2\prec\frac{1}{\eta}\,, \label{100735}
\end{align}
where we used assumption~\eqref{091071} and estimate~\eqref{091090}.
This completes the proof.
\end{proof}
All the other arguments in Sections \ref{s.5}-\ref{s.7} work for the orthogonal case as well without modifications. This proves~\eqref{0912100} of Theorem~\ref{thm.091201} for the Haar orthogonal case.  
 
 For (\ref{100710}), analogously to~\eqref{091099}, we shall estimate the second to the seventh terms on the right side of~\eqref{091361}, under the assumption of Proposition~\ref{pro.100501}. To bound these terms, we can pursue the discussion from~\eqref{100731}-\eqref{100735} with $\Im G_{ii}^{(i)}$ replaced by $\Im G_{jj}^{(i)}$ in the bounding procedure. Hence, it suffices to show for all $j\neq i$ that $|G_{jj}^{(i)}|\prec 1$. To see this, we use the fact $|G_{jj}|\prec 1$  from Theorem~\ref{thm.091201} and $|G_{jj}-G_{jj}^{(i)}|\prec \frac{1}{\sqrt{N\eta}}$ from Lemma~\ref{lem.091060} with $k=j$. Note that the assumptions of Lemma \ref{lem.091060} are guaranteed by Theorem~\ref{thm.091201}, \eqref{100155} and assumption~\eqref{093022}. Hence,~\eqref{100710} also holds in the orthogonal case.

\section{Two point mass case} \label{s.a.b}
In this section, we present our result when both, $\mu_\alpha$ and $\mu_\beta$, are convex combinations of two point masses. Without loss of generality (up to shifting and scaling), we may assume that $\mu_\alpha$ and $\mu_\beta$ are of the following form,
\begin{align}
&\mu_\alpha=\xi\delta_1+(1-\xi)\delta_0\,,\qquad\quad \mu_\beta=\zeta\delta_{\theta}+(1-\zeta)\delta_0\,,\label{081210}
\end{align}
with real parameters $\xi,\zeta$ and $\theta$ satisfying
\begin{align*}
 \theta\neq 0\,,\qquad\quad \xi,\zeta\in \Big(0,\frac{1}{2}\Big]\,, \qquad \quad\xi\leq \zeta\,,\qquad \quad (\theta, \xi,\zeta)\neq \Big(-1,\frac12,\frac12\Big)\,. 
\end{align*}
Here we excluded the case  $(\theta, \xi,\zeta)=(-1,\frac12,\frac12)$ since it is equivalent to $(\theta, \xi,\zeta)=(1,\frac12,\frac12)$ under a shifting, where the latter is a special case of $\mu_\alpha=\mu_\beta$.
In Section~7 of~\cite{BES15}, we explained why the setting of~\eqref{081210} is special, and we thus excluded it from Theorem~\ref{thm.091201}.

 Following~\cite{Kargin2012a}, we argued in Lemma~7.1 of~\cite{BES15} that in the setting of~\eqref{081210} we have 
\begin{align}
\mathcal{B}_{\mu_\alpha\boxplus\mu_\beta}=(\ell_{1}, \ell_{2})\cup (\ell_{3}, \ell_{4})\,, \label{def of B case 1}
\end{align}
in case  $\mu_\alpha\neq \mu_\beta$, while we have
\begin{align}
\mathcal{B}_{\mu_\alpha\boxplus\mu_\alpha}=(\ell_{1}, \ell_{4}) \label{def of B case 2}\,,
\end{align}
in case $\mu_\alpha=\mu_\beta$, where
\begin{align*}
\ell_{1}&\deq\min\Big\{\frac{1}{2}\Big(1+\theta-\sqrt{(1-\theta)^2+4\theta r_+}\Big), \frac{1}{2}\Big(1+\theta-\sqrt{(1-\theta)^2+4\theta r_-}\Big)\Big\}\,,\\
\ell_{2}&\deq\max\Big\{\frac{1}{2}\Big(1+\theta-\sqrt{(1-\theta)^2+4\theta r_+}\Big), \frac{1}{2}\Big(1+\theta-\sqrt{(1-\theta)^2+4\theta r_-}\Big)\Big\}\,,\\
\ell_{3}&\deq\min\Big\{\frac{1}{2}\Big(1+\theta+\sqrt{(1-\theta)^2+4\theta r_+}\Big), \frac{1}{2}\Big(1+\theta+\sqrt{(1-\theta)^2+4\theta r_-}\Big)\Big\}\,,\\
\ell_{4}&\deq\max\Big\{\frac{1}{2}\Big(1+\theta+\sqrt{(1-\theta)^2+4\theta r_+}\Big), \frac{1}{2}\Big(1+\theta+\sqrt{(1-\theta)^2+4\theta r_-}\Big)\Big\}\,,
\end{align*}
with $r_{\pm}\deq\xi+\zeta-2\xi\zeta\pm\sqrt{4\xi\zeta(1-\xi)(1-\zeta)}$.

In Remark 7.2 of~\cite{BES15} we argued that, in the case $\mu_\alpha=\mu_\beta$, the point $E=1\in \mathcal{B}_{\mu_\alpha\boxplus\mu_\alpha}$ is special in the sense that $m_{\mu_{\alpha}\boxplus\mu_{\alpha}}(1+\mathrm{i}0)$ is unstable under small perturbations. We thus expect a modified local law in neighborhoods of this special point. To proceed we need some more notation. Recall the domains $\mathcal{S}_{\mathcal{I}}(a,b)$ in~\eqref{le domain S}. For given (small) $\varsigma,\gamma>0$, we set
\begin{align*}
\mathcal{S}_\mathcal{I}^{\varsigma}(a,b)\deq\bigg\{z\in \mathcal{S}_{\mathcal{I}}(a,b): \varsigma|z-1|\geq\max\Big\{\sqrt{{\rm{d_L}}(\mu_A,\mu_\alpha)},\sqrt{{\rm{d_L}}(\mu_B,\mu_\beta)}\Big\} \bigg\}
\end{align*}
and
\begin{align}
\widetilde{\mathcal{S}}_\mathcal{I}^{\varsigma}(a,b)\deq\mathcal{S}_\mathcal{I}^{\varsigma}(a,b) \cap\bigg\{z\in\mathbb{C}: |z-1|\geq \frac{N^{\gamma}}{(N\eta)^{\frac14}}\bigg\}\,.\label{081501}
\end{align}
The following proposition presents the local law under the setting~\eqref{081210}.
\begin{proposition} \label{pro.a.1}
Let $\mu_\alpha,\mu_\beta$ be as in~\eqref{081210}, with fixed $\xi$, $\zeta$ and $\theta$. Assume that the empirical eigenvalue distributions $\mu_A$, $\mu_B$ of the sequences of matrices $A$, $B$ satisfy~\eqref{le assumptions convergence empirical measures}. Fix any compact nonempty interval $\mathcal{I}\subset\mathcal{B}_{\mu_\alpha\boxplus\mu_\beta}$.  With the notations and assumptions of Theorem~\ref{thm.091201}, we have the following conclusions:
\begin{enumerate}
\item[$(i)$] If $\mu_{\alpha}\neq \mu_{\beta}$, then, for any fixed $\gamma>0$,
\begin{align*}
&\max_{1\le i\le N}\big|G_{ii}(z)-(a_i-\omega_B(z))^{-1}\big|\prec  \frac{1}{\sqrt{N\eta}}\,, \nonumber\\
&|\omega_A^c(z)-\omega_A(z)|\prec\frac{1}{\sqrt{N\eta}}\,,\qquad|\omega_B^c(z)-\omega_B(z)|\prec\frac{1}{\sqrt{N\eta}}\,,
\end{align*}
hold uniformly for all $z\in \mathcal{S}_{\mathcal{I}}(\eta_{\mathrm{m}},1)$. 
\item[$(ii)$] If $\mu_{\alpha}=\mu_{\beta}$, then, for sufficiently small $\varsigma>0$ and any fixed $\gamma>0$, 
\begin{align}\label{092101}
&\max_{1\le i\le N}\big|G_{ii}(z)- (a_i-\omega_B(z))^{-1}\big|\prec  \frac{1}{|z-1|}\frac{1}{\sqrt{N\eta}}\,\nonumber\\
&|\omega_A^c(z)-\omega_A(z)|\prec \frac{1}{|z-1|}\frac{1}{\sqrt{N\eta}}\,,\qquad |\omega_B^c(z)-\omega_B(z)|\prec \frac{1}{|z-1|}\frac{1}{\sqrt{N\eta}}\,,
\end{align}
hold uniformly for all $z\in \widetilde{\mathcal{S}}^\varsigma_{\mathcal{I}}(\eta_{\mathrm{m}},1)$.
\end{enumerate}
\end{proposition}

\begin{proof} Recall the notation $\Gamma_{\mu_1,\mu_2}(\omega_1,\omega_2)$ in~\eqref{le what stable means}. In~\cite{BES15} (see the proof of Proposition 7.4 therein), we proved that under the setting~\eqref{081210} and assumption~\eqref{le assumptions convergence empirical measures}, one has the following results on the stability of the system $\Phi_{\mu_A,\mu_B}(\omega_A,\omega_B,z)=0$: There exists a positive constant~$S$ such that the following two estimates hold.
\begin{enumerate}
\item[(i):] If $\mu_{\alpha}\neq \mu_{\beta}$, we have 
\begin{align*}
\Gamma_{\mu_A,\mu_B}(\omega_A,\omega_B)\leq S\,,\qquad\qquad z\in \mathcal{S}_{\mathcal{I}}(0,1)\,,
\end{align*}
and~\eqref{091460},~\eqref{le upper bound on omega AB} and~\eqref{le lower bound on omega AB} hold on $ \mathcal{S}_{\mathcal{I}}(0,1)$.
\item[(ii:)] If $\mu_{\alpha}= \mu_{\beta}$, we have
\begin{align}
\Gamma_{\mu_A,\mu_B}(\omega_A,\omega_B)\leq \frac{S}{|z-1|}\,,\qquad\qquad z\in  \mathcal{S}^\varsigma_{\mathcal{I}}(0,1)\,, \label{092157}
\end{align}
and~\eqref{le upper bound on omega AB} and~\eqref{le lower bound on omega AB} hold on $\mathcal{S}^\varsigma_{\mathcal{I}}(0,1)$, while ~\eqref{091460} holds on $\mathcal{S}^\varsigma_{\mathcal{I}}(0,1)$ with $S$ replaced by $\frac{S}{|z-1|}$.
\end{enumerate}
Note that, the proofs of Lemma~\ref{lem090501}, Lemma~\ref{lem.0910111} and Lemma~\ref{lem.091282} still work since we have the bounds~\eqref{091460},~\eqref{le upper bound on omega AB}, and~\eqref{le lower bound on omega AB} as well. Although the bound in ~\eqref{091460} should be replaced by $\frac{2S}{|z-1|}$ in the case $\mu_{\alpha}= \mu_{\beta}$, it is harmless for our proof. Hence, analogously to the proof of Theorem~\ref{thm.091201}, one can use Lemma~\ref{le lemma perturbation of system},~Lemma~\ref{lem.091282} and estimates~\eqref{091460},~\eqref{le upper bound on omega AB} and~\eqref{le lower bound on omega AB}, to complete the proof of Proposition~\ref{pro.a.1}. Especially, the proof in the case $\mu_\alpha\neq \mu_\beta$ exactly agrees with the proof of Theorem~\ref{thm.091201}.

For the case $\mu_{\alpha}= \mu_{\beta}$, we need to replace $S$ by $\frac{S}{|z-1|}$ in Lemma \ref{le lemma perturbation of system} due to~\eqref{092157}.  In the sequel, we simply illustrate the continuity argument in this case.  
Let $z,z'\in \widetilde{\mathcal{S}}_\mathcal{I}^{\varsigma}(a,b)$, where $z=E+\mathrm{i}\eta$ and $z'=E+\mathrm{i}\eta'$, with $\eta'=\eta+N^{-5}$. In addition, we set $z_0=z$, $\omega_1=\omega_A$, $\omega_2=\omega_B$, $\wt{\omega}_1=\omega_A^c$ and $\wt{\omega}_2=\omega_B^c$ in Lemma~\ref{le lemma perturbation of system}. 
Suppose now that~\eqref{092101} holds for $z'$. Using the Lipschitz continuity of the Green function (\ie $\|G(z)-G(z')\|\leq N^2|z-z'|$) and of the subordination functions~$\omega_A(z)$ and~$\omega_B(z)$ (\cf~\eqref{091460} with $S$ replaced by $\frac{S}{|z-1|}$), we can choose $\delta$ in~\eqref{le apriori closeness} to be
\begin{align}
\delta=\frac{N^{\gamma}}{|z-1|}\frac{1}{\sqrt{N\eta}}+O(N^{-3})\,,\label{092150}
\end{align}
In light of the condition $k^2>\delta K\frac{S}{|z-1|}$ (\cf sentence above~\eqref{le stability sum ims}, with $S$ replaced by $\frac{S}{|z-1|}$), one needs to guarantee that $\delta S\leq |z-1|\varepsilon$, for sufficiently small constant $\varepsilon>0$, which is a direct consequence of the assumption that $z\in \widetilde{\mathcal{S}}_\mathcal{I}^{\varsigma}(a,b)$ and~\eqref{092150}.
 Note that $\|\wt{r}(z)\|_2\prec\frac{1}{\sqrt{N\eta}}$ remains valid since estimate~\eqref{091701} does not depend on the stability of the system $\Phi_{\mu_A,\mu_B}(\omega_A,\omega_B,z)=0$, as long as~\eqref{le upper bound on omega AB},~\eqref{le lower bound on omega AB} and~\eqref{091460} hold. The remaining parts of the proof are analogous to those of Theorem~\ref{thm.091201} and we thus omit the details.
\end{proof}

\end{appendix}

\end{document}